\documentclass[10pt,oneside,reqno]{amsart}
\usepackage[english]{babel}
\usepackage{indentfirst}
\usepackage[T1]{fontenc}
\usepackage[utf8]{inputenc}
\usepackage{paralist}
\usepackage{amssymb}
\usepackage{amsthm}
\usepackage{amsmath}
\usepackage{amscd}
\usepackage{graphicx}
\usepackage[colorlinks=true]{hyperref}
\usepackage[margin=1in]{geometry}
\usepackage{cite}
\usepackage{lineno} 

\renewcommand{\emptyset}{\varnothing}

\theoremstyle{plain} 
\newtheorem{theorem}{Theorem}[section] 
\newtheorem{corollary}[theorem]{Corollary} 
\newtheorem{lemma}[theorem]{Lemma}
\newtheorem{prop}[theorem]{Proposition}
\newtheorem{defn}[theorem]{Definition}

\theoremstyle{definition}
\newtheorem{example}{Example}[section]

\theoremstyle{remark}
\newtheorem{remark}[theorem]{Remark}

\numberwithin{equation}{section} 

\newcommand{\R}{\ensuremath{\mathbb{R}}}
\newcommand{\Rn}{\ensuremath{\mathbb{R}^n}}

\newcommand{\eps}{\ensuremath{\varepsilon}}

\newcommand{\lr}[1]{\left( #1 \right)}

\newcommand{\eqlab}[1]{\begin{equation}  \begin{aligned}#1 \end{aligned}\end{equation}} 
\newcommand{\bgs}[1]{\begin{equation*} \begin{aligned}#1\end{aligned}\end{equation*}} 
  \newcommand{\sys}[2][]{\begin{equation*}#1  \left\{\begin{aligned}#2\end{aligned}\right.\end{equation*}}
\newcommand{\alig}[1] {\left\{\begin{aligned}#1 \end{aligned}\right.}

\newcommand{\I}{\mathcal I}

\newcommand{\Ha}{\mathcal H}
\newcommand{\Ll}{\mathcal L}
\newcommand{\Co}{\mathcal C}

\begin{document}
\date{}
\author{Claudia Bucur}

\address{Claudia Bucur: School of Mathematics and Statistics\\ The University of Melbourne \\  813 Swanston Street \\ Parkville VIC 3010, Australia}

\email{claudia.bucur@unimi.it}

\author{Luca Lombardini}
\address{Luca Lombardini: Dipartimento di Matematica\\ 
Universit\`a degli Studi di Milano \\ Via Cesare Saldini 50 
\\ 20133 Milano-Italy\\
and 
Facult\'e des Sciences\\
Universit\'e de Picardie Jules Verne\\ 
33 Rue Saint Leu\\ 
80039 Amiens CEDEX 1-France}
\email{luca.lombardini@unimi.it}

\author{Enrico Valdinoci}
\address{Enrico Valdinoci: 
Department of Mathematics and Statistics\\
University of Western Australia\\
35 Stirling Hwy\\ Crawley WA 6009-Australia\\
and
Dipartimento di Matematica\\ Universit\`a degli 
Studi di Milano \\ Via Cesare Saldini 50 \\ 20133 Milano-Italy\\
and
Istituto di Matematica Applicata e Tecnologie Informatiche\\
Consiglio Nazionale delle Ricerche\\
Via Ferrata 1\\ 27100 Pavia-Italy
}
\email{enrico@mat.uniroma3.it}
\email{}
\thanks{}

\title[]{Complete stickiness of nonlocal minimal surfaces
for small values of the fractional parameter}\thanks{This work has been supported by
the Australian Research Council Discovery Project DP170104880 NEW ``Nonlocal
Equations at Work''. The third author is member of INdAM/GNAMPA}
\begin{abstract}In this paper, we consider the asymptotic behavior
of the fractional mean curvature when $s\to 0^+$. Moreover, we deal with the behavior of $s$-minimal surfaces when the fractional parameter $s\in(0,1)$ is small, in a bounded and connected open set with $C^2$ boundary $\Omega\subset \Rn$. We classify the behavior of $s$-minimal surfaces with respect to the fixed exterior data (i.e. the $s$-minimal set fixed outside of $\Omega$). So, for $s$ small and depending
on the data at infinity,
the $s$-minimal set can be either empty in $\Omega$, fill all $\Omega$, 
or possibly develop a wildly oscillating boundary.

Also, we prove the continuity of the fractional mean curvature in all variables, for $s\in [0,1]$. Using this, we see that as the parameter $s$ varies, the fractional mean curvature may change sign. 
\end{abstract}

\maketitle

\tableofcontents

\section{Introduction and main results}

Since introduced by Caffarelli, Roquejoffre and Savin in 2010 in \cite{nms}, nonlocal minimal surfaces have become a very interesting subject of study.
The non-expert reader may take a look at \cite{lukes, bucval, senonlocal} and the references cited therein for an introduction of some recent results on this argument.

\medskip

In this paper, we deal with the behavior of nonlocal minimal surfaces when the fractional parameter (that we denote by $s\in (0,1)$) is small. In particular
\begin{itemize}
\item we give the asymptotic behavior of the fractional mean curvature as $s\to 0^+$, 
\item we classify the behavior of $s$-minimal surfaces, in dependence of the exterior data at infinity.
\end{itemize}  
Moreover, we prove the continuity of the fractional mean curvature in all variables for $s\in [0,1]$.


 \medskip
 
As a first thing, let us recall that the fractional perimeter is defined as
\begin{equation} \label{nmspf1} P_s(E,\Omega) := \Ll_s(E\cap \Omega,\Co E)
+\Ll_s(E\setminus \Omega,\Omega \setminus E), \end{equation}
where the interaction $\Ll_s(A,B)$ between two disjoint subsets of $\Rn$ is
	\begin{equation}
		\Ll_s(A,B):=\int_A \int_B \frac{dx\, dy}{|x-y|^{n+s}}
			=\int_{\Rn} \int_{\Rn} \frac{ \chi_A(x) \chi_B (x) }{|x-y|^{n+s}}\, dx \, dy.
	\end{equation}
	Let~$\Omega$ be an open set of~$\R^n$. 
We say that a set $E\subset \Rn$ is $s$-minimal in~$\Omega$
if $P_s(E,\Omega)$ is finite and if, for any competitor (for any set~$F$ such that $E\setminus\Omega =F\setminus \Omega$), we have that
	\[ P_s(E,\Omega) \leq P_s(F,\Omega).\]
The boundary of an $s$-minimal set is referred to as an $s$-minimal surface.
Furthermore, we introduce the $s$-fractional mean curvature of a set $E$ at a point $q\in\partial E$ (as the fractional counterpart of the classical mean curvature). It is defined as the principal value integral
\[\I_s[E](q):=P.V.\int_{\R^n}\frac{\chi_{\Co E}(y)-\chi_E(y)}{|y-q|^{n+s}}\,dy,\]
that is
\[\I_s[E](q):=\lim_{\rho\to0^+}\I_s^\rho[E](q),\qquad\textrm{where}\qquad
\I_s^\rho[E](q):=\int_{
\Co B_\rho(q)}\frac{\chi_{\Co E}(y)-\chi_E(y)}{|y-q|^{n+s}}\,dy.\] 
For the main properties of the fractional mean curvature, we refer e.g. to \cite{Abaty}.\\

Let us also introduce here the notation for the area of the $(n-1)$-dimensional sphere as 
 \[ \omega_n=\mathcal H^{n-1}\left( \{ x\in \Rn \, \big| \, |x|=1\}\right) 
 ,\]
 where  $\mathcal H^{n-1}
 $ is the $(n-1)$-dimensional Hausdorff  measure. The volume of the $n$-dimensional unit ball is 
  \[\Big|\{ x\in \Rn \, \big| \, |x|<1\} \Big|=\frac{\omega_n}n.\]
We denote also
  \[\omega_0=0.\]

The asymptotic behavior of nonlocal minimal surfaces as $s$ reaches $0$ or $1$ is, of course, 
a very interesting matter. Indeed, the small $s$ regime corresponds
to that of ``very strongly nonlocal interactions'' and, for small values
of~$s$, the regularity theory for nonlocal minimal surfaces may degenerate.\\
As $s\to 1^-$, one obtains the classical counterpart of the objects under study,
as the following known results show. 
For a set $E \subset \Rn$ with $C^{1,\gamma}$ boundary in $B_R$ for some $R>0$ and $\gamma\in(0,1)$, for almost any $r<R$ and up to constants one has indeed that
  \[ \lim_{s \to 1^-} (1-s)P_s(E,B_r)= P(E,B_r),\]
(see Theorem 1 in \cite{uniform}). A refined version of this asymptotic property can be obtained by making use
of Theorem 1 in \cite{davila} (see Theorem 1.8 in \cite{Myfractal}).

Moreover (see Theorem 12 in \cite{Abaty}, and \cite{regularity}) for a set $E\subset \Rn$ with $C^2$ boundary and any $x\in \partial E$, one has that
\[ \lim_{s \to 1} \mathcal (1-s)\I_s[E] (x) = \omega_{n-1}H[E](x),\]
where $H$ is the classical mean curvature of $E$ at the point $x$ (with the convention that we take $H$ such that the curvature of the ball is a positive quantity). \textcolor{black}{We notice that for $n=1$, we have that
\[  \lim_{s \to 1} \mathcal (1-s)\I_s[E] (x) = 0,\]
which is consistent with the notation $\omega_0=0$. See also Remark \ref{nuno}.}

Finally, as $s\to1^-$, $s$-minimal sets converge to minimizers of the classical perimeter, both in a ``uniform sense'' (see \cite{uniform, regularity}) and in a $\Gamma$-convergence sense (see \cite{gammaconv}).
As a consequence, one is able to prove (see \cite{regularity}) that for $s$ sufficiently close to 1, nonlocal minimal surfaces have the same regularity of classical minimal surfaces.
See also \cite{senonlocal} for a recent and quite comprehensive survey of the properties of $s$-minimal sets when $s$ is close to 1. 

\smallskip

As $s\to 0^+$, the asymptotic behavior is more involved and some surprising behavior may arise.  This is due to the fact that as $s$ gets smaller, the nonlocal contribution to the perimeter becomes more and more important, and the local behavior loses influence. Some precise
results in this sense were achieved in \cite{asympt1}. There, in order to mathematically encode the behavior at infinity of a set, the authors introduce the following quantity:
   \eqlab{\label{alpha} \alpha(E)=\lim_{s\to 0^+} s\int_{\Co B_1} \frac{\chi_E(y)}{|y|^{n+s}}\, dy,  }
  (see formula~(2.2)  in~\cite{asympt1}). The set function $\alpha(E)$ appears naturally when looking at the
behavior near $s=0$ of the fractional perimeter (see \cite{asympt1}). Indeed, let $\Omega$ be a bounded open set with $C^{1,\gamma}$ boundary, for some $\gamma\in (0,1)$, and $E \subset \Rn$ be a set with finite $s_0$-perimeter, for some $s_0\in (0,1)$. If $\alpha(E)$ exists, then
		\bgs{  \lim_{s\to 0^+} sP_s(E,\Omega)= &\; \alpha(\Co E) |E\cap \Omega| + \alpha(E) |\Co E \cap \Omega| 	.}

On the other hand, the asymptotic behavior for $s\to 0^+$ of the fractional mean curvature is studied in this paper
(see also \cite{senonlocal} for the particular case in which the set $E$ is bounded).

Moreover, as $s \to 0^+$, $s$-minimal sets may exhibit a rather unexpected behavior.
For instance, in \cite[Theorem 1.3]{boundary} it is proved that fixing the first quadrant of the plane as boundary data, quite surprisingly the $s$-minimal set in $B_1\subset\R^2$ is empty in $B_1$ for $s$ small enough.
The main results in this paper take their inspiration from this result.

\medskip


Let us mention that the stickiness phenomena described in \cite{boundary}
and in this paper are specific for nonlocal minimal surfaces (since
classical minimal surfaces cross transversally the boundary of a 
convex domain). 

Interestingly, these stickiness phenomena
are not present in the case of the fractional Laplacian,
where the boundary datum of the Dirichlet problem is attained continuously under rather
general assumptions, see \cite{MR3168912},
though solutions of $s$-Laplace
equations are in general not better than $C^s$ at the boundary, hence
the uniform continuity degenerates as~$s\to0^+$. Also,
solutions of $s$-Laplace equations with data growing like~$|x|^\alpha$
with~$\alpha\in(0,2)$ diverge as~$s\to(\alpha/2)^+$, as can be 
checked using the fractional
Poisson kernel,
and we plan to investigate in details
in a future project the continuity properties in dependence of suitably scaled
singular data at infinity.

On the other hand, in case of fractional harmonic functions, a partial counterpart of
the stickiness phenomenon is, in a sense, given by the boundary explosive
solutions constructed in \cite{MR3393247,MR2985500}
(namely, in this case, the boundary of the subgraph of the fractional harmonic function
contains vertical walls).
Other stickiness phenomena in nonlocal settings will be also studied
in a forthcoming article by the first two authors.\medskip

This paper is organized as follows. We set some notations and recall some known results in the following Subsection \ref{defnsknwown}. Also, we give some  preliminary results on the contribution from infinity of sets in Section \ref{contr_infty}.

In Section \ref{classify}, we consider exterior data ``occupying at infinity'' in measure, with respect to an appropriate weight, less than an half-space. To be precise
\bgs{\label{mainhyp} \alpha(E_0)<\frac{\omega_n}{2}. }
In this hypothesis:
\begin{itemize}
\item In Subsection \ref{sectnotfull} we give some asymptotic estimates of the density, in particular showing that when $s$ is small
enough $s$-minimal sets cannot fill their domain. 
\item In Subsection \ref{estimatecurvature} we give some estimates on the fractional mean curvature. In particular we show that if a set $E$ has an exterior tangent ball of radius $\delta$ at some point $p\in\partial E$, then the $s$-fractional mean curvature of $E$ in $p$ is strictly positive for every $s<s_\delta$. 
\item In Subsection \ref{alternative} we prove that when the fractional parameter is small and the exterior data at infinity occupies (in measure, with respect to the weight) less than half the space, then $s$-minimal sets completely stick at the boundary (that is, they are empty inside the domain), or become ``topologically dense'' in their domain. A similar result, which says that $s$-minimal sets fill the domain or their complementaries become dense, can be obtained in the same way, when the exterior data occupies in the appropriate sense more than half the space (so this threshold is somehow optimal).
\item Subsection \ref{sticky} narrows the set of minimal sets that become dense in the domain for $s$ small. As a matter of fact, if the exterior data does not completely surround the domain, $s$-minimal sets completely stick at the boundary. 
\end{itemize}
In Section \ref{sectexamples}, we provide some examples in which we are able to explicitly compute the contribution from infinity of sets. 
Section \ref{cont} contains the continuity of the fractional mean curvature operator in all its variables for $s\in[0,1]$. As a corollary, we show that for $s\to 0^+$ the fractional mean curvature at a regular point of the boundary of a set, takes into account only the behavior of that set at infinity.   The continuity property implies that the mean curvature at a regular point on the boundary of a set may change sign, as $s$ varies,
depending on the signs of the two asymptotics as $s\to1^-$ and $s\to0^+$.

In Appendix \ref{appendicite} and Appendix \ref{appendicite2} we collect some useful results that we use in this paper. Worth mentioning are Appendixes \ref{brr2} and \ref{appendicite3}. The first of the two gathers some known results on the regularity of $s$-minimal surfaces, so as to state the Euler-Lagrange equation pointwisely in the interior of $\Omega$. In the latter we prove that the Euler-Lagrange equation holds (at least as a inequality) at $\partial E \cap \partial \Omega$, as long as the two boundaries do not intersect ``transversally''.


\subsection {Statements of the main results}

 We remark that the quantity $\alpha$ (defined in \eqref{alpha}) may not exist (see Example 2.8 and 2.9 in \cite{asympt1}). For this reason, we also define
\eqlab{\label {baralpha1} \overline \alpha (E):= \limsup_{s\to 0^+} s\int_{\Co B_1} \frac{\chi_E(y)}{|y|^{n+s}}\, dy ,\quad \quad \underline \alpha(E) := \liminf_{s\to 0^+} s\int_{\Co B_1} \frac{\chi_E(y)}{|y|^{n+s}}\, dy.}

This set parameter plays an important role in describing
the asymptotic behavior of the fractional mean curvature as~$s\to0^+$
for unbounded sets. As a matter of fact, the limit as~$s\to0^+$
of the fractional mean curvature for a {\em bounded} set
is a positive, universal constant (independent of the set),
see e.g. (Appendix~B in \cite{senonlocal}).
On the other hand, this asymptotic behavior changes for {\em unbounded}
sets, due to the set function $\alpha(E)$, as described explicitly
in the following result:
  \begin{theorem}\label{asympts}[Proof in Section \ref{cont}]
Let $E\subset\Rn$ and let $p\in\partial E$ be such that $\partial E$ is $C^{1,\gamma}$ near $p$,
for some $\gamma\in(0,1]$. Then
\bgs{& \liminf_{s\to0^+} s\,\I_s[E](p) =\omega_n -2 \overline \alpha(E)\\
& \limsup_{s\to0^+} s\,\I_s[E](p) =\omega_n-2 \underline\alpha(E).}
\end{theorem}
We notice that if~$E$ is bounded, then $\underline\alpha(E)
=\overline\alpha(E)=\alpha(E)=0$, hence Theorem~\ref{asympts}
reduces in this case to formula~(B.1) in  \cite{senonlocal}.
Actually, we can estimate the fractional mean curvature from below (above) uniformly with respect to the radius of the exterior
(interior) tangent ball to $E$. To be more precise, if there exists an exterior tangent ball at $p\in \partial E$ of radius $\delta>0$, then for every $s<s_\delta$ we have 
\[ \liminf_{\rho \to 0^+} s\,\I^\rho_s[E](p)\geq \frac{\omega_n -2\overline \alpha(E)}4.\]    
More explicitly, we have the following result:
   \begin{theorem}\label{positivecurvature}[Proof in Section \ref{estimatecurvature}]
Let $\Omega\subset\Rn$ be a bounded open set. 
Let $E_0\subset\Co\Omega$ be such that
\eqlab{\label{weak_hp_beta}\overline \alpha(E_0)<\frac{\omega_n}2,}
and let
\[\beta=\beta(E_0):=\frac{\omega_n-2\overline \alpha(E_0)}4.\] We define 
\eqlab{\label{delta_wild_index_def}
\delta_s=\delta_s(E_0):=e^{-\frac{1}{s}\log \frac{\omega_n+2\beta}{\omega_n+\beta}} ,}
for every $s\in(0,1)$.
Then, there exists $s_0=s_0(E_0,\Omega)\in(0,\frac{1}{2}]$ such that, if $E\subset\Rn$ is such that $E\setminus\Omega=E_0$
and $E$ has an exterior tangent ball
of radius (at least) $\delta_\sigma$, for some $\sigma\in(0,s_0)$, at some point $q\in\partial E\cap\overline{\Omega}$, then
\eqlab{\label{unif_pos_curv}\liminf_{\rho\to0^+}\I_s^\rho[E](q)\geq\frac{\beta}{s}>0,\qquad\forall\,s\in(0,\sigma].}
 \end{theorem}

Given an open set $\Omega\subset\R^n$ and $\delta\in\R$, we consider the open set
\[\Omega_\delta:=\{x\in\R^n\,|\,\bar{d}_\Omega(x)<\delta\},\]
where $\bar{d}_\Omega$ denotes the signed distance function from $\partial\Omega$, negative inside $\Omega$.

It is well known (see e.g. \cite{GilTru,Ambrosio})
that if $\Omega$ is bounded and $\partial \Omega$ is of class $C^2$, then the distance function is also of class $C^2$
in a neighborhood of $\partial\Omega$. Namely, there exists $r_0>0$
such that
\[\bar{d}_\Omega\in C^2(N_{2r_0}(\partial\Omega)),\quad
\mbox{ where }\quad N_{2r_0}(\partial\Omega):=\{x\in\R^n\,|\,|\bar{d}_\Omega(x)|<2r_0\}.\]
As a consequence, since $|\nabla\bar{d}_\Omega|=1$,
the open set $\Omega_\delta$ has $C^2$ boundary for every $|\delta|<2r_0$.
For a more detailed discussion, see Appendix \ref{A2} and the references cited therein.

The constant $r_0$ will have the above meaning throughout this whole paper.

\smallskip

We give the next definition.
\begin{defn}\label{wild}
   Let $\Omega\subset \Rn$ be an open, bounded set.
   We say that a set $E$ is $\delta$-{dense} in $\Omega$ for some fixed $\delta>0$ if $|B_\delta(x)\cap E|>0$ for any $x\in \Omega$ for which $B_\delta(x)\subset\subset\Omega$.
  \end{defn}
  
\noindent Notice that if $E$ is $\delta$-dense  then $E$ cannot have an exterior tangent ball of radius greater or equal than $\delta$ at any point $p\in \partial E\cap \Omega_{-\delta}$.

\noindent We observe that the notion for a set of
being $\delta$-dense is a ``topological'' notion,
rather than a measure theoretic one. 
Indeed, $\delta$-dense sets need not be ``irregular'' nor ``dense'' in the measure theoretic sense (see Remark \ref{deltadance}).

\smallskip
With this definition and using Theorem \ref{positivecurvature} we obtain the following classification.
\begin{theorem}\label{THM}[Proof in Section \ref{alternative}]
  Let $\Omega$ be a bounded and  connected open set with $C^2$ boundary. Let $E_0\subset \Co \Omega$ such that\[\overline \alpha(E_0)<\frac{\omega_n}{2}.\]  
 Then the following two results hold.\\
  A)  Let $s_0$ and $\delta_s$ be as in Theorem \ref{positivecurvature}. There exists
  $s_1=s_1(E_0,\Omega)\in (0,s_0]$ such that if $s<s_1$ and $E$ is an $s$-minimal set in $\Omega$ with exterior data $E_0$, then either
     \bgs{(A.1) \;  E\cap \Omega=\emptyset \quad  \mbox{ or} \quad\; (A.2)\;  E \mbox{ is } \delta_s-\mbox{dense}.}
 \noindent
 B) Either \\
(B.1) there exists
  $\tilde s=\tilde s(E_0,\Omega)\in (0,1)$ such that if $E$ is an $s$-minimal set in $\Omega$ with exterior data $E_0$ and $s\in(0,\tilde s)$, then
     \bgs{  E\cap \Omega=\emptyset,}
     or \\
    (B.2)    there exist  $\delta_k \searrow 0$, $s_k \searrow 0$ and a sequence of sets  $E_k$ such that each $E_k$ is $s_k$-minimal in $\Omega$ with exterior data $E_0$ and for every $k$
     \bgs{ \partial E_k \cap B_{\delta_k}(x) \neq \emptyset \quad \forall \; B_{\delta_k}(x)\subset\subset \Omega.}
     \end{theorem}

We remark here that Definition \ref{wild} allows the $s$-minimal set 
to completely fill $\Omega$. The next theorem states that for $s$ small enough (and $\overline \alpha(E)<\omega_n/2$) we can exclude  this possibility.
\begin{theorem}\label{notfull}[Proof in Section \ref{sectnotfull}]
Let $\Omega\subset\R^n$ be a bounded open set of finite classical perimeter and let $E_0\subset\Co\Omega$ be such that
\[\overline{\alpha}(E_0)<\frac{\omega_n}{2}.\]
For every $\delta>0$ and every $\gamma\in(0,1)$ there exists $\sigma_{\delta,\gamma}=\sigma_{\delta,\gamma}(E_0,\Omega)\in(0,\frac{1}{2}]$ such that if $E\subset\R^n$ is $s$-minimal in $\Omega$, with exterior data $E_0$ and $s<\sigma_{\delta,\gamma}$, then
\eqlab{\label{1666}\big|(\Omega\cap B_\delta(x))\setminus E\big|\ge\gamma\, \frac{\omega_n-2\overline{\alpha}(E_0)}{\omega_n-\overline{\alpha}(E_0)}\big|\Omega\cap B_\delta(x)\big|,\qquad\forall\,x\in\overline{\Omega}.}
\end{theorem}

\begin{remark}
 Let $\Omega$ and $ E_0$ be as in Theorem \ref{notfull} and fix $\gamma=\frac{1}{2}$.
 \begin{enumerate}
\item Notice that we can find $\bar \delta >0$ and  $\bar x \in \Omega$ such that
\[ B_{2\bar\delta} (\bar x ) \subset \Omega.\]
Now if $s<\sigma_{\bar \delta,\frac{1}{2}}$ and $E$ is $s$-minimal in $\Omega$ with respect to $E_0$, \eqref{1666} says that
\[ |B_{\bar\delta} (\bar x ) \cap \Co E|>0.\] 
Then (since the ball is connected), either $B_{\bar\delta} (\bar x ) \subset \Co E$ or there exists a point
\[x_0\in\partial E\cap \overline B_{\bar\delta} (\bar x ).\]
In this case, since $d(x_0, \partial \Omega )\ge\bar \delta$, Corollary 4.3 of \cite{nms} implies that
\[B_{\bar\delta c_s}(z)\subset\Co E\cap B_{\bar \delta}(x_0)\subset\Co E\cap\Omega\] for some $z$, where $c_s\in(0,1]$ denotes the constant of the clean ball condition (as introduced in Corollary 4.3 in \cite{nms}) and depends only on $s$ (and $n$). In both cases, there exists a ball of radius $\bar \delta c_s$ contained in $\Co E \cap \Omega$. 
\item If $s<\sigma_{\bar \delta,\frac{1}{2}}$ and $E$ is $s$-minimal and $\delta_s$-dense, then 
we have that
\[\delta_s>c_s\bar \delta.\]
On the other hand, we have an explicit expression for $\delta_s$, given in \eqref{delta_wild_index_def}. Therefore, if one could prove that $c_s$ goes to zero slower than $\delta_s$, one could exclude the existence of $s$-minimal sets that are $\delta_s$-dense (for all sufficiently small $s$). 
\end{enumerate}
\end{remark}
\smallskip 

     
     An interesting result is related to $s$-minimal sets whose exterior data does not completely surround $\Omega$. In this case, the $s$-minimal set, for small values of $s$, is always
empty in $\Omega$. More precisely:

 \begin{theorem}\label{boundedset}[Proof in Section \ref{sticky}]
  Let $\Omega$ be a  bounded and  connected open set with $C^2$ boundary. Let $E_0\subset \Co \Omega$ such that  \[\overline \alpha(E_0)<\frac{\omega_n}{2},\]
 and let $s_1$ be as in Theorem \ref{THM}. Suppose that there exists $R>0$ and $x_0\in \partial \Omega$ such that \[B_R(x_0)\setminus \Omega \subset \Co E_0.\]  Then, there exists $s_3=s_3(E_0,\Omega)\in(0,s_1]$ such that if $s<s_3$ and $E$ is an $s$-minimal set in $\Omega$ with exterior data $E_0$, then 
    \[  E\cap \Omega=\emptyset .\]
     \end{theorem}

We notice that Theorem \ref{boundedset} prevents the existence 
of $s$-minimal sets that are $\delta$-dense (for any $\delta$).    

\begin{remark}
The indexes $s_1$ and $s_3$ are defined as follows
\[s_1:=\sup\{s\in(0,s_0)\,|\,\delta_s<r_0\}\]
and
\[s_3:=\sup\Big\{s\in(0,s_0)\,\big|\,\delta_s<\frac{1}{2}\min\{r_0,R\}\Big\}.\]
Clearly, $s_3\leq s_1\leq s_0$.
\end{remark}

\begin{remark} We point out that condition \eqref{weak_hp_beta}
is somehow optimal. Indeed,
when $\alpha(E_0)$ exists and 
\[ \alpha(E_0)=\frac{\omega_n}2,\]
several configurations may occur, depending on the position of $\Omega$ with respect to the exterior data $E_0\setminus \Omega$. As an example, take 
\[ \mathfrak P =\{ (x',x_n) \; \big| \; x_n> 0\}.\] Then, for any $\Omega\subset \Rn$  a bounded open set with $C^2$ boundary, the only $s$-minimal set with exterior data given by $\mathfrak P \setminus \Omega$ is $\mathfrak P$ itself. So, if $E$ is $s$-minimal with respect to $E_0\setminus \Omega$ then
\bgs{&\Omega\subset \mathfrak P & \quad \implies \quad &E\cap \Omega=\Omega\\
 & \Omega\subset \Rn \setminus \mathfrak P &\quad \implies \quad &E\cap \Omega=\emptyset.} 
 On the other hand, if one takes $\Omega= B_1$, then 
 \[  E\cap B_1 = \mathfrak P  \cap B_1.\] 
 
 As a further example, we consider the supergraph
 \[ E_0:=\{(x',x_n) \; \big| \; x_n > \tanh x_1\},\] for which we have that (see Example \ref{tanh})
 \[\alpha(E_0)=\frac{\omega_n}2.\]  Then for every $s$-minimal set in $\Omega$ with exterior data $E_0\setminus \Omega$, we have that
 \bgs{&\Omega\subset \{ (x',x_n) \; \big| \; x_n> 1\} & \quad \implies \quad &E\cap \Omega=\Omega\\
 & \Omega\subset \{ (x',x_n) \; \big| \; x_n<-1\} &\quad \implies \quad &E\cap \Omega=\emptyset.} 
Taking $\Omega=B_2$, we have by the maximum principle in Proposition \ref{maximum_principle}  that every set $E$ which is $s$-minimal in $B_2$, with respect to $E_0\setminus B_2$, satisfies
\bgs{  B_2\cap  \{ (x',x_n) \; \big| \; x_n> 1\}\subset E,   \qquad 
  B_2 \cap \{ (x',x_n) \; \big| \; x_n<-1\} \subset \Co E .} 
 On the other hand, we are not able to establish what happens in $B_2\cap \{ (x',x_n) \; \big| \; -1<x_n< 1\} $.
\end{remark}

\begin{remark}
We notice that when $E$ is $s$-minimal in $\Omega$ with respect to $E_0$, then $\Co E$ is $s$-minimal in $\Omega$ with respect to $\Co E_0$. Moreover
\[ \underline \alpha(E_0) >\frac{\omega_n}{2} \qquad \implies \qquad \overline \alpha (\Co E_0)< \frac{\omega_n}{2}.\]
So in this case we can apply Theorems \ref{positivecurvature}, \ref{THM}, \ref{notfull} and \ref{boundedset} to $\Co E$ with respect to the exterior data $\Co E_0$. For instance, if
$E$ is $s$-minimal in $\Omega$ with exterior data $E_0$ with
\[ \underline \alpha(E_0) >\frac{\omega_n}{2}, \]
and $s<s_1(\Co E_0, \Omega)$,
 then either
\[ E\cap \Omega=\Omega \qquad \mbox{ or }  \qquad  \Co E \; \mbox{ is } \; \delta_s(\Co E_0)-\mbox{dense}.\]
 The analogues of the just mentioned Theorems can be obtained similarly.
\end{remark}

We point out that from our main results and the last two remarks, we have a complete classification of nonlocal minimal surfaces when $s$ is small whenever
\[ \alpha(E_0)\neq  \frac{\omega_n}{2} .\] 

In the last section of the paper, we prove the continuity of the fractional mean curvature in all variables (see Theorem \ref{everything_converges} and Proposition \ref{propsto0}). As a consequence, we have the following result. 

\begin{prop}\label{rsdfyish}
Let $E\subset\R^n$ and let $p\in\partial E$ such that $\partial E$ is $C^{1,\alpha}$ in $B_R(p)$ for some
$R>0$ and $\alpha\in(0,1]$. Then the function
\[\I_{(-)}[E](-):(0,\alpha)\times(\partial E\cap B_R(p))\longrightarrow\R,
\qquad(s,x)\longmapsto\I_s[E](x)\]
is continuous.\\
Moreover, if $\partial E\cap B_R(p)$ is $C^2$ and for every $x\in \partial E\cap B_R(p)$ we define
\sys[ \tilde \I_s  {[}E{]} (x):=]{ &s(1-s)\I_s[E](x),  & \mbox{ for } &s\in (0,1) \\	
						&{\omega_{n-1}} H[E](x), &\mbox{ for } &s=1,}
then the function
\[\tilde \I_{(-)}[E](-):(0,1]\times(\partial E\cap B_R(p))\longrightarrow\R,
\qquad(s,x)\longmapsto \tilde \I_s[E](x)\]
is continuous.\\
Finally, if $\partial E\cap B_R(p)$ is $C^{1,\alpha}$ and $\alpha(E)$ exists, and if for every $x\in \partial E\cap B_R(p)$ we denote
\[\tilde\I_0[E](x):=\omega_n-2\alpha(E),\]
then the function
\[\tilde \I_{(-)}[E](-):[0,\alpha)\times(\partial E\cap B_R(p))\longrightarrow\R,
\qquad(s,x)\longmapsto \tilde \I_s[E](x)\]
is continuous.
\end{prop}

As a consequence of the continuity of the fractional mean curvature and the asymptotic result in  Theorem \ref{asympts} we
establish that, by varying the fractional parameter $s$,
the nonlocal mean curvature may change sign at a point
where the classical mean curvature is negative, as one can observe in Theorem \ref{changeyoursign}.


\subsection{Definitions, known facts and notations}\label{defnsknwown}
We recall here some basic facts on $s$-minimal sets and surfaces, on the fractional mean curvature operator, and some notations, that we will use in the course of this paper.

\subsubsection{Measure theoretic assumption}\label{MEAS:ASS:SEC}
The following notations and measure theoretic assumptions are assumed throughout the paper.

Let $E\subset\R^n$. Up to modifying $E$ on a set of measure zero we can assume (see e.g. Appendix C of \cite{Myfractal})
that $E$ contains the measure theoretic interior
\begin{equation*}
E_{int}:=\Big\{x\in\R^n\,|\,\exists\,r>0\textrm{ s.t. }|E\cap B_r(x)|=\frac{\omega_n}n r^n\Big\}\subset E,
\end{equation*}
the complementary $\Co E$ contains its measure theoretic interior
\begin{equation*}
E_{ext}:=\{x\in\R^n\,|\,\exists\,r>0\textrm{ s.t. }|E\cap B_r(x)|=0\}\subset\Co E,
\end{equation*}
and the topological boundary of $E$ coincides with its measure theoretic boundary, $\partial E=\partial^-E$,
where
\begin{equation*}\begin{split}
\partial^-E&:=\R^n\setminus(E_{int}\cup E_{ext})\\
&
=\Big\{x\in\R^n\,|\,0<|E\cap B_r(x)|<\frac{\omega_n}{n}r^n\textrm{ for every }r>0\Big\}.
\end{split}
\end{equation*}
In particular, we remark that both $E_{int}$ and $E_{ext}$ are open sets.


\subsubsection{H\"{o}lder continuous functions}
We will use the following notation for the class of H\"{o}lder continuous functions.

Let $\alpha\in (0,1]$, let $S\subset\R^n$ and let $v:S\longrightarrow\R^m$. The $\alpha$-H\"{o}lder semi-norm of $v$ in $S$
is defined as
\[[v]_{C^{0,\alpha}(S,\R^m)}:=\sup_{x\neq y\in S}\frac{|v(x)-v(y)|}{|x-y|^\alpha}.\]
With a slight abuse of notation, we will omit the $\R^m$ in the formulas.
We also define
\[\|v\|_{C^0(S)}:=\sup_{x\in S}|v(x)|\quad\textrm{and}\quad\|v\|_{C^{0,\alpha}(S)}
:=\|v\|_{C^0(S)}+[v]_{C^{0,\alpha}(S)}.\]

Given an open set $\Omega\subset\R^n$, we define the space of uniformly H\"{o}lder continuous functions
$C^{0,\alpha}(\overline \Omega,\R^m)$ as
\[C^{0,\alpha}(\overline \Omega,\R^m):=\{v\in C^0(\overline{\Omega},\R^m)\,|\,
\|v\|_{C^{0,\alpha}(\overline{\Omega})}<\infty\}.\] 

Recall that $C^1(\overline{\Omega})$ is the space of those functions $u:\overline{\Omega}\longrightarrow\R$ such that
$u\in C^0(\overline{\Omega})\cap C^1(\Omega)$ and
such that $\nabla u$ can be continuously extended to $\overline{\Omega}$.
For every $S\subset\overline{\Omega}$ we write
\[\|u\|_{C^{1,\alpha}(S)}:=\|u\|_{C^0(S)}+\|\nabla u\|_{C^{0,\alpha}(S)},\]
and we define
\[C^{1,\alpha}(\overline \Omega):=\{u\in C^1(\overline{\Omega})\,|\,
\|u\|_{C^{1,\alpha}(\overline{\Omega})}<\infty\}.\]

We will usually consider the local versions of the above spaces. Given an open set $\Omega\subset\R^n$,
the space of locally H\"{o}lder continuous functions $C^{k,\alpha}(\Omega)$, with $k\in\{0,1\}$, is defined as
\[C^{k,\alpha}(\Omega):=\{u\in C^k(\Omega)\,|\,\|u\|_{C^{k,\alpha}(\mathcal O)}<\infty \mbox{ for every } \mathcal O\subset \subset \Omega\}.\] 

\subsubsection{The Euler-Lagrange equation}
 
We recall that the fractional mean curvature gives the Euler-Lagrange equation of an $s$-minimal set.
To be more precise, if $E$ is $s$-minimal in $\Omega$, then
\[\I_s[E]=0,\qquad\textrm{on}\quad\partial E\cap\Omega,\]
in an appropriate viscosity sense (see Theorem 5.1 of\cite{nms}).

Actually, by exploiting the interior regularity theory of $s$-minimal sets, the equation is satisfied in the classical sense
in a neighborhood of every ``viscosity point'' (see Appendix \ref{brr2}). 
That is, if $E$ has at $p\in\partial E\cap\Omega$ a tangent ball (either interior or exterior), then $\partial E$ is $C^\infty$ in $B_r(p)$, for some $r>0$ small enough, and
\[\I_s[E](x)=0,\qquad\forall\,x\in\partial E\cap B_r(p).\] 
Moreover, if $\Omega$ has a $C^2$ boundary, then the Euler-Lagrange equation (at least as an inequality) holds also at a point $p\in\partial E\cap\partial\Omega$,
provided that the boundary $\partial E$ and the boundary $\partial\Omega$ do not intersect ``transversally'' in $p$ (see Theorem \ref{EL_boundary_coroll}).


\section{Contribution to the mean curvature coming from infinity}\label{contr_infty}
In this section, we study in detail the quantities $\alpha(E)$, $\overline\alpha(E),\underline \alpha(E)$) as defined in \eqref{alpha}, \eqref{baralpha1}. As a first remark, notice that these definitions  are independent on the radius of the ball (see Observation 3 in \cite{asympt1}, Subsection 3.3) so we have that for any $R>0$
\eqlab{ \label{baralpha}  \overline \alpha (E)= \limsup_{s\to 0^+} s\int_{\Co B_R} \frac{\chi_E(y)}{|y|^{n+s}}\, dy, \quad \underline \alpha(E) := \liminf_{s\to 0^+} s\int_{\Co B_R} \frac{\chi_E(y)}{|y|^{n+s}}\, dy .}
Notice that
\[ \overline \alpha(E) = \omega_n -\underline \alpha(\Co E),  \quad \underline \alpha(E) = \omega_n - \overline \alpha(\Co E).\] 
We define
\[ \alpha_s(q,r,E):=\int_{\Co B_r(q)} \frac{\chi_{E}(y) }{|q-y|^{n+s}} \, dy .\]
Then, the quantity $\alpha_s(q,r,E)$
somehow ``stabilizes'' for small $s$ independently on how large or where we take the ball, as rigorously given by the following result:

\begin{prop}\label{unifrq}
Let $K\subset \Rn$ be a compact set and $[a,b]\subset \R$ be a closed interval, with $0<a<b$. Then 
\bgs{\label{name1} \lim_{s\to 0^+}s|\alpha_s(q,r,E)-\alpha_s(0,1,E)| =0\quad \mbox{ uniformly in } q \in K, r\in [a,b].
}
Moreover,
for any bounded open set $\Omega\subset \Rn$ and any fixed $r>0$,
we have that \eqlab{\label{claimalpha}\limsup_{s\to 0^+} s\inf_{q\in \overline \Omega} \alpha_s(q,r,E)= \limsup_{s\to 0^+} s\sup_{q\in \overline \Omega} \alpha_s(q,r,E)=\overline \alpha(E).}

\end{prop}
\begin{proof}
Let us fix $r\in [a,b]$ and $q\in K$, and $R>0$ such that $K\subset \overline B_{ R}$. Let also $\eps\in (0,1)$ be a fixed positive small quantity (that we will take arbitrarily small further on), such that 
\[ R>(\eps b)/(1-\eps).\]
We notice that if $x\in B_r(q)$, we have that 
$|x|<r+|q|<{R}/{\eps},$
hence $B_r(q)\subset B_{R/\eps}$.
We write that
\[\alpha_s(q,R,E)= \int_{\Co B_r(q)}\frac{\chi_E(y)}{|q-y|^{n+s}}\, dy = \int_{\Co B_{ R/\eps}} \frac{\chi_E(y)}{|q-y|^{n+s}}\, dy + \int_{ B_{ R/\eps}\setminus B_r(q)} \frac{\chi_E(y)}{|q-y|^{n+s}}\, dy.\]
Now for $y\in \Co B_{R/\eps}$ we have that $|y-q|\geq |y|-|q| \geq (1-\eps)|y|$, thus for any $q\in \overline B_R$ 
\eqlab{\label{secondaaa1} \int_{ \Co B_{ R/\eps}} \frac{\chi_E(y)}{|q-y|^{n+s}} \, dy \leq &\; (1-\eps)^{-n-s} \int_{\Co B_{ R/\eps}} \frac{\chi_E(y)}{|y|^{n+s}}\, dy =(1-\eps)^{-n-s} \alpha_s(0,R/\eps,E)
. }
Moreover 
\eqlab{\label{secondaaa11} \int_{ B_{ R/\eps}\setminus B_r(q)} \frac{\chi_E(y)}{|q-y|^{n+s}}\, dy \leq &\; \int_{B_{ R/\eps}\setminus B_r(q)} \frac{dy}{|q-y|^{n+s}} \leq \omega_n \int_r^{R/\eps+ R} t^{-s-1}\, dt\\ = &\; \omega_n \frac{r^{-s}- R^{-s}\eps^s(1+\eps)^{-s} }{s} \leq \omega_n \frac{a^{-s}- R^{-s}\eps^s(1+\eps)^{-s} }{s} . }
Notice also that
since $B_r(q)\subset B_{R/\eps}$ and $|q-y|\leq |q|+|y|\leq (\eps+1)|y|$ for any $y\in \Co B_{R/\eps}$,
we obtain that
\eqlab{\label{primariga}\int_{\Co B_r(q)}\frac{\chi_E(y) }{|q-y|^{n+s}}\, dy \geq &\; \int_{\Co B_{ R/\eps}} \frac{\chi_E(y) }{|q-y|^{n+s}}\, dy  \geq (1+\eps)^{-n-s} \int_{\Co B_{ R/\eps}} \frac{\chi_E(y)}{|y|^{n+s}}\, dy.
}
Putting\eqref{secondaaa1}, \eqref{secondaaa11} and \eqref{primariga} together, we get that
\bgs{  0\leq \alpha_s(q,r,E) -(1+\eps)^{-n-s}  \alpha_s(0,R/\eps,E)
 \leq &\;\alpha_s(0, R/\eps, E) \left((1-\eps)^{-n-s}-(1+\eps)^{-n-s}\right)  \\
 &\;+ \omega_n \frac{a^{-s}- R^{-s}\eps^s(1+\eps)^{-s} }{s}.}
Now we have that
\bgs{| \alpha_s(0, R/\eps,E)
-\alpha_s(0,1,E)| \leq \left|\int_{B_{R/\eps}\setminus B_1} \frac{dy}{|y|^{n+s}} \right| \leq \omega_n \frac{ |1-R^{-s}\eps^{s}|}{s}. 
}
So by the triangle inequality we obtain 
\bgs{ |\alpha_s(q,r,E)-&(1+\eps)^{-n-s}\alpha_s(0,1,E)|
\leq \alpha_s(0, R/\eps, E)  \left((1-\eps)^{-n-s}-(1+\eps)^{-n-s}\right)
 \\
  &\;+ \frac{\omega_n}s \big[ a^{-s}- R^{-s}\eps^s(1+\eps)^{-s} +  (1+\eps)^{-n-s}   { |1-R^{-s}\eps^{s}|}\big] . }
Hence,
it holds that
\[ \limsup_{s\to 0^+}s |\alpha_s(q,r,E)-(1+\eps)^{-n}\alpha_s(0,1,E)| \leq  \left((1-\eps)^{-n}-(1+\eps)^{-n}\right) \overline \alpha(E) ,\]
uniformly in $q\in K$ and in $r\in [a,b]$.\\
Letting $\eps \to 0^+$, 
we conclude that
 \[\lim_{s\to 0^+}s|\alpha_s(q,r,E)-\alpha_s(0,1,E)| =0,\]
uniformly in $q\in K$ and in $r\in [a,b]$.

Now, we consider $K$ such that $K=\overline \Omega$.
 Using the inequalities \eqref{secondaaa1}, \eqref{secondaaa11} and \eqref{primariga} we have that for any $q\in \overline \Omega$
\bgs{  (1+\eps)^{-n-s} \int_{\Co B_{ R/\eps}} \frac{\chi_E(y)}{|y|^{n+s}} \,dy \leq &\int_{\Co B_r(q)}\frac{\chi_E(y)}{|q-y|^{n+s}}\, dy\\ \leq&\;  
   (1-\eps)^{-n-s} \int_{\Co B_{R/\eps}} \frac{\chi_E(y)}{|y|^{n+s}}\, dy + \omega_n \frac{a^{-s}-R^{-s}\eps^s(1+\eps)^{-s}}s. }
Passing to limsup 
it follows that
\bgs{ &(1+\eps)^{-n} \overline \alpha(E) \leq \limsup_{s\to 0^+} s\inf_{q\in \overline \Omega} \int_{\Co B_r(q)}\frac{\chi_E(y)}{|q-y|^{n+s}} \, dy\\& \leq \limsup_{s\to 0^+} s\sup_{q\in \overline \Omega}\int_{\Co B_r(q)}\frac{\chi_E(y)}{|q-y|^{n+s}} \, dy  \leq  (1-\eps)^{-n} \overline\alpha(E).}
Sending $\eps \to 0$  we obtain the conclusion.
\end{proof}

\begin{remark}\label{finmeas}
Let $E\subset \Rn$ be such that $|E|<\infty$. Then
\[ \alpha(E)=0.\]
Indeed,
\[ |\alpha_s(0,1,E)|\leq |E|,\]
hence
\[ \limsup_{s\to 0} s|\alpha_s(0,1,E)|=0.\]
 \end{remark}
 
Now, we discuss some useful properties of $\overline\alpha$. 
Roughly speaking, the quantity $\overline\alpha$ takes into account
the ``largest possible asymptotic opening'' of a set, and so it
possesses nice geometric features such as monotonicity, additivity
and geometric invariances. The detailed list of these properties is
the following:

\begin{prop}\label{subsetssmin} \quad \\ 
(i) (Monotonicity) Let $E,F\subset \Rn$ be such that for some $r>0$ and $q\in \Rn$\[ E\setminus B_r(q)\subset  F\setminus B_r(q).\] Then
\[\overline \alpha(E)\leq \overline \alpha(F).\]
(ii) (Additivity) Let $E,F\subset \Rn$ be such that for some $r>0$ and $q\in \Rn$ \[ (E\cap F)\setminus B_r(q)= \emptyset.\] Then
\[ \overline \alpha (E\cup F)\leq \overline \alpha(E)+\overline \alpha(F).\]
Moreover, if $\alpha(E), \alpha(F)$ exist, then $\alpha(E\cup F)$ exists and
\[ \alpha(E\cup F)= \alpha(E)+\alpha(F).\]
(iii) (Invariance with respect
to rigid motions) Let $E\subset \Rn$, $x\in \Rn$ and $\mathcal R \in \mathcal {SO}(n)$ be a rotation. Then
\[ \overline \alpha(E+x)=\overline \alpha(E) \quad{\mbox{ and }}\quad \overline \alpha( \mathcal R E)=\overline \alpha(E).\]
(iv) (Scaling) Let $E\subset \Rn$ and $\lambda >0$. Then for some $r>0$ and $q\in \Rn$
\[ \alpha_s(q,r,\lambda E) =  \lambda^{-s} 
\alpha_s\left(\frac{q}{\lambda},\frac{r}{\lambda},E\right)
\quad{\mbox{ and }}\quad
\overline \alpha(\lambda E) =\overline \alpha(E).\] 
(v) (Symmetric difference)
Let  $E, F\subset \Rn$. Then for every $r>0$ and $q\in \Rn$
\[ |\alpha_s(q,r,E)-\alpha_s(q,r,F)|\leq \alpha_s(q,r,E\Delta F).\]
As a consequence, if $|E\Delta F|<\infty$ and $\alpha(E)$  exists, then $\alpha(F)$ exists and  
\[\alpha(E)=\alpha(F). \]
\end{prop}

\begin{proof} (i) It is enough to notice that for every $s\in (0,1)$
\[ \alpha_s(q, r,E) \leq \alpha_s(q, r,F) .\]
Then, passing to limsup and recalling \eqref{claimalpha} we conclude that
\[\overline \alpha(E)\leq \overline \alpha(F).\]
(ii)  We notice that for every $s\in (0,1)$ \[ \alpha_s(q,r,E\cup F) = \alpha_s(q,r,E)+\alpha_s(q,r, F) \]  
and passing to limsup and liminf as $s\to 0^+$  we obtain the desired claim.\\
(iii) By a change of variables,
we have that
\[ \alpha_s(0,1,E+x)= \int_{\Co B_1} \frac{\chi_{E+x}(y)}{|y|^{n+s}}\, dy = \int_{\Co B_1(-x)} \frac{\chi_{E}(y)}{|x+y|^{n+s}}\, dy= \alpha_s(-x,1,E).\]
Accordingly, the invariance by translation 
follows after passing to limsup and using \eqref{claimalpha}. \\
In addition, the invariance by rotations is obvious, using a change of variables.\\
(iv) Changing the variable $y=\lambda x$ we deduce that 
\[ \alpha_s(q,r,\lambda E)=\int_{\Co B_r(q)}\frac{\chi_{\lambda E} (y)}{|q-y|^{n+s}}\, dy
=\lambda ^{-s}\int_{\Co B_{\frac{r}{\lambda}}(\frac{q}{\lambda})} \frac{\chi_E(x)}{|\frac{q}{\lambda}-x|^{n+s}}\, dx
=\lambda^{-s} \alpha_s\left(\frac{q}{\lambda},\frac{r}{\lambda},E\right).\] Hence, the claim follows by passing to limsup as $s\to 0^+$.\\
(v) We have that
\[ |\alpha_s(q,r,E)-\alpha_s(q,r,F)| \leq \int_{\Co B_r(q)} \frac{ |\chi_{E}(y)-\chi_F(y)|}{|y-q|^{n+s}}\, dy = \int_{\Co B_r(q)} \frac{ \chi_{E\Delta F}(y)}{|y-q|^{n+s}}\, dy = \alpha_s(q,r,E\Delta F).\]
The second part of the claim follows applying the Remark \ref{finmeas}.
\end{proof}

We recall the definition (see (3.1) in  \cite{asympt1})
\[\mu(E):=\lim_{s \to 0^+} s P_s(E,\Omega),\]
where $\Omega$ is a bounded open set with $C^2$ boundary.
Moreover, we define 
\[ \overline \mu(E)= \limsup_{s\to 0^+} sP_s(E,\Omega)\]
and give the following result:
\begin{prop}\label{barmubaral} Let $\Omega\subset\Rn$ be a bounded open set with finite classical perimeter
and let $E_0\subset\Co \Omega$. Then
\[ \overline \mu(E_0)= \overline \alpha(E_0) |\Omega|.\]
\end{prop}
\begin{proof}
Let $R>0$ be fixed such that $\Omega \subset B_R$, $y\in \Omega$ be any fixed point and $\eps\in(0,1)$ be small enough such that $R/\eps> R+1$. This choice of $\eps$ assures that $B_1(y)\subset B_{R/\eps}$. 
We have that
\[ \int_{\Rn} \frac{\chi_{E_0}(x)}{|x-y|^{n+s}}\, dx =  \int_{\Co B_{R/\eps} } \frac{\chi_{E_0}(x)}{|x-y|^{n+s}}\, dx  + \int_{ B_{R/\eps}\setminus B_{1}(y)} \frac{\chi_{E_0}(x)}{|x-y|^{n+s}}\, dx + \int_{  B_{1}(y)} \frac{\chi_{E_0}(x)}{|x-y|^{n+s}}\, dx  .\]
Since $|x-y|\geq (1-\eps)|x|$ whenever $x\in \Co B_{R/\eps}$, we get
\[ \int_{\Co B_{R/\eps} } \frac{\chi_{E_0}(x)}{|x-y|^{n+s}}\, dx\leq (1-\eps)^{-n-s}\int_{\Co B_{R/\eps} } \frac{\chi_{E_0}(x)}{|x|^{n+s}}\, dx.\]
Also we have that
\[ \int_{ B_{R/\eps}\setminus B_{1}(y)} \frac{\chi_{E_0}(x)}{|x-y|^{n+s}}\, dx\leq \omega_n\int_1^{R/\eps+R} \rho^{-s-1}\, d\rho \leq \omega_n\frac{1- \left(\frac{R}{\eps}+R\right)^{-s}}s.\]
Also, we can assume that $s<1/2$ 
(since we are interested in what happens for $s\to 0$). In this way,
if $|x-y|<1$ we have that $|x-y|^{-n-s}\leq|x-y|^{-n-\frac{1}{2}}$, and so
\[\int_{  B_1(y)} \frac{\chi_{E_0}(x)}{|x-y|^{n+s}}\, dx\leq \int_{  B_1(y)} \frac{\chi_{E_0}(x)}{|x-y|^{n+\frac{1}2}}\, dx.\] 
Also, since $E_0\subset\Co\Omega$, we have that 
\[\int_{  B_1(y)} \frac{\chi_{E_0}(x)}{|x-y|^{n+\frac{1}2}}\, dx  \leq \int_{  B_1(y)\setminus \Omega} \frac{dx}{|x-y|^{n+\frac{1}2}}
\leq \int_{\Co \Omega} \frac{dx}{|x-y|^{n+\frac{1}2}}.  \] 
This means that
\[\int_{\Omega} \int_{  B_1(y)} \frac{\chi_{E_0}(x)}{|x-y|^{n+s}}\, dx\,dy\leq \int_{\Omega} \int_{\Co \Omega} \frac{dx}{|x-y|^{n+\frac{1}2}}=P_{\frac{1}2}(\Omega)=c<\infty,\]
since $\Omega$ has a finite classical perimeter.
In this way, it follows that
\eqlab{\label{mumu1} sP_s(E_0,\Omega) = \int_{\Omega} \int_{\Rn}\frac{\chi_{E_0}(x) }{|x-y|^{n+s}} \, dx \, dy\leq  &\;s (1-\eps)^{-n-s} |\Omega|\int_{\Co B_{R/\eps} } \frac{\chi_{E_0}(x)}{|x|^{n+s}}\, dx\\
&\; +\omega_n\Big(1- \Big(\frac{R}{\eps}+R\Big)^{-s}\Big)|\Omega| +sc. }
Furthermore, notice that if $x\in B_{R/\eps}$ we have that $|x-y|\leq (1+\eps)|x|$, hence
\[ \int_{\Rn}\frac{\chi_{E_0}(x)}{|x-y|^{n+s}}\, dx\geq \int_{\Co B_{{R}/{\eps}}} \frac{\chi_{E_0}(x)}{|x-y|^{n+s}} \, dx \geq (1+\eps)^{-n-s} \int_{\Co B_{{R}/{\eps}}} \frac{\chi_{E_0}(x)}{|x|^{n+s}} \, dx.  \]
Thus for any $\eps>0$
\[ sP_s(E_0,\Omega) \geq s |\Omega|  (1+\eps)^{-n-s}  \int_{\Co B_{{R}/{\eps}}} \frac{\chi_{E_0}(x)}{|x|^{n+s}} \, dx .\] 
Passing to  limsup as $s\to 0^+$ here above and in \eqref{mumu1} it follows that
\[ (1+\eps)^{-n} \overline \alpha(E_0)\,  |\Omega| \leq \overline \mu(E_0) \leq (1-\eps)^{-n} \overline \alpha(E_0)\, |\Omega| .\]
Sending $\eps \to 0$, we obtain the desired conclusion.
\end{proof}

\section{Classification of nonlocal minimal surfaces for small $s$}\label{classify}

\subsection{Asymptotic estimates of the density (Theorem \ref{notfull})}\label{sectnotfull}

 
The importance of Theorem \ref{notfull} is threefold:
\begin{itemize}
\item first of all, it is an interesting result in itself, by stating (in the usual hypothesis in which the contribution from infinity of the exterior data $E_0$ is less than that of a half-space) that any ball of fixed radius, centered at some $x\in\overline\Omega$, contains at least a portion of the complement of an $s$-minimal set $E$, when $s$ is small enough. We further observe that Theorem \ref{notfull} actually provides a ``uniform'' measure theoretic estimate of how big this portion is, purely in terms of the fixed datum
$\overline{\alpha}(E_0)$.

\item Moreover, we point out that Definition \ref{wild} does not exlude apriori ``full'' sets, i.e. sets $E$ such that $E\cap\Omega=\Omega$.  Hence, in the situation of point $(A)$ of Theorem \ref{THM}, one may wonder whether an $s$-minimal set $E$, which is
$\delta_s$-dense, can actually completely cover $\Omega$.  The answer is no: Theorem \ref{notfull}
proves in particular that the contribution from infinity forces the domain $\Omega$, for $s$ small enough, to contain at least
a non-trivial portion of the complement of $E$.

\item Finally, the density estimate of Theorem \ref{notfull} serves as an auxiliary result for the proof of part (B) of our main Theorem \ref{THM}.
\end{itemize}

 \begin{proof}[Proof of Theorem \ref{notfull}]
We begin with two easy but useful preliminary remarks. We
observe that,
given a set $F\subset\R^n$ and two open sets $\Omega'\subset\Omega$, we have
\eqlab{\label{monny}
P_s(F,\Omega')\leq P_s(F,\Omega).
}
Also, we point out that,
given an open set $\mathcal O\subset\R^n$ and a set $F\subset\R^n$, then by the definition \eqref{nmspf1}
of the fractional perimeter,
it holds
\eqlab{\label{Yamamoto}
F\cap\Omega=\emptyset\quad\implies\quad
P_s(F,\mathcal O)=\int_F\int_{\mathcal O}\frac{dx\,dy}{|x-y|^{n+s}}.
}

With these observations at hand, we are ready to proceed with the proof of the Theorem.
We argue by contradiction.

Suppose that there exists $\delta>0$ and $\gamma\in(0,1)$ for which we can find a sequence $s_k\searrow0$,
a sequence of sets $\{E_k\}$ such that each $E_k$ is $s_k$-minimal in $\Omega$ with exterior data $E_0$, and a sequence
of points $\{x_k\}\subset\overline{\Omega}$ such that
\eqlab{\label{density_contrad_proof}
\big|(\Omega\cap B_\delta(x_k))\setminus E_k\big|< \gamma\,\frac{\omega_n-2\overline{\alpha}(E_0)}{\omega_n-\overline{\alpha}(E_0)}\big|\Omega\cap B_\delta(x_k)\big|.
}
As a first step, we are going to exploit \eqref{density_contrad_proof} in order to obtain a
bound from below for the limit as $k\to\infty$ of
$s_k P_{s_k}(E_k,\Omega\cap B_\delta(x_k))$ (see the forthcoming inequality \eqref{dolph1}).

First of all we remark that, since $\overline{\Omega}$ is compact, up to passing to subsequences we can suppose that $x_k\longrightarrow x_0$, for some $x_0\in\overline{\Omega}$.
Now we observe that from \eqref{density_contrad_proof} it follows that 
\bgs{
|E_k\cap(\Omega\cap B_\delta(x_k))\big|=|\Omega\cap B_\delta(x_k)|-\big|(\Omega\cap B_\delta(x_k))\setminus E_k\big|
>\frac{(1-\gamma)\omega_n-(1-2\gamma)\overline{\alpha}(E_0)}{\omega_n-\overline{\alpha}(E_0)}\,|\Omega\cap B_\delta(x_k)|,
}
and hence, since $x_k\longrightarrow x_0$,
\eqlab{\label{Esti_pf_eqn1}
\liminf_{k\to\infty}|E_k\cap(\Omega\cap B_\delta(x_k))\big|
\geq\frac{(1-\gamma)\omega_n-(1-2\gamma)\overline{\alpha}(E_0)}{\omega_n-\overline{\alpha}(E_0)}\,|\Omega\cap B_\delta(x_0)|.
}
Notice that, since $\Omega$ is bounded, we can find $R>0$ such that $\Omega\subset\subset B_R(q)$ for every $q\in\overline{\Omega}$.
Then
we obtain that
\bgs{
P_{s_k}(E_k,\Omega\cap B_\delta(x_k))&\ge
\int_{E_k\cap(\Omega\cap B_\delta(x_k))}
\Big(\int_{\Co E_k\setminus(\Omega\cap B_\delta(x_k))}\frac{dz}{|y-z|^{n+s_k}}\Big)dy\\
&
\ge\int_{E_k\cap(\Omega\cap B_\delta(x_k))}\Big(\int_{\Co \Omega}\frac{\chi_{\Co E_0}(z)}{|y-z|^{n+s_k}}\,dz\Big)dy\\
&\ge \int_{E_k\cap(\Omega\cap B_\delta(x_k))}\Big(\inf_{q\in \overline \Omega } \int_{\Co \Omega}\frac{\chi_{\Co E_0}(z)}{|q-z|^{n+s_k}}\,dz\Big)dy\\
 &
\ge\big|E_k\cap(\Omega\cap B_\delta(x_k))\big|\inf_{q\in\overline{\Omega}}\int_{\Co B_R(q)}
\frac{\chi_{\Co E_0}(z)}{|q-z|^{n+s_k}}\,dz.
}
So, thanks to Proposition \ref{unifrq} and recalling \eqref{Esti_pf_eqn1}, we find
\eqlab{\label{dolph1}
\liminf_{k\to\infty}s_k P_{s_k}&(E_k,\Omega\cap B_\delta(x_k))\\
&
\ge\Big(\liminf_{k\to\infty}|E_k\cap(\Omega\cap B_\delta(x_k))\big|\Big)\Big(\liminf_{k\to\infty}s_k\,
\inf_{q\in\overline{\Omega}}\int_{\Co B_R(q)}\frac{\chi_{\Co E_0}(z)}{|q-z|^{n+s_k}}\,dz\Big)\\
&
=\big(\omega_n-\overline{\alpha}(E_0)\big)\Big(\liminf_{k\to\infty}|E_k\cap(\Omega\cap B_\delta(x_k))\big|\Big)\\
&
\geq\big(\omega_n-\overline{\alpha}(E_0)\big)
\frac{(1-\gamma)\omega_n-(1-2\gamma)\overline{\alpha}(E_0)}{\omega_n-\overline{\alpha}(E_0)}\,|\Omega\cap B_\delta(x_0)|.
}

On the other hand, as a second step we claim that
\eqlab{\label{contrad_dens_limsup}
\limsup_{k\to\infty}s_k P_{s_k}(E_k,\Omega\cap B_\delta(x_k))\le\overline{\alpha}(E_0)\big|\Omega\cap B_\delta(x_0)\big|.
}
We point out that obtaining the inequality \eqref{contrad_dens_limsup} is a crucial step of the proof. Indeed,
exploiting both \eqref{contrad_dens_limsup} and \eqref{dolph1}, we obtain
\eqlab{\label{contrad_dens_liminf}
\overline{\alpha}(E_0)\,|\Omega\cap B_\delta(x_0)|\ge\liminf_{k\to\infty}s_k P_{s_k}&(E_k,\Omega\cap B_\delta(x_k))\ge
\big((1-\gamma)\omega_n-(1-2\gamma)\overline{\alpha}(E_0)\big)|\Omega\cap B_\delta(x_0)|.}
Then, since $x_0\in\overline{\Omega}$ implies that
\[
|\Omega\cap B_\delta(x_0)|>0,
\]
by \eqref{contrad_dens_liminf} we get
\[
\overline{\alpha}(E_0)\ge(1-\gamma)\omega_n-(1-2\gamma)\overline{\alpha}(E_0)
\quad\textrm{ that is }\quad
(1-\gamma)\overline{\alpha}(E_0)\ge(1-\gamma)\frac{\omega_n}{2}.
\]
Therefore, since $\gamma\in(0,1)$ and by hypothesis $\overline{\alpha}(E_0)<\frac{\omega_n}{2}$, we reach a contradiction, concluding the proof.

We are left to prove \eqref{dolph1}. For this, we exploit the minimality of the sets $E_k$ in order to compare the $s_k$-perimeter
of $E_k$ with the $s_k$-perimeter of appropriate competitors $F_k$.

We first remark that, since $x_k\longrightarrow x_0$, for every $\eps>0$
there exists $\tilde k_\eps$ such that
\eqlab{\label{subset_balls_eps}
\Omega\cap B_\delta (x_k)\subset\Omega\cap B_{\delta+\eps}(x_0),\qquad
\forall\,k\ge\tilde k_\eps.}
We fix a small $\eps>0$. We will let $\eps\to0$ later on.

We also observe that, since $E_k$ is $s_k$-minimal in $\Omega$, it is $s_k$-minimal also in every
$\Omega'\subset\Omega$, hence in particular in $\Omega \cap B_{\delta+\eps}(x_0)$.
Now we proceed to define the sets
\eqlab{\label{density_competitor_def}
F_k:=E_0\cup(E_k\cap(\Omega\setminus B_{\delta+\eps}(x_0)))
=E_k\setminus \big(\Omega\cap B_{\delta+\eps}(x_0)\big).
}
Then, by \eqref{monny}, \eqref{subset_balls_eps}, \eqref{density_competitor_def} and by the minimality of $E_k$
in $\Omega \cap B_{\delta+\eps}(x_0)$, for every $k\ge\tilde k_\eps$ we find that
\bgs{
P_{s_k}(E_k,\Omega\cap B_\delta(x_k))\le P_{s_k}(E_k,\Omega\cap B_{\delta+\eps}(x_0))
\le P_{s_k}(F_k,\Omega\cap B_{\delta+\eps}(x_0)).
}
We observe that by the definition \eqref{density_competitor_def} we have that
\[
F_k\cap\big(\Omega\cap B_{\delta+\eps}(x_0)\big)=\emptyset.
\] 
Therefore, recalling \eqref{Yamamoto}
and the definition \eqref{density_competitor_def} of the sets $F_k$, we obtain that
\bgs{
P_{s_k}(F_k,\Omega\cap B_{\delta+\eps}(x_0))&
=\int_{E_0\cup(E_k\cap(\Omega\setminus B_{\delta+\eps}(x_0)))}\int_{\Omega\cap B_{\delta+\eps}(x_0)}
\frac{dy\,dz}{|y-z|^{n+s_k}}\\
&
=\int_{E_0}\int_{\Omega\cap B_{\delta+\eps}(x_0)}\frac{dy\,dz}{|y-z|^{n+s_k}}
+
\int_{E_k\cap(\Omega\setminus B_{\delta+\eps}(x_0))}\int_{\Omega\cap B_{\delta+\eps}(x_0)}\frac{dy\,dz}{|y-z|^{n+s_k}}\\
&
\le\int_{E_0}\int_{\Omega\cap B_{\delta+\eps}(x_0)}\frac{dy\,dz}{|y-z|^{n+s_k}}
+\int_{\Omega\setminus B_{\delta+\eps}(x_0)}\int_{\Omega\cap B_{\delta+\eps}(x_0)}\frac{dy\,dz}{|y-z|^{n+s_k}}\\
&
=:I^1_k+I^2_k.
}
Furthermore, again by \eqref{Yamamoto}, we have that 
\eqlab{\label{dens_asympt_pf}
I^1_k=P_{s_k}(E_0,\Omega\cap B_{\delta+\eps}(x_0))
\quad\mbox{and}\quad
I^2_k=P_{s_k}(\Omega\setminus B_{\delta+\eps}(x_0),\Omega\cap B_{\delta+\eps}(x_0)).
}
We observe that the open set $\Omega\cap B_{\delta+\eps}(x_0)$ has finite classical perimeter.
Thus, we can exploit the equalities \eqref{dens_asympt_pf} and apply Proposition \ref{barmubaral} twice, obtaining
\[
\limsup_{k\to\infty}s_k I^1_k\le
\overline{\alpha}(E_0)\big|\Omega\cap B_{\delta+\eps}(x_0)\big|,
\]
and
\eqlab{\label{dens_asympt_pf1}
\limsup_{k\to\infty}s_k I^2_k\le\overline{\alpha}(\Omega\setminus B_{\delta+\eps}(x_0))\big|\Omega\cap B_{\delta+\eps}(x_0)\big|,
}
for every $\eps>0$. Also notice that, since $\Omega$ is bounded, by Remark \ref{finmeas} we have
\[
\overline{\alpha}(\Omega\setminus B_{\delta+\eps}(x_0))=\alpha(\Omega\setminus B_{\delta+\eps}(x_0))=0,
\]
and hence, by \eqref{dens_asympt_pf1},
\[
\lim_{k\to\infty}s_kI^2_k=0.
\]
Therefore, combining these computations we find that
\bgs{
\limsup_{k\to\infty}s_k P_{s_k}(E_k,\Omega\cap B_\delta(x_k))\le
\limsup_{k\to\infty}s_k I^1_k\le
\overline{\alpha}(E_0)\big|\Omega\cap B_{\delta+\eps}(x_0)\big|,
}
for every $\eps>0$ small.
To conclude, we let $\eps\to0$ and we obtain \eqref{contrad_dens_limsup}.
\end{proof}

It is interesting to observe that, as a straightforward consequence of Theorem \ref{notfull}, when $\alpha(E_0)=0$ we
know that any sequence of $s$-minimal sets is asymptotically empty inside $\Omega$, as $s\to0^+$. More precisely 
\begin{corollary}
Let $\Omega\subset\R^n$ be a bounded open set of finite classical perimeter and let $E_0\subset\Co\Omega$ be such that
$\alpha(E_0)=0$. Let $s_k\in(0,1)$ be such that $s_k\searrow0$ and let $\{E_k\}$
be a sequence of sets such that each $E_k$ is $s_k$-minimal in $\Omega$ with exterior data $E_0$. Then
\[\lim_{k\to\infty}|E_k\cap\Omega|=0.\]
\end{corollary}

\begin{proof}
Fix $\delta>0$. Since $\overline{\Omega}$ is compact, we can find a finite number of points $x_1,\dots,x_m\in\overline{\Omega}$
such that
\[\overline{\Omega}\subset\bigcup_{i=1}^mB_\delta(x_i).\]
By Theorem \ref{notfull} (by using the fact that $\alpha(E_0)=0$) we know that for every $\gamma\in(0,1)$ we can find a $k(\gamma)$ big enough such that
\bgs{
\big|(\Omega\cap B_\delta(x_i))\setminus E _k\big|\ge\gamma\, \big|\Omega\cap B_\delta(x_i)\big|.
}
Then, 
\bgs{
		\big|E_k\cap(\Omega\cap B_\delta(x_i))\big|= \big|\Omega\cap B_\delta(x_i)\big|- 	\big|(\Omega\cap B_\delta(x_i))\setminus E_k\big|
				\le 
(1-\gamma)|\Omega\cap B_\delta(x_i)|,
}
for every $i=1,\dots,m$ and every $k\ge k(\gamma)$. Thus
\[|E_k\cap\Omega|\le(1-\gamma)\sum_{i=1}^m|\Omega\cap B_\delta(x_i)|,\]
for every $k\ge k(\gamma)$, and hence
\[\limsup_{k\to\infty}|E_k\cap\Omega|\le(1-\gamma)\sum_{i=1}^m|\Omega\cap B_\delta(x_i)|,\]
for every $\gamma\in(0,1)$. Letting $\gamma\longrightarrow1^-$ concludes the proof.
\end{proof}

We recall here that any set $E_0$ of finite measure has $\alpha(E_0)=0$ (check Remark \ref{finmeas}).

\subsection{Estimating the fractional mean curvature (Theorem \ref{positivecurvature})}\label{estimatecurvature}

Thanks to the previous preliminary work, we are now in the position
of completing the proof of Theorem \ref{positivecurvature}.

\begin{proof}[Proof of Theorem \ref{positivecurvature}]
Let $R:=2\,\max\{1,\textrm{diam}(\Omega)\}$. First of all, \eqref{claimalpha} implies that
\bgs{ \liminf_{s\to 0^+} \bigg(\omega_n R^{-s} -2s\sup_{q\in \overline \Omega} \int_{\Co B_R(q)}\frac{\chi_E(y)}{|q-y|^{n+s}} \, dy \bigg) = {\omega_n-2\overline \alpha(E_0)}=4 \beta.} Notice that by \eqref{weak_hp_beta}, $\beta>0$. Hence for every $s$ small enough, say $s<s'\leq\frac{1}{2}$ with $s'=s'(E_0,\Omega)$, we have that
\eqlab{\label{tildebeta}\omega_n R^{-s} - 2s\sup_{q\in \overline \Omega} \int_{\Co B_R(q)}\frac{\chi_E(y)}{|q-y|^{n+s}} \, dy \geq \frac {7}2 \beta.}

Now, let $E\subset\Rn$ be such that $E\setminus\Omega=E_0$, suppose that $E$ has an exterior tangent ball
of radius $\delta<R/2$ at $q\in\partial E\cap\overline{\Omega}$, that is
\[B_\delta(p)\subset\Co E\quad\textrm{and}\quad q\in\partial B_\delta(p),\]
and let $s<s'$.
Then for $\rho$ small enough (say $\rho<\delta/2$) we conclude that
\bgs{\label{uno} \mathcal{I}^\rho_s[E](q)= \int_{B_R(q)\setminus B_{\rho}(q)} \frac{\chi_{ \Co E} (y)-\chi_{E}(y)}{|q-y|^{n+s}}\, dy + \int_{\Co B_R(q)} \frac{\chi_{ \Co E} (y)-\chi_{E}(y)}{|q-y|^{n+s}}\, dy.}

 Let $D_\delta= B_\delta(p) \cap B_\delta(p')$, where $p'$ is the symmetric point of $p$ with respect to $q$, i.e. the ball $B_\delta(p')$ is the ball tangent to  $B_\delta(p)$ in $q$. Let also $K_\delta$ be the convex hull of $D_\delta$ and let $P_\delta:=K_\delta-D_\delta$. Notice that $B_\rho(q)\subset K_\delta \subset B_R(q)$ . Then
\bgs{ \int_{B_R(q)\setminus B_{\rho}(q)} \frac{\chi_{ \Co E} (y)-\chi_{E}(y)}{|q-y|^{n+s}}\, dy=& \int_{D_\delta\setminus B_{\rho}(q)} \frac{\chi_{ \Co E} (y)-\chi_{E}(y)}{|q-y|^{n+s}}\, dy +\int_{P_\delta\setminus B_{\rho}(q) } \frac{\chi_{ \Co E} (y)-\chi_{E}(y)}{|q-y|^{n+s}}\, dy\\ &+\int_{B_R(q) \setminus K_\delta} \frac{\chi_{ \Co E} (y)-\chi_{E}(y)}{|q-y|^{n+s}}\, dy.}
Since $B_\delta(p) \subset \Co E$, by symmetry we obtain that
\bgs{  \int_{D_\delta \setminus B_{\rho}(q)} \frac{\chi_{ \Co E} (y)-\chi_{E}(y)}{|q-y|^{n+s}}\, dy =\int_{B_\delta(p)\setminus B_{\rho}(q)} \frac{dy}{|q-y|^{n+s}} +   \int_{B_\delta(p')\setminus B_{\rho}(q) }\frac{\chi_{ \Co E} (y)-\chi_{E}(y)}{|q-y|^{n+s}}\, dy  \geq 0.   }
Moreover, from Lemma 3.1 in \cite{Graph} (here applied with $\lambda =1$) we have that
\[\bigg| \int_{P_\delta \setminus B_{\rho}(q)} \frac{\chi_{ \Co E} (y)-\chi_{E}(y)}{|q-y|^{n+s}}\, dy \bigg|\leq  \int_{P_\delta} \frac{dy}{|q-y|^{n+s}}\leq  \frac{C_0}{1-s} {\delta^{-s} }
 ,\] 
with $C_0=C_0(n)>0$.
Notice that $B_\delta(q)\subset K_\delta$  so
\bgs{\bigg | \int_{B_R(q)\setminus K_\delta } \frac{\chi_{ \Co E} (y)-\chi_{E}(y)}{|q-y|^{n+s}}\, dy\bigg | \leq &\;\int_{B_R(q)\setminus B_\delta(q)} \frac{dy}{|q-y|^{n+s}} = \omega_n  \frac{\delta^{-s}- R^{-s}}{s} .} 
Therefore for every $\rho <\delta/2$
one has that
\bgs{\label{due}\int_{B_R(q)\setminus B_\rho(q)} \frac{\chi_{ \Co E} (y)-\chi_{E}(y)}{|q-y|^{n+s}}\, dy \geq -\frac{C_0}{1-s}\delta^{-s} - \frac{\omega_n}{s} \delta^{-s} + \frac{\omega_n}{s} R^{-s}  .}
Thus, using \eqref{tildebeta}
\eqlab{\label{uno} \mathcal{I}^\rho_s[E](q)=&\;\int_{B_R(q)\setminus B_\rho(q)} \frac{\chi_{ \Co E} (y)-\chi_{E}(y)}{|q-y|^{n+s}}\, dy + \int_{\Co B_R(q)} \frac{\chi_{ \Co E} (y)-\chi_{E}(y)}{|q-y|^{n+s}}\, dy \\
\geq &\; -\frac{C_0}{1-s}\delta^{-s} - \frac{\omega_n}{s} \delta^{-s} + \frac{\omega_n}{s} R^{-s} + \int_{\Co B_R(q)} \frac{dy}{|q-y|^{n+s}} -  2\int_{ \Co B_R(q)} \frac{\chi_E(y)}{|q-y|^{n+s}}\,dy \\ 
\geq&\; - \delta^{-s} \Big(\frac{C_0}{1-s} +\frac{\omega_n} s\Big) +  \frac{\omega_n}s R^{-s} + \bigg(\frac{\omega_n}s R^{-s}- 2\sup_{q\in \overline \Omega} 
 \int_{ \Co B_R(q)} \frac{\chi_E(y)}{|q-y|^{n+s}}\,dy\bigg)\\
 \geq&\; - \delta^{-s} \Big(\frac{C_0}{1-s} +\frac{\omega_n} s\Big) +  \frac{\omega_n}s R^{-s}  +\frac{7\beta}{2s}\\
 \geq &\; - \delta^{-s} \Big(2C_0 +\frac{\omega_n} s\Big) +  \frac{\omega_n}s R^{-s}  +\frac{7\beta}{2s},}
where we also exploited that $s<s'\leq 1/2$.
Since $R>1$, we have
\[R^{-s}\to1^-,\qquad\textrm{as }s\to0^+.\]
Therefore we can find $s''=s''(E_0,\Omega)$ small enough such that
\[\omega_nR^{-s}\geq\omega_n-\frac{\beta}{2},\qquad\forall s<s''.\]

Now let
\[s_0=s_0(E_0,\Omega):=\min\Big\{s',s'',\frac{\beta}{2C_0}\Big\}.\]
Then, for every $s<s_0$ we have
\eqlab{\label{important_estimate_curv}\I_s^\rho[E](q)&\geq\frac{1}{s}\Big\{-\delta^{-s}\big((2C_0)s+\omega_n\big)+\omega_n R^{-s}+\frac{7}{2}\beta\Big\}\\
&
\geq\frac{1}{s}\big\{-\delta^{-s}(\omega_n+\beta)+\omega_n+3\beta\big\},}
for every $\rho\in(0,\delta/2)$.

Notice that if we fix $s\in(0,s_0)$, then for every
\[\delta\geq e^{-\frac{1}{s}\log\frac{\omega_n+2\beta}{\omega_n+\beta}}=:\delta_s(E_0),\]
we have that
\[-\delta^{-s}(\omega_n+\beta)+\omega_n+3\beta\geq\beta>0.\]
To conclude, we
let $\sigma\in(0,s_0)$ and suppose that $E$ has an exterior tangent ball of radius $\delta_\sigma$
at $q\in\partial E\cap\overline{\Omega}$.
Notice that, since $\delta_\sigma<1$, we have
\[-(\delta_\sigma)^{-s}(\omega_n+\beta)+\omega_n+3\beta\geq
-(\delta_\sigma)^{-\sigma}(\omega_n+\beta)+\omega_n+3\beta=\beta,\qquad\forall\,s\in(0,\sigma].\]
Then \eqref{important_estimate_curv} gives that
\[\liminf_{\rho\to0^+}\I_s^\rho[E](q)\geq\frac{\beta}{s}>0,\qquad\forall\,s\in(0,\sigma],\]
which concludes the proof.
\end{proof}

\begin{remark}
We remark that
\[\log\frac{\omega_n+2\beta}{\omega_n+\beta}>0,\]
thus
\[\delta_s\to0^+\qquad\textrm{as }s\to 0^+.\]
\end{remark}


As a consequence of Theorem \ref{positivecurvature}, we have that, as $s\to0^+$,
the $s$-minimal sets with small mass at infinity have small mass in $\Omega$.
The precise result goes as follows:

\begin{corollary}
Let $\Omega\subset\Rn$ be a bounded open set,
let $E\subset\Rn$ be such that
\[\overline \alpha(E) <\frac{\omega_n}2,\]
and suppose that $\partial E$ is of class $C^2$ in $\Omega$.
Then, for every $\Omega'\subset\subset\Omega$ there exists
$\tilde s=\tilde s(E\cap\overline{\Omega'})\in(0,s_0)$ such that for every $s\in(0,\tilde s]$
\eqlab{\I_s[E](q)\geq\frac{\omega_n -2\overline \alpha(E)}{4s}>0,\qquad\forall\,q\in\partial E\cap\overline{\Omega'}.}
\end{corollary}

\begin{proof}
Since $\partial E$ is of class $C^2$ in $\Omega$ and $\Omega'\subset\subset\Omega$, the set $E$ satisfies a uniform exterior
ball condition of radius $\tilde\delta=\tilde\delta(E\cap\overline{\Omega'})$ in $\overline{\Omega'}$, meaning that
$E$ has an exterior tangent ball of radius at least $\tilde\delta$ at every point $q\in\partial E\cap\overline{\Omega'}$.

Now, since $\delta_s\to0^+$ as $s\to0^+$, we can find $\tilde s=\tilde s(E\cap\overline{\Omega'})
<s_0(E\setminus\Omega,\Omega)$,
small enough such that $\delta_s<\tilde\delta$ for every $s\in(0,\tilde s]$. Then we can conclude by applying Theorem \ref{positivecurvature}.
\end{proof}

\subsection{Classification of $s$-minimal surfaces (Theorem \ref{THM})}\label{alternative}

To classify the behavior of the $s$-minimal surfaces when $s$ is small,
we need to take into account the ``worst case scenario'',
that is the one in which the set behaves very badly in terms
of oscillations and lack of regularity. To this aim,
we make an observation about $\delta$-dense sets. See Figure \ref{fig:x3}.

   \begin{center}
\begin{figure}[htpb]
	\hspace{0.79cm}
	\begin{minipage}[b]{0.79\linewidth}
	\centering
	\includegraphics[width=0.79\textwidth]{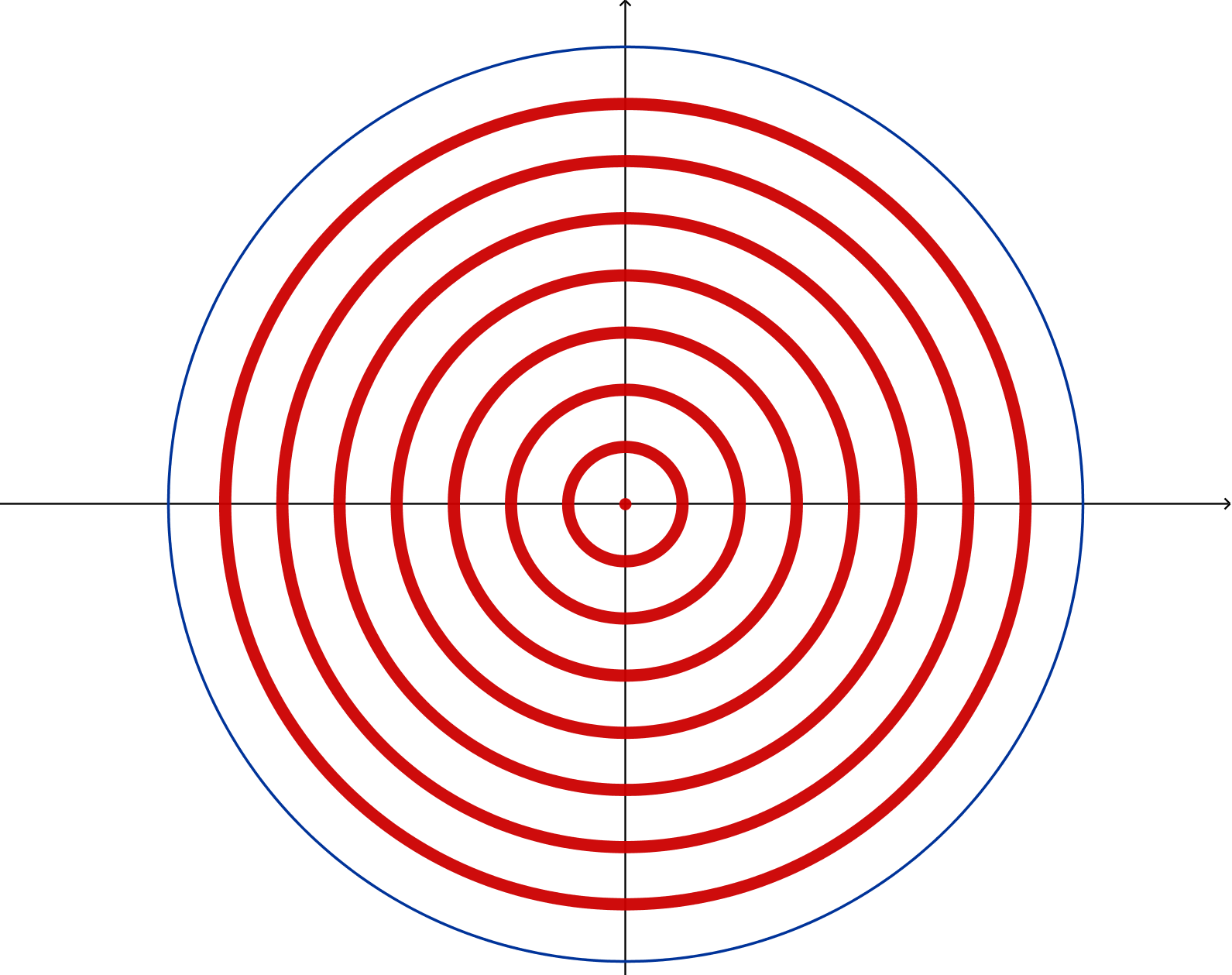}
	\caption{A $\delta$-dense set of measure $<\eps$}   
	\label{fig:x3}
	\end{minipage}
\end{figure} 
	\end{center}
  \begin{remark}\label{deltadance}
  For every $k\ge1$ and every $\eps<2^{-k-1}$, we define the sets
\[\Gamma_k^\eps:=B_\eps\cup\bigcup_{i=1}^{2^k-1}\Big\{x\in\R^n\,\big|\,\frac{i}{2^k}-\eps<|x|<\frac{i}{2^k}+\eps\Big\}
\quad\textrm{ and }\quad\Gamma_k:=\{0\}\cup\bigcup_{i=1}^{2^k-1}\partial B_\frac{i}{2^k}.\] 
%
\noindent Notice that for every $\delta>0$ there exists $\tilde k=\tilde k(\delta)$ such that for every $k\geq\tilde k$ we have
\[B_\delta(x)\cap\Gamma_k\not=\emptyset,\qquad\forall\,B_\delta(x)\subset B_1.\]
Thus, for every $k\geq\tilde k(\delta)$ and $\eps<2^{-k-1}$, the set $\Gamma_k^\eps$ is $\delta$-dense in $B_1$.
\noindent Moreover, notice that
\[\Gamma_k=\bigcap_{\eps\in(0,2^{-k-1})}\Gamma_k^\eps\quad\textrm{ and }\quad\lim_{\eps\to0^+}|\Gamma_k^\eps|=0.\]
It is also worth remarking that the sets $\Gamma_k^\eps$ have smooth boundary.
\noindent In particular, for every $\delta>0$ and every $\eps>0$ small, we can find a set $E\subset B_1$ which is $\delta$-dense in $B_1$ and whose measure is $|E|<\eps$.
This means that we can find an open set $E$ with smooth boundary, whose measure is arbitrarily small
and which is ``topologically arbitrarily dense'' in $B_1$.
\end{remark}

We introduce the following useful geometric observation. 

\begin{prop}\label{tgball}
Let $\Omega\subset \Rn$ be a bounded and connected open set with $C^2$ boundary and let $\delta\in(0,r_0)$,
for $r_0$ given in \eqref{r01}. 
If $E$ is not $\delta$-{dense} in $\Omega$ 
and $|E\cap\Omega|>0$, 
then there exists a point $q\in\partial E\cap\Omega$ such that $E$ has an exterior tangent ball at $q$ of radius $\delta$ (contained in $\Omega$),
i.e.
there exist $p\in \Co E\cap \Omega$ such that
\[ B_\delta(p)\subset\subset\Omega,\qquad q \in \partial B_{\delta}(p) \cap \partial E
\quad \mbox{ and }  \quad B_{ \delta}(p) \subset \Co E.\] 
\end{prop}
\begin{proof}
Using Definition \ref{wild}, we have that there exists  $x\in \Omega$ for which $B_\delta(x)\subset\subset \Omega$ and $|B_\delta(x)\cap E|=0$, so $B_\delta(x)\subset E_{ext}$. If $B_\delta(x)$ is tangent to $\partial E$ then we are done.

Notice that
\[B_\delta(x)\subset\subset\Omega\quad\Longrightarrow\quad d(x,\partial\Omega)>\delta,\]
and let
\[\delta':=\min\{r_0,d(x,\partial\Omega)\}\in(\delta,r_0].\]
Now we consider the open set $\Omega_{-\delta'}\subset \Omega$
\[ \Omega_{-\delta'} :=\{ \bar d_{\Omega}<-\delta'\},\]  so $x\in \Omega_{-\delta'}$. According to Remark \ref{c21} and Lemma \ref{geomlem} we have that $\Omega_{-\delta'}$ has $C^2$ boundary and that 
\eqlab{\label{unifball} \Omega_{-\delta'} \mbox{ satisfies the uniform interior ball condition of radius at least } r_0.} 
 
 We have two possibilities:
 \eqlab{\label{inside} &\mbox{ i) } && \overline  E\cap \Omega_{-\delta'}\neq \emptyset \\ 
 				  &\mbox{ ii) }&&\emptyset \neq  \overline{E}\cap\Omega \subset \Omega \setminus \Omega_{-\delta'}.} 
 
 If i) happens, we pick any point $y\in \overline{E}\cap \Omega_{-\delta'}$.  
The set $\overline{\Omega_{-\delta'}}$ is path connected (see Proposition \ref{retract}), so there exists a path $c:[0,1]\longrightarrow
\R^n$ that connects $x$ to $y$ and that stays inside $\overline{\Omega_{-\delta'}}$, that is
\[c(0)=x,\qquad c(1)=y\quad\textrm{ and }\quad c(t)\in\overline{\Omega_{-\delta'}},\quad\forall\,t\in[0,1].\]
Moreover, since $\delta<\delta'$, we have
\[B_\delta\big(c(t)\big)\subset\subset\Omega\qquad\forall\,t\in[0,1].\]
Hence,
we can ``slide the ball'' $B_{\delta}(x)$ along the path and we obtain the desired
claim thanks to Lemma \ref{slidetheballs}.

%

Now, if we are in the case ii) of \eqref{inside}, then $\Omega_{-\delta'}\subset E_{ext}$, so we dilate $\Omega_{-\delta'}$ until we first touch 
$\overline E$. That is, we consider 
\[\tilde \rho:=\inf\{\rho\in[0,\delta'] \; \big| \; \Omega_{-\rho}\subset E_{ext}\}.\]
Notice that by hypothesis $\tilde \rho>0$. Then 
\[\overline{\Omega_{-\tilde \rho}} \subset \overline {E_{ext}}=E_{ext}\cup\partial E. \]
If
\[ \partial \Omega_{-\tilde\rho} \cap \partial E=\emptyset\quad  \mbox { then }\quad  \overline{ \Omega_{-\tilde \rho}} \subset  E_{ext},\] 
hence we have that
\[ d= d\left(  \overline E\cap \Omega\setminus \Omega_{-\delta'}, \overline {\Omega_{-\tilde \rho}}\right)\in(0,\tilde \rho),\]
therefore
\[ \Omega_{-\tilde \rho} \subset \Omega_{-(\tilde \rho-d)}\subset E_{ext}.\]
This is in contradiction with the definition of $\tilde \rho$. Hence, 
 there exists $q \in \partial \Omega_{-\tilde \rho} \cap \partial E$. 

Recall that, by definition of $\tilde\rho$, we have
$\Omega_{-\tilde \rho} \subset \Co E.$ Thanks to \eqref{unifball}, there exists a tangent ball at $q$ interior to $\Omega_{-\tilde \rho}$, hence a tangent ball at $q$ exterior to $E$, of radius at least $r_0>\delta$. 
This concludes the proof of the lemma.
\end{proof}

We observe that part $(A)$ of Theorem \ref{THM}
is essentially a consequence of Theorem \ref{positivecurvature}.
Indeed, if an $s$-minimal set $E$ is not $\delta_s$-dense and it is not empty in $\Omega$, then
by Proposition \ref{tgball} we can find a point $q\in\partial E\cap\Omega$ at which $E$ has an exterior tangent ball of radius $\delta_s$. Then Theorem \ref{positivecurvature} implies that the $s$-fractional mean curvature of $E$ in $q$ is strictly positive,
contradicting the Euler-Lagrange equation.

On the other hand, part $(B)$ of Theorem \ref{THM} follows from a careful asymptotic use
of the density estimates provided by Theorem \ref{notfull}.
For the reader's facility, we also recall that $r_0$ 
has the same meaning here and across the paper, 
as clarified in Appendix \ref{A2}.
We now proceed with the precise arguments of the proof.

\begin{proof}[Proof of Theorem \ref{THM}]
We begin by proving part $(A)$.\\
First of all, since $\delta_s\to0^+$, we can find $s_1=s_1(E_0,\Omega)\in(0,s_0]$ such that $\delta_s<r_0$
for every $s\in(0,s_1)$.

Now let $s\in(0,s_1)$ and let $E$ be $s$-minimal in $\Omega$, with exterior data $E_0$.

We suppose that  $E\cap \Omega\neq \emptyset$ and prove that $E$ has to be $\delta_s$-dense. 

Suppose by contradiction that $E$ is not $\delta_s$-dense. Then, in view of 
Proposition \ref{tgball}, there exists $p\in \Co E\cap \Omega$ such that
\[ q \in \partial B_{\delta_s}(p) \cap (\partial E\cap\Omega) \quad \mbox{ and }  \quad B_{ \delta_s}(p) \subset \Co E.\] 
Hence we use the Euler-Lagrange theorem at $q$, i.e.
\[\I_s[E](q) \leq 0,\] 
to obtain a contradiction with Theorem \ref{positivecurvature}. This says that $E$ is not $\delta_s$-dense and concludes the proof of part $(A)$ of Theorem \ref{THM}.

Now we prove the part $(B)$ of the Theorem.\\
Suppose that point $(B.1)$ does not hold true. Then we can find a sequence $s_k\searrow0$ and a sequence of sets $E_k$
such that each $E_k$ is $s_k$-minimal in $\Omega$ with exterior data $E_0$ and
\[E_k\cap\Omega\not=\emptyset.\]
We can assume that $s_k<s_1(E_0,\Omega)$ for every $k$. Then part $(A)$ implies that each $E_k$ is $\delta_{s_k}$-dense,
that is
\[|E_k\cap B_{\delta_{s_k}}(x)|>0\quad\forall\,B_{\delta_{s_k}}(x)\subset\subset\Omega.\]
Fix $\gamma=\frac{1}{2}$, take a sequence $\delta_h\searrow0$ and let $\sigma_{\delta_h,\frac{1}{2}}$ be as in Theorem \ref{notfull}. Recall that $\delta_s\searrow0$ as $s\searrow0$.
Thus for every $h$ we can find $k_h$ big enough such that
\eqlab{\label{koala}s_{k_h}<\sigma_{\delta_h,\frac{1}{2}}\qquad\textrm{and}\qquad\delta_{s_{k_h}}<\delta_h.}
In particular, this implies
\eqlab{\label{koala1}|E_{k_h}\cap B_{\delta_h}(x)|\ge|E_k\cap B_{\delta_{s_{k_h}}}(x)|
>0\quad\forall\,B_{\delta_h}(x)\subset\subset\Omega,}
for every $h$. On the other hand, by \eqref{koala} and Theorem \ref{notfull}, we also have that
\eqlab{\label{koala2}|\Co E_{k_h}\cap B_{\delta_h}(x)|>0\quad\forall\,B_{\delta_h}(x)\subset\subset\Omega.}
This concludes the proof of part $(B)$. Indeed, notice that since $B_{\delta_h}(x)$ is connected, \eqref{koala1} and \eqref{koala2}
together imply that
\[\partial E_{k_h}\cap B_{\delta_h}(x)\neq\emptyset\quad\forall\,B_{\delta_h}(x)\subset\subset\Omega.\]
\end{proof}

\subsection{Stickiness to the boundary is a typical behavior (Theorem \ref{boundedset})}\label{sticky}

Now we show that the ``typical behavior''
of the nonlocal minimal surfaces
is to stick at the boundary whenever
they are allowed to do it,
in 
the precise
sense given 
by Theorem \ref{boundedset}.

\begin{proof}[Proof of Theorem \ref{boundedset}]
Let
\[\delta:=\frac{1}{2}\min\{r_0,R\},\]
and notice that (see Remark \ref{ext_unif_omega})
\[B_\delta(x_0+\delta\nu_\Omega(x_0))\subset B_R(x_0)\setminus\Omega\subset\Co E_0.\]
Since $\delta_s\to0^+$, we can find $s_3=s_3(E_0,\Omega)\in(0,s_0]$ such that $\delta_s<\delta$ for every $s\in(0,s_3)$.

Now let $s\in(0,s_3)$ and let $E$ be $s$-minimal in $\Omega$, with exterior data $E_0$.

We claim that
\eqlab{\label{pf_bdedset_eq1}B_\delta(x_0-r_0\nu_\Omega(x_0))\subset E_{ext}.}
We observe that this is indeed a crucial step to
prove Theorem \ref{boundedset}. Indeed, once this is established,
by Remark \ref{ext_unif_omega} we obtain that
\[B_\delta(x_0-r_0\nu_\Omega(x_0))\subset\subset\Omega.\]
Hence, since $\delta_s<\delta$, we deduce from \eqref{pf_bdedset_eq1}
that $E$ is not $\delta_s$-dense.
Thus, since $s<s_3\leq s_1$, Theorem \ref{THM} 
implies that $E\cap\Omega=\emptyset$, which concludes the proof of Theorem \ref{boundedset}.

This, we are left to prove \eqref{pf_bdedset_eq1}. Suppose by contradiction that
\[\overline{E}\cap B_\delta(x_0-r_0\nu_\Omega(x_0))\not=\emptyset,\]
and consider the segment $c:[0,1]\longrightarrow\Rn$,
\[c(t):=x_0+\big((1-t)\delta-t\,r_0\big)\nu_\Omega(x_0).\]
Notice that
\[B_\delta\big(c(0)\big)\subset E_{ext}\quad\mbox{ and }\quad B_\delta\big(c(1)\big)\cap\overline{E}\not=\emptyset,\]
so
\[t_0:=\sup\Big\{\tau\in[0,1]\,\big|\,\bigcup_{t\in[0,\tau]}B_\delta\big(c(t)\big)\subset E_{ext}\Big\}<1.\]
Arguing as in Lemma \ref{slidetheballs}, we conclude that
\[B_\delta\big(c(t_0)\big)\subset E_{ext}\quad\mbox{ and }\quad \exists\,q\in\partial B_\delta\big(c(t_0)\big)\cap\partial E.\]
By definition of $c$, we have that either $q\in\Omega$ or
\[q\in\partial\Omega\cap B_R(x_0).\]
In both cases (see Theorem 5.1 in \cite{nms} and Theorem \eqref{EL_boundary_coroll}) we have
\[\I_s[E](q)\leq0,\]
which gives a contradiction with Theorem \ref{positivecurvature} and concludes the proof.
\end{proof}

 
   \section{The contribution from infinity of some supergraphs} \label{sectexamples} We compute  in this Subsection the contribution from infinity of some particular supergraphs.
 \begin{example}[The cone] \label{THECONE} Let $ S \subset  \mathbb S^{n-1}$ be a portion of the unit sphere, $\mathfrak o:=\Ha^{n-1}(S)$ and
  \[ C:=\{ t\sigma  \; \big| \; t\geq 0, \;\sigma\in  S)\}.\] 
  Then the contribution from infinity is given by the opening of the cone,
  \eqlab{\label{cony} \alpha(C) = \mathfrak o.}
  Indeed,
  \[\alpha_s(0,1,C)= \int_{\Co B_1} \frac{\chi_C(y)}{|y|^{n+s}}\, dy = \Ha^{n-1}(S) \int_1^{\infty} t^{-s-1}\, dt=
  \frac{\mathfrak o}s,  \]
and we obtain the claim by passing to the limit. Notice that this says in particular that the contribution from infinity of a half-space is $\omega_n/2$.
  \end{example} 
  
  \begin{center}
\begin{figure}[htpb]
	\hspace{0.99cm}
	\begin{minipage}[b]{0.99\linewidth}
	\centering
	\includegraphics[width=0.99\textwidth]{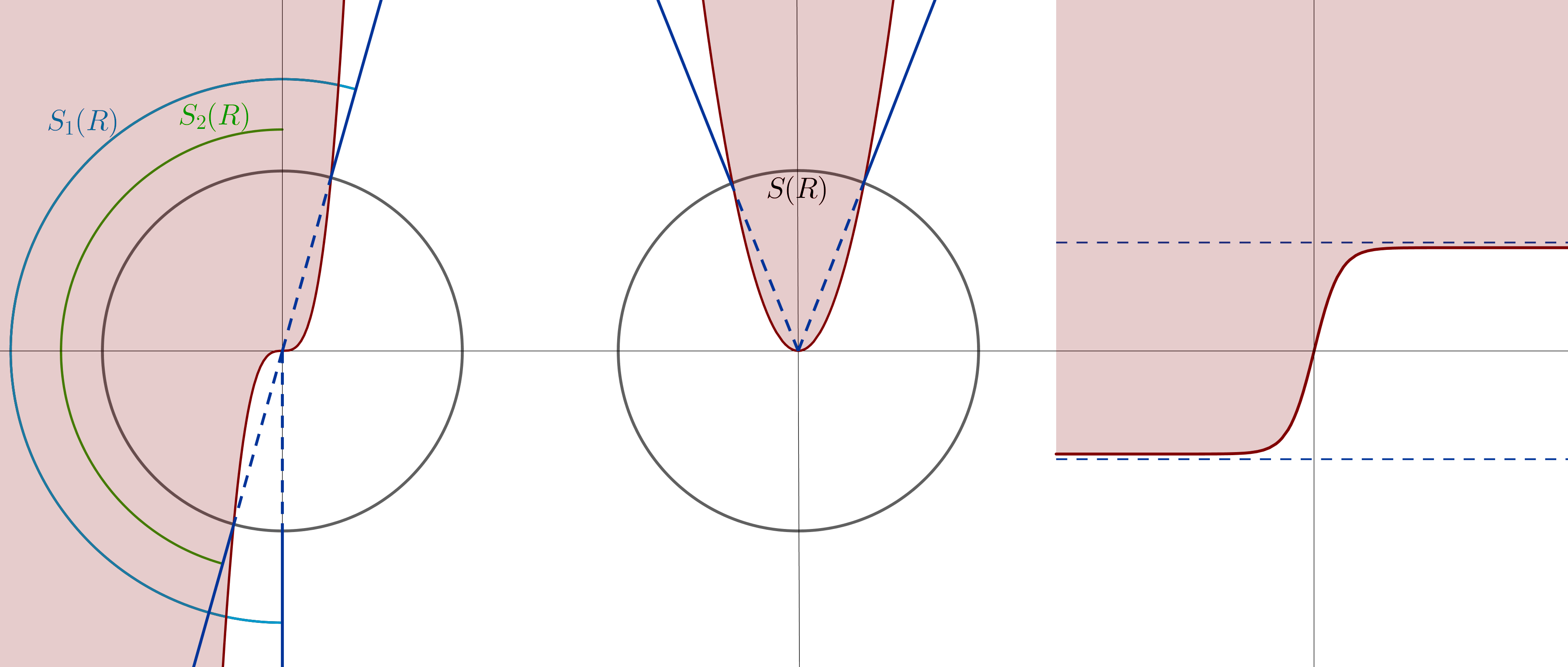}
	\caption{The contribution from infinity of $x^3$, $x^2$ and $\tanh x$}   
	\label{x3}
	\end{minipage}
\end{figure} 
	\end{center}
	  \begin{example}[The parabola]
  We consider the supergraph 
  \[ E:=\{ (x',x_n) \; \big| \; x_n\geq |x'|^2\},\] and we show that, in this case, \[ \alpha(E)=0.\]
  In order to see this, we take any $R>0$, intersect the ball $B_R$ with the parabola and build a cone on this intersection (see the second picture in Figure \ref{x3}), i.e. we define
  \[ S(R):=\partial B_R \cap E , \quad \quad C_R =\{ t\sigma  \; \big| \; t\geq 0, \;\sigma\in  S(R)\}.\] 
  We can explicitly compute the opening of this cone, that is
  \[ \mathfrak o (R)= \Bigg( \arcsin \frac{\left(\sqrt{4R^2+1}-1\right)^{1/2}}{R\, \sqrt 2 }\Bigg)\frac{\omega_n}{\pi} .\]
  Since $E\subset C_R$ outside of $B_R$, thanks to the monotonicity property in Proposition \ref{subsetssmin} and to \eqref{cony}, we have that
  \[  \overline \alpha (E) \leq \overline \alpha(C_R)  =\mathfrak o(R).\]
  Sending $R\to \infty$, we find that
\[  \overline \alpha (E) =0, \quad \mbox{ thus } \quad \alpha(E)=0.\]
  \end{example}
  More generally, if we consider for any given $c, \eps>0$ a function $u$ such that
  \[ u(x')>c|x'|^{1+\eps}, \quad \mbox{ for  any }|x'|>R \mbox{ for  some } R>0\]
  and\[   E:=\{ (x',x_n) \; \big| \; x_n\geq u(x')\},\]
  then
\[ \alpha(E)=0.\]
On the other hand, if we consider a function that is not rotation invariant, things can go differently, as we see in the next example. 
  
   \begin{example}[The supergraph of $x^3$]
 We consider the supergraph 
  \[ E:=\{ (x,y) \; \big| \; y\geq x^3\}.\] 
In this case, we show that \[\alpha(E) = \pi .\]
For this, given $R>0$, we intersect  $\partial B_R$ with $E$  and denote by $S_1(R)$ and $S_2(R)$ the arcs on the circle as the first picture in Figure \ref{x3}. We consider
the cones\[ C^1_R :=\{ t\sigma  \; \big| \; t\geq 0, \;\sigma\in  S_1(R)\}\, \quad  C^2_R :=\{ t\sigma  \; \big| \; t\geq 0, \;\sigma\in  S_2(R)\} \] and
notice that outside of $B_R$, it holds that $C_R^2\subset E\subset C_R^1$. 
Let  $\overline x_R $ be the solution of \[ x^6+x^2=R^2,\] that is the $x$-coordinate in absolute value of the intersection points $\partial B_R \cap \partial E$. Since $f(x)=x^6+x^2$ is increasing on $(0,\infty)$ and $ R^2=f(\overline x_R)<f(R^{1/3}),$ we have that  $\overline x_R< R^{1/3}$. Hence  
\[ \mathfrak o^1(R)= \pi + \arcsin \frac{\overline x_R}R  \leq \pi+  \arcsin \frac{R^{1/3}}R  ,\quad   \mathfrak o^2(R)  \geq \pi- \arcsin \frac{R^{1/3}}R .\] 
Thanks to the monotonicity property in Proposition \ref{subsetssmin} and to \eqref{cony} we have that
  \[  \overline \alpha (E) \leq\alpha(C_R^1)=  \mathfrak o^1(R), \quad \underline \alpha(E) \geq \alpha(C_R^2)= \mathfrak o^2(R)  \]
  and sending  $R\to \infty$ we obtain that
\[ \overline \alpha(E) \leq  \pi , \quad \underline \alpha(E) \geq \pi.\] Thus $\alpha(E)$ exists and we obtain the desired conclusion.
 \end{example}

 \begin{example}[The supergraph of a bounded function]\label{tanh}
 We consider the supergraph 
  \[ E:=\{ (x',x_n) \; \big| \; x_n\geq u(x') \}, \quad \quad{\mbox{with}}
\quad \quad \|u\|_{L^\infty(\Rn)} <M.\] 
We show that, in this case,
\[\alpha(E) = \frac{\omega_n}2 .\]
To this aim, let \bgs{ &\mathfrak{P}_1:=\{ (x',x_n)\; \big| \; x_n>M\} \\
	&\mathfrak{P}_2:=\{ (x',x_n)\; \big| \; x_n<-M\}.}
	We have that 
 \[   \mathfrak{P}_1 \subset E , \quad \quad \mathfrak{P}_2 \subset \Co E . \] 
  Hence by Proposition \ref{subsetssmin}
  \[ \underline \alpha(E) \geq \overline \alpha (\mathfrak{P}_1)=\frac{\omega_n}2,\quad  \quad  
  \underline \alpha(\Co E)\geq \overline\alpha (\mathfrak{P}_2)=\frac{\omega_n}2 .\]
  Since $\underline \alpha(CE) =\omega_n - \overline \alpha(E)$ we find that
   \[ \overline \alpha(E)\leq \frac{\omega_n}2,\]
  thus the conclusion. An example of this type is depicted in Figure \ref{x3} (more generally, the result holds for the supergraph in $\Rn$  
 $ \{ (x',x_n) \; \big| \; x_n\geq \tanh x_1\}$).
 \end{example}

 \begin{example}[The supergraph of a sublinear graph]\label{candygr}
   More generally, we can take the supergraph of a function that grows sublinearly at infinity, i.e.  
  \[  E:=\{(x',x_n)\;\big|\; x_n>u(x')\}, \qquad{\mbox{with}}\qquad 
\lim_{|x'|\to +\infty} \frac{|u(x')|}{|x'|}=0 .\]
In this case, we show that \[ \alpha(E)= \frac{\omega_n}2.\]
Indeed, for any $\eps>0$ we have that there exists $R=R(\eps)>0$ such that
\[ |u(x')|<\eps |x'|, \quad \forall \;|x'|>R.\]
We denote
\[ S_1(R):= \partial B_R \cap \{(x',x_n)\;\big|\; x_n>\eps|x'|\}, \qquad S_2(R):= \partial B_R \cap \{(x',x_n)\;\big|\; x_n<-\eps|x'|\}\]
and
\[ C_R^i=\{ t\sigma\; \big| \; t\geq 0, \; \sigma \in S_i(R)\}, \quad \mbox{for }\; i=1,2.\]
We have that outside of $B_R$
\[ C_R^1\subset E, \qquad C_R^2\subset \Co E ,\] and 
\[ \alpha (C_R^1)= \alpha (C_R^2)= \frac{ \omega_n}{\pi} \left( \frac{\pi}2- \arctan \eps \right).\]
We use Proposition \ref{subsetssmin}, (i), and letting $\eps$ go to zero, we obtain that $\alpha(E)$ exists and 
\[ \alpha (E)= \frac{\omega_n}2.\]
 A particular example of this type is given by
 \[ E:=\{(x',x_n)\;\big|\; x_n>c|x'|^{1-\eps} \}, \quad \mbox{ when } |x'|>R\; \;  \mbox{ for some } \eps\in(0,1],\,c\in\R,\, R>0     .\]
 \end{example}
 In particular using the additivity property in Proposition \ref{subsetssmin} we can compute $\alpha$ for sets that lie between two graphs. 

  \begin{center}
\begin{figure}[htpb]
	\hspace{0.89cm}
	\begin{minipage}[b]{0.89\linewidth}
	\centering
	\includegraphics[width=0.89\textwidth]{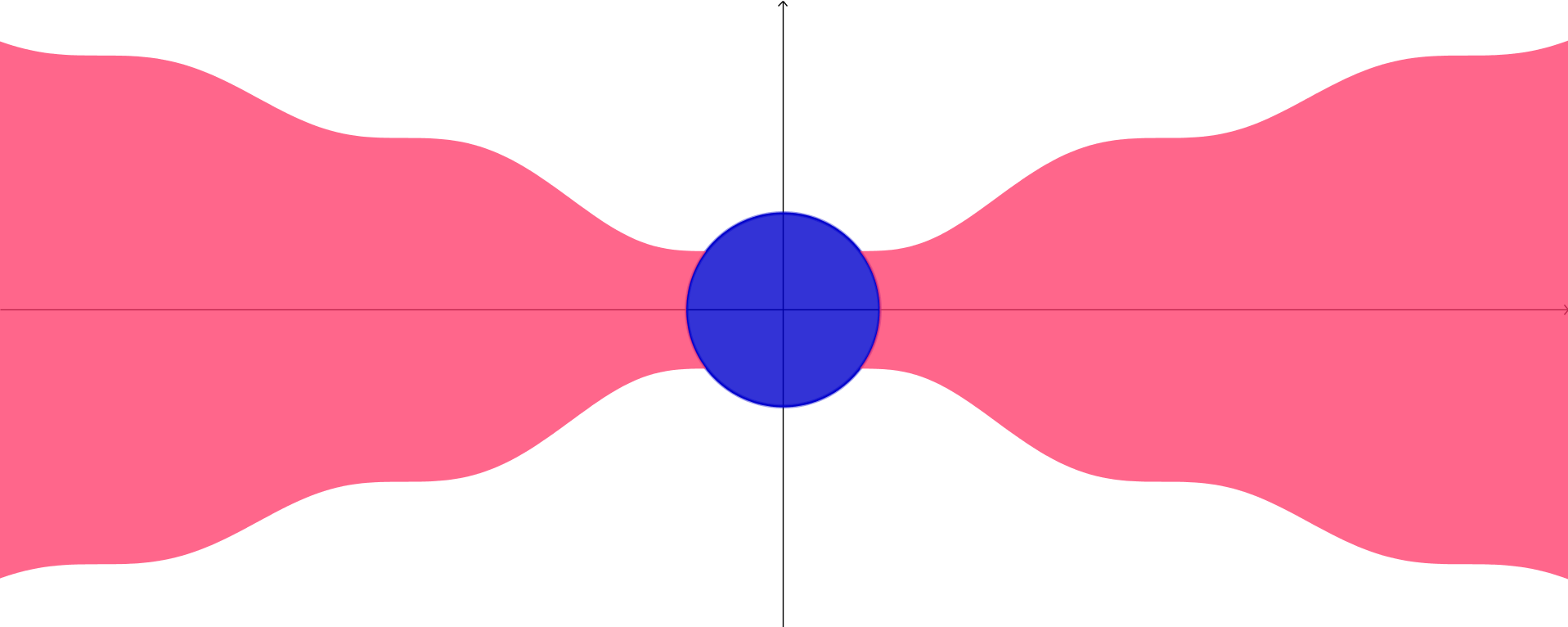}
	\caption{The ``butterscotch hard candy'' graph}   
	\label{candy_pic}
	\end{minipage}
\end{figure} 
	\end{center}
	 \begin{example}[The ``butterscotch hard candy'']
Let~$E\subset\R^n$ be such that
\[ E\cap\{|x'|>R\}\subset
\{(x',x_n)\;\big|\; |x'|>R\;,\;\,|x_n|<c|x'|^{1-\eps} \},   \;\; 
\; \;  \mbox{ for some } \eps\in(0,1],\,c>0,\,  R>0     ,\]
(an example of such a set $E$ is given in Figure \ref{candy_pic}).
In this case, we have that \[\alpha(E)=0.\]
Indeed, we can write $E_1:=E\cap\{|x'|>R\}$ and $E_2:=
E\cap\{|x'|\le R\}$. Then,
using
 the computations in Example \ref{candygr}, we have by
the monotonicity and the additivity properties in Proposition \ref{subsetssmin} that
\[ \overline{\alpha}(E_1) 
\le \alpha\big(\{x_n>-c|x'|^{1-\eps}\}\big)-\alpha\big(\{x_n>c|x'|^{1-\eps}\}\big)=0.\]
Moreover, $E_2$ lies inside $\{|x_1|\le R\}$. Hence, again by Proposition \ref{subsetssmin} and by Example \ref{THECONE},
we find
\[\overline{\alpha}(E_2)\le\alpha\big(\{|x_1|\le R\}\big)=\alpha\big(\{x_1\le R\}\big)-\alpha\big(\{x_1<-R\}\big)=0.\]
Consequently, using again the additivity property in Proposition \ref{subsetssmin}, we obtain that
$$ \overline\alpha(E)\le \overline\alpha(E_1)+\overline\alpha(E_2)
=0,$$
that is the desired result.
\end{example}
We can also compute $\alpha$ for sets that have different growth ratios in different directions. For this, we have the following example.
   \begin{example}[The supergraph of a superlinear function on a small cone]
  We consider a set lying in the half-space, deprived of a set that grows linearly at infinity.  We denote by $\tilde S$ the portion of the sphere given by 
 \bgs{ \tilde{S}:=\Big\{\sigma \in \mathbb S^{n-2}\,\Big|\,\sigma=(&\cos\sigma_1, \sin \sigma_1\cos\sigma_2,\dots, \sin \sigma_1\dots\sin \sigma_{n-2}),\\ &  \mbox{ with } \sigma_i\in \lr{\frac{\pi}2-\bar \eps, \frac{\pi}2+\bar \eps},\;  i=1,\dots,n-2\Big\} ,} where $\overline \eps \in (0,{\pi}/{2})$.
 For $x_0\in \Rn$ and $k>0$ we define  the supergraph $E\subset \Rn$ as
\bgs{\label{sigma}  E:=\big\{ (x',x_n)\in \Rn \; &\big| \; x_n\geq u(x')\big \} \quad   \mbox{ where }  \quad
 u(x')=\alig{ & \, k|x'-x_0'| &\mbox{ for } &x'\in X, \\
		 & \, 0 &\mbox{ for } &x'\notin X,  }\\
	  &X= \{ x'   \in \R^{n-1}  \mbox{ s.t. }x' = t\sigma +x_0',\,  \sigma \in  \tilde S\}.   }
We remark that $X\subset \{x_n=0\}$ is the cone ``generated'' by $\tilde S$ 
and centered at $x_0$.	 Then
\eqlab{\label{alphasigma} \alpha(E) = \frac{\omega_n}2 -  \mathcal{H}^{n-2}(\tilde S)  \int_{0}^k \frac{dt}{(1+t^2)^{\frac{n}2}}.}
Let 
\[ \mathfrak{P}_+
:=\{ (x',x_n) \; \big| \; x_n>0\}, \quad \quad \mathfrak{P}_-:=\{ (x',x_n) \; \big| \; x_n<0\}\] and we consider
the subgraph
\[ F:= \big\{ (x',x_n)\; \big| \; 0<x_n< u(x')\big \}.\]
Then \[ E \cup F=\mathfrak P_+, \quad \quad \mathfrak P_-\cup F = \Co E.\] 
Using the additivity property in Proposition \ref{subsetssmin}, we see that
\eqlab{\label{ah123}  \overline \alpha(E)\geq \frac{\omega_n}2-\overline \alpha(F), \quad \quad  \omega_n - \underline \alpha(E) = \overline \alpha(\Co E)\leq \frac{\omega_n}2 +\overline \alpha(F).} 
Let $R>0$ be arbitrary. We get that
\[\alpha_s(x_0,R,F) \leq  \int_{\left(B'_R(x'_0)\times \R\right)\cap \Co B_R(x_0)}\frac{\chi_{F}(y)}{|y-x_0|^{n+s}}\, dy + \int_{\Co \left( B'_R(x'_0)\times \R\right)}\frac{\chi_{F}(y)}{|y-x_0|^{n+s}}\, dy\]
so
\eqlab{\label{i123}  \alpha_s(x_0,R,F) \leq &\;  \int_{B'_R(x'_0)}\frac{dy'}{|y'-x_0'|^{n-1+s}}\int_{\frac{\sqrt{R^2-|y'-x_0'|^2}}{|y'-x_0'|}}^{\infty} \frac{dt}{(1+t^2)^{\frac{n+s}2}}
   \\ &\; + 
   \int_{\Co B'_R(x'_0)\cap X}\frac{dy'}{|y'-x_0'|^{n-1+s}}\int_{0}^{k} \frac{dt}{(1+t^2)^{\frac{n+s}2}}\\=&\; I_1+I_2.} 
   Using that $1+t^2\geq \max\{1,t^2\}$ and passing to polar coordinates, we obtain that
   \bgs{ I_1= &\; \int_{B'_R(x_0')}\frac{dy'}{|y'-x_0'|^{n-1+s}}\bigg(\int_{\frac{\sqrt{R^2-|y'-x_0'|^2}}{|y'-x_0'|}}^{\frac{R}{|y'-x_0'|}} \frac{dt}{(1+t^2)^{\frac{n+s}2}} + \int_{\frac{R}{|y'-x_0'|}}^{\infty}\frac{dt}{(1+t^2)^{\frac{n+s}2}}\bigg)\\
   \leq&\; \omega_{n-1} \bigg( \int_0^R \tau^{-s-2}   \left(R- \sqrt{R^2-\rho^2}\right)\, d \rho  +  \frac{R^{-n-s+1}}{n+s-1} \int_0^R \rho^{n-2}\, d \rho    \bigg)\\
   =&\;   \omega_{n-1} \bigg( R^{-s} \int_0^1 \tau^{-s-2}   \left(1- \sqrt{1-\tau^2}\right)\, d \tau  +  \frac{R^{-s}}{(n+s-1)(n-1)}  \bigg)  .}
   Also, for any $\tau \in (0,1)$ we have that
\[1 -  \sqrt{1-\tau^2 }\leq c \tau^2,\] for some positive constant $c$, independent on $n,s$.
Therefore 
\bgs{ I_1\leq  \frac{ c \omega_{n-1} R^{-s} }{1-s} + \frac{ \omega_{n-1} R^{-s} }{(n-1)(n+s-1)} .}
Moreover, 
\[ I_2 =   \mathcal{H}^{n-2}(\tilde S) \frac{R^{-s}}s \int_{0}^k \frac{dt}{(1+t^2)^{\frac{n+s}2}}.\]
So
passing to limsup and liminf as $s\to 0^+$ in \eqref{i123} and using Fatou's lemma we obtain that
\[\overline \alpha(F)\leq   \mathcal{H}^{n-2}(\tilde S)\int_{0}^k \frac{dt}{(1+t^2)^{\frac{n}2}},\quad \quad \underline \alpha(F) \geq  \mathcal{H}^{n-2}(\tilde S)\int_{0}^k \frac{dt}{(1+t^2)^{\frac{n}2}}.\] In particular $\alpha(F)$ exists, and from \eqref{ah123} we get that 
\[ \frac{\omega_n}2 -\alpha(F) \leq \underline \alpha(E)\leq \overline \alpha(E)\leq\frac{\omega_n}2 -\alpha(F).\] 
Therefore, $\alpha(E)$ exists and 
\[  \alpha(E) =  \frac{\omega_n}2 -  \mathcal{H}^{n-2}(\tilde S)  \int_{0}^k \frac{dt}{(1+t^2)^{\frac{n}2}} .\] 
 \end{example}

\section{Continuity of the fractional mean curvature and a sign changing
property \\of the nonlocal mean curvature}\label{cont}

We use a formula proved in \cite{regularity} to show that the $s$-fractional mean curvature is continuous
with respect to $C^{1,\alpha}$ convergence of sets, for any $s<\alpha$ and with respect to $C^2$ convergence of sets, for $s$ close to 1.

By $C^{1,\alpha}$ convergence of sets we mean that our sets locally converge in measure and can locally be described as the supergraphs
of functions which converge in $C^{1,\alpha}$. 

\begin{defn}\label{convofsets}
Let $E\subset\R^n$ and let $q\in\partial E$ such that $\partial E$ is $C^{1,\alpha}$ near $q$, for some $\alpha\in(0,1]$. We say that the sequence $E_k\subset\R^n$ converges to $E$ in a $C^{1,\alpha}$ sense (and write $E_k\xrightarrow{C^{1,\alpha}}E$) in a neighborhood of $q$ if:\\
(i)  the sets $E_k$ locally converge in measure to $E$, i.e.
\[|(E_k\Delta E)\cap B_r|\xrightarrow{k\to\infty}0 \quad \mbox{ for any } r>0\]
and \\
(ii) the boundaries $\partial E_k$ converge to $\partial E$ in $C^{1,\alpha}$ sense in a neighborhood of $q$.\\
We define in a similar way the $C^2$ convergence of sets.
\end{defn}

More precisely, 
we denote  
\[Q_{r,h}(x):=B'_r(x')\times(x_n-h,x_n+h),\]
for $x\in\R^n$, $r,h>0$. If $x=0$, we drop it in formulas and simply write $Q_{r,h}:=Q_{r,h}(0)$. Notice that up to a translation and a rotation, we can suppose that $q=0$ and
\eqlab{\label{opossum1}E\cap Q_{2r,2h}=\{(x',x_n)\in\R^n\,|\,x'\in B'_{2r},\,u(x')<x_n<2h\},}
for some $r,h>0$ small enough and $u\in C^{1,\alpha}(\overline B'_{2r})$ such that $u(0)=0$. 
Then, point $(ii)$ means that we can write
\eqlab{\label{opossum}E_k\cap Q_{2r,2h}=\{(x',x_n)\in\R^n\,|\,x'\in B'_{2r},\,u_k(x')<x_n<2h\},}
for some functions $u_k\in C^{1,\alpha}(\overline B_{2r}')$ such that
\eqlab{\label{norm_graph_conv}\lim_{k\to\infty}\|u_k-u\|_{C^{1,\alpha}(\overline B'_{2r})}=0.}
\smallskip
We remark that, by the continuity of $u$,
up to considering a smaller $r$, we can suppose that
\eqlab{\label{bded_graph_hp}|u(x')|<\frac{h}{2},\qquad\forall\,x'\in B'_{2r}.}

We have the following result.

\begin{theorem}\label{everything_converges}
Let $E_k\xrightarrow{C^{1,\alpha}}E$ in a neighborhood of $q\in \partial E$. Let $q_k\in\partial E_k$ be such that $
q_k\longrightarrow q$ and let $s,s_k\in(0,\alpha)$ be such that $s_k\xrightarrow{k\to\infty} s$.
Then
\[\lim_{k\to\infty}\I_{s_k}[E_k](q_k)=\I_s[E](q).\]

Let $E_k\xrightarrow{C^2}E$ in a neighborhood of $q\in \partial E$. Let $q_k\in\partial E_k$ be such that $q_k\longrightarrow q$ and let $s_k\in(0,1)$ be such that $s_k\xrightarrow{k\to\infty} 1$. Then
\[\lim_{k\to\infty}(1-s_k)\I_{s_k}[E_k](q_k)=\omega_{n-1}H[E](q).\]
%
\end{theorem}
\medskip 

A similar problem is studied also in \cite{mattheorem}, where the author estimates the
difference between the fractional mean curvature of a set $E$ with $C^{1,\alpha}$
boundary and that of the set
$\Phi(E)$, where $\Phi$ is a $C^{1,\alpha}$ diffeomorphism of $\R^n$,
%
in terms of the $C^{0,\alpha}$ norm of the Jacobian of the diffeomorphism $\Phi$.


\smallskip

When $s\to 0^+$ we do not need the $C^{1,\alpha}$ convergence of sets, but only the uniform boundedness  of the $C^{1,\alpha}$ norms of the functions defining the boundary of $E_k$ in a neighborhood of the boundary points. However, we have to require that the measure of the symmetric difference is uniformly bounded. More precisely:

\begin{prop}\label{propsto0}
Let $ E\subset \Rn$
be such that $\alpha(E)$ exists.  Let 
$q \in \partial E$ be such that 
\bgs{ E\cap Q_{r,h}(q)=\{(x',x_n)\in\R^n\,|\, x'\in B'_{r}(q'),\,u(x')<x_n<h+q_n\},}
for some $r,h>0$ small enough and $u\in C^{1,\alpha}(\overline B'_{r}(q'))$ such that $u(q')=q_n$. 
Let $E_k\subset \Rn$ be such that
\[ |E_k\Delta E|<C_1  \] 
for some $C_1>0$. Let $q_k\in \partial E_k \cap B_d$, for some $d>0$, such that 
 \bgs{E_k\cap Q_{r,h}(q_k)=\{(x',x_n)\in\R^n\,|\,x'\in B'_{r}(q_k'),\,u_k(x')<x_n<h+q_{k,n}\} }
for some functions $u_k\in C^{1,\alpha}(\overline B_{r}'(q_k'))$ such that $u_k(q_k')=q_{k,n}$ and
\[ \|u_k\|_{C^{1,\alpha}(\overline B'_{r}(q'_k))} <C_2 \] 
for some $C_2>0$. Let $s_k\in(0,\alpha)$ be such that $s_k\xrightarrow{k\to\infty} 0$. Then
\bgs{\lim_{k\to\infty}s_k\I_{s_k}[E_k](q_k)=\omega_n-2\alpha(E).}
\end{prop}

In particular, fixing $E_k=E$ in Theorem \ref{everything_converges} and Proposition \ref{propsto0} we obtain Proposition \ref{rsdfyish} stated in the Introduction. 

To prove
Theorem \ref{everything_converges}
we prove at first the following preliminary result. 

\begin{lemma}\label{supergraph_hp_for_proof}
Let  $E_k\xrightarrow{C^{1,\alpha}}E$ in a neighborhood of $0\in \partial E$. Let $q_k\in\partial E_k$ be such that $
q_k\longrightarrow 0$.  Then \[ E_k-q_k\xrightarrow{C^{1,\beta}}E  \quad \mbox{ in a neighborhood of $0$},\] 
for every $\beta \in(0,\alpha)$.\\
Moreover, if $E_k\xrightarrow{C^{2}}E$ in a neighborhood of $0\in \partial E$, $q_k\in\partial E_k$ are such that $
q_k\longrightarrow 0$ and $\mathcal{R}_k\in  SO(n)$ are such that \[\lim_{k\to \infty} |\mathcal R_k -\mbox{Id}|=0,\] then
\[ \mathcal R_k (E_k-q_k) \xrightarrow{C^{2}}E \quad \mbox{ in a neighborhood of $0$ }.\]
%
%
%
\end{lemma}

\begin{proof}
First of all, notice that since $q_k\longrightarrow0$, for $k$ big enough we have
\[|q'_k|<\frac{1}{2}r\qquad\textrm{and}\qquad|q_{k,n}|=|u_k(q'_k)|<\frac{1}{8}h.\]
 By \eqref{bded_graph_hp} and \eqref{norm_graph_conv}, we
see that for $k$ big enough
\[|u_k(x')|\leq\frac{3}{4}h,\qquad\forall\,x'\in B_{2r}'.\]
Therefore
\[|u_k(x')-q_{k,n}|<\frac{7}{8}h<h,\qquad\forall\,x'\in B_{2r}'.\]
If we define
\[\tilde{u}_k(x'):=u_k(x'+q'_k),\qquad x'\in \overline{B}'_r,\]
for every $k$ big enough we have 
\eqlab{\label{graphs_for_the_proof_eq}
(E_k-q_k) \cap Q_{r,h} =\{(x',x_n)\in\R^n\,|\,x'\in B'_r,\, \tilde u_k(x')<x_n<h\}.}
It is easy to check that the sequence $E_k-q_k$ locally converges in measure to $E$. We claim that
\eqlab{\label{conv_transl_graph}
\lim_{k\to\infty}\|\tilde{u}_k-u\|_{C^{1,\beta}(\overline{B}'_r)}=0.}
Indeed, let 
\[ \tau_k u(x'):= u(x'+q_k').\] 
We have that
\[ \|\tilde u_k-\tau_k u\|_{C^{1}(\overline B_r')}  \leq \|u_k-u\|_{C^{1}\big(\overline B'_{\frac{3}2r}\big)}  \]
and that
\[ \|\tau_k u-u\|_{C^{1}(\overline B_r')} \leq  \|\nabla u\|_{C^0\big(\overline B'_{\frac{3}2r}\big)}  |q_k'| + \|u\|_{C^{1,\alpha}\big( \overline B'_{\frac{3r}2}\big)}|q'_k|^{\alpha} .\] Thus by the triangular inequality 
\[  \lim_{k \to \infty} \|\tilde u_k -u\|_{C^{1}(\overline B_r')} =0,\]
thanks to \eqref{norm_graph_conv} and the fact that $q_k\to 0$.

Now, notice that $\nabla (\tilde u_k) =\tau_k (\nabla u_k)$, so 
\[ [\nabla \tilde u_k -\nabla u]_{C^{0,\beta}(\overline B'_r)} \leq  [\tau_k (\nabla  u_k -\nabla u)]_{C^{0,\beta}(\overline B'_r)}+ [\tau_k(\nabla  u) -\nabla u)]_{C^{0,\beta}(\overline B'_r)}.\] 
Therefore
\[[\tau_k (\nabla  u_k -\nabla u)]_{C^{0,\beta}(\overline B'_r)} \leq  [\nabla  u_k -\nabla u]_{C^{0,\beta}\big(\overline B'_{\frac{3r}2}\big)}\]
and for every $\delta >0$ we obtain
\[  [\tau_k( \nabla  u) -\nabla u]_{C^{0,\beta}(\overline B'_r)} \leq \frac{2}{\delta^\beta} \|\tau_k (\nabla u) -\nabla u\|_{C^{0}\big(\overline B'_{\frac{3r}2}\big)} + 2[\nabla u]_{C^{0,\alpha}(\overline B'_r)} \delta^{\alpha-\beta}.\]
Sending $k\to \infty$ we find that
\[ \limsup_{k\to \infty}  [\tau_k (\nabla  u) -\nabla u)]_{C^{0,\beta}(\overline B'_r)}  \leq 2[\nabla u]_{C^{0,\alpha}(\overline B'_r)} \delta^{\alpha-\beta}\]
for every $\delta>0$, hence
\[ \lim_{k \to \infty}  [\nabla \tilde u_k -\nabla u]_{C^{0,\beta}(\overline B'_r)} =0.\]
This concludes the proof of the first part of the Lemma.\\
As for the second part, the $C^2$ convergence of sets in a neighborhood of $0$ can be proved similarly. Some care must be taken when considering rotations, since one needs to use the implicit function theorem.
\end{proof}

\begin{proof}[Proof of Theorem \ref{everything_converges}]

Up to a translation and a rotation, we can suppose that $q=0$ and $\nu_E(0)=0$. Then we can find $r,h>0$ small enough
and $u\in C^{1,\alpha}(\overline B'_r)$ such that we can write $E\cap Q_{2r,2h}$ as in \eqref{opossum1}.

Since $s_k\to s\in(0,\alpha)$ for $k$ large enough we can suppose that $s_k,s \in[\sigma_0,\sigma_1]$ for $0<\sigma_0<\sigma_1<\beta<\alpha$.
Notice that there exists $\delta>0$ such that
\eqlab{\label{continuity_eq2}
B_\delta\subset\subset Q_{r,h}.}
We take an arbitrary $R>1$ as large as we want and define the sets
\[F_k:= (E_k\cap B_R) -q_k.\]
From Lemma \ref{supergraph_hp_for_proof} we have that in a neighborhood of $0$
\[ F_k\xrightarrow{C^{1,\beta}} E\cap B_R.\] 
In other words,
\eqlab{\label{convvvvv1}\lim_{k\to \infty} |F_k \Delta (E\cap B_R)| =0.}
Moreover, if $u_k$ is a function defining $E_k$ as a supergraph in a neighborhood of $0$ as in \eqref{opossum},
denoting $\tilde u_k(x')=u_k(x'+q_k')$ we have that
\[F_k\cap Q_{r,h}=\{(x',x_n)\in\Rn\,|\,x'\in B'_r,\,\tilde{u}_k(x')<x_n<h\}\]
and that
\eqlab{\label{convvvvv} \lim_{k\to \infty} \|\tilde u_k -u\|_{C^{1,\beta}(\overline B'_r)} =0, \quad  \quad  \|\tilde u_k\|_{C^{1,\beta}(\overline B'_r)}\leq M \; \mbox{ for some } \; M>0.}
We also remark that, by \eqref{bded_graph_hp} we can write
\[E\cap Q_{r,h}=\{(x',x_n)\in\R^n\,|\,x'\in B'_r,\,u(x')<x_n<h\}.\]

Exploiting \eqref{graphs_for_the_proof_eq} we can write the fractional mean curvature of $F_k$ in $0$
by using formula \eqref{complete_curv_formula}, that is
\begin{equation}\label{5rs6ydbfd}
\begin{split}
\I_{s_k}[F_k](0)&=2\int_{B'_r}\Big\{G_{s_k}\Big(\frac{\tilde{u}_k(y')-\tilde{u}_k(0)}{|y'|}\Big)
-G_{s_k}\Big(\nabla \tilde{u}_k(0)\cdot\frac{y'}{|y'|}\Big)\Big\}\frac{dy'}{|y'|^{n-1+s_k}}\\
&
\qquad\qquad+\int_{\Rn}\frac{\chi_{\Co F_k}(y)-\chi_{F_k}(y)}{|y|^{n+s_k}}\chi_{\Co Q_{r,h}}(y)\,dy.\end{split}\end{equation}
Now, we denote as in \eqref{mathcalg}
\[\mathcal G(s_k,\tilde{u}_k,y'):=\mathcal G(s_k,\tilde{u}_k,0,y')= G_{s_k}\Big(\frac{\tilde{u}_k(y')-\tilde{u}_k(0)}{|y'|}\Big)
-G_{s_k}\Big(\nabla \tilde{u}_k(0)\cdot\frac{y'}{|y'|}\Big)\] and we rewrite the
identity in \eqref{5rs6ydbfd} as
\bgs{\I_{s_k}[F_k](0)&=2\int_{B'_r}\mathcal G(s_k,\tilde{u}_k,y') \frac{dy'}{|y'|^{n-1+s_k}}+\int_{\R^n}\frac{\chi_{\Co F_k}(y)-\chi_{F_k}(y)}{|y|^{n+s_k}}\chi_{\Co Q_{r,h}}(y)\,dy.} 
Also, with this notation and by formula \eqref{complete_curv_formula} we have for $E$
\[\I_s[E\cap B_R](0)=2\int_{B'_r}\mathcal G(s,u,y')\frac{dy'}{|y'|^{n-1+s}}
+\int_{\R^n}\frac{\chi_{\Co (E\cap B_R)}(y)-\chi_{E\cap B_R}(y)}{|y|^{n+s}}\chi_{\Co Q_{r,h}}(y)\,dy.\]
%
We can  suppose that $r<1$. We begin by showing that for every $y'\in B_r'\setminus\{0\}$ we have
\eqlab{\label{pwise_conv}\lim_{k\to\infty}\mathcal G(s_k,\tilde{u}_k,y')=\mathcal G(s,u,y').}
First of all, we observe that
\[|\mathcal G(s_k,\tilde{u}_k,y')-\mathcal G(s,u,y')|
\leq|\mathcal G(s_k,\tilde{u}_k,y')-\mathcal G(s,\tilde{u}_k,y')|+|\mathcal G(s,\tilde{u}_k,y')-\mathcal G(s,u,y')|.\]
Then
\bgs{|\mathcal G(s_k,\tilde{u}_k,y')-\mathcal G(s,\tilde{u}_k,y')|&
=\Big|\int_{\nabla \tilde{u}_k(0)\cdot\frac{y'}{|y'|}}^{\frac{\tilde{u}_k(y')-\tilde{u}_k(0)}{|y'|}}(g_{s_k}(t)-g_s(t))\,dt\Big|\\
&
\leq2\int_0^{+\infty}|g_{s_k}(t)-g_s(t)|\,dt.}
Notice that for every $t\in\R$
\[\lim_{k\to\infty}|g_{s_k}(t)-g_s(t)|=0,\qquad\textrm{and}\qquad |g_{s_k}(t)-g_s(t)|\leq2 g_{\sigma_0}(t),\quad\forall\,k\in\mathbb N.\]
Since $g_{\sigma_0}\in L^1(\R)$, by the Dominated Convergence Theorem we obtain that
\[\lim_{k\to\infty}|\mathcal G(s_k,\tilde{u}_k,y')-\mathcal G(s,\tilde{u}_k,y')|=0.\]
We estimate
\bgs{|\mathcal G&(s,\tilde{u}_k,y')-\mathcal G(s,u,y')|\leq\Big|G_s\Big(\frac{\tilde{u}_k(y')-\tilde{u}_k(0)}{|y'|}\Big)
-G_s\Big(\frac{u(y')-u(0)}{|y'|}\Big)\Big|\\
&
\qquad\qquad\qquad\qquad
+\Big|G_s\Big(\nabla\tilde{u}_k(0)\cdot\frac{y'}{|y'|}\Big)-G_s\Big(\nabla u(0)\cdot\frac{y'}{|y'|}\Big)\Big|\\
&
\leq\Big|\frac{\tilde{u}_k(y')-\tilde{u}_k(0)}{|y'|}-\frac{u(y')-u(0)}{|y'|}\Big|
+|\nabla\tilde{u}_k(0)-\nabla u(0)|\\
&
=\Big|\nabla(\tilde{u}_k-u)(\xi)\cdot\frac{y'}{|y'|}\Big|+|\nabla\tilde{u}_k(0)-\nabla u(0)|\\
&
\leq2\|\nabla\tilde{u}_k-\nabla u\|_{C^0(\overline B'_r)},
}
which, by \eqref{conv_transl_graph}, tends to 0 as $k\to\infty$. This proves the pointwise convergence claimed in \eqref{pwise_conv}.\\
Therefore, for every $y'\in B'_r\setminus\{0\}$,
\[\lim_{k\to\infty}\frac{\mathcal G(s_k,\tilde{u}_k,y')}{|y'|^{n-1+s_k}}=
\frac{\mathcal G(s,u,y')}{|y'|^{n-1+s}}.\]
Thus, by \eqref{Holder_useful} we obtain that
\[\Big|\frac{\mathcal G(s_k,\tilde{u}_k,y')}{|y'|^{n-1+s_k}}\Big|
\leq\|\tilde{u}_k\|_{C^{1,\beta}(\overline B'_r)}\frac{1}{|y'|^{n-1-(\beta-s_k)}}
\leq \frac{M}{|y'|^{n-1-(\beta-\sigma_1)}}\in L^1_{loc}(\R^{n-1}),
\]
given \eqref{convvvvv}. 
The Dominated Convergence Theorem then implies that
\eqlab{\label{first_piece_conv}
\lim_{k\to\infty}\int_{B'_r}\mathcal G(s_k,\tilde{u}_k,y')\frac{dy'}{|y'|^{n-1+s_k}}=
\int_{B'_r}\mathcal G(s,u,y')\frac{dy'}{|y'|^{n-1+s}}.
}

Now, we show that 
\eqlab{\label{second_piece} \lim_{k \to \infty}\int_{\Rn}\frac{\chi_{\Co F_k}(y)-\chi_{F_k}(y)}{|y|^{n+s_k}} \chi_{\Co Q_{r,h}}(y)\, dy = \int_{\Rn}\frac{\chi_{\Co (E\cap B_R)}(y)-\chi_{E\cap B_R}(y)}{|y|^{n+s}}\chi_{\Co Q_{r,h}}(y) \, dy.}  
For this, we observe that
\bgs{\Big|\int_{\Co Q_{r,h}}&(\chi_{\Co (E\cap B_R)}(y)-\chi_{E\cap B_R}(y))\Big(\frac{1}{|y|^{n+s_k}}-\frac{1}{|y|^{n+s}}\Big)
dy\Big|
\leq\int_{\Co B_\delta}\Big|\frac{1}{|y|^{n+s_k}}-\frac{1}{|y|^{n+s}}\Big|dy,
}where we have used
\eqref{continuity_eq2}  in the last inequality. 
For $y\in \Co B_1$ 
\bgs{ \Big|\frac{1}{|y|^{n+s_k}}-\frac{1}{|y|^{n+s}}\Big| \leq \frac{2}{|y|^{n+\sigma_0} }\in L^1(\Co B_1) }
and  for  $y\in  B_1\setminus B_\delta$ 
\bgs{ \Big|\frac{1}{|y|^{n+s_k}}-\frac{1}{|y|^{n+s}}\Big| \leq \frac{2}{|y|^{n+\sigma_1}} \in L^1(B_1\setminus B_\delta).} 
We use then the Dominated Convergence Theorem and get that
\bgs{ \lim_{k \to \infty} \int_{\Co Q_{r,h}}(\chi_{\Co (E\cap B_R)}(y)-\chi_{E\cap B_R}(y))\Big(\frac{1}{|y|^{n+s_k}}-\frac{1}{|y|^{n+s}}\Big)
dy =0.} 
Now
\bgs{ \bigg|\int_{\Co Q_{r,h}} & \;\frac{ \chi_{\Co F_k}(y)-\chi_F{_k}(y) -\left(\chi_{\Co (E\cap B_R)(y) }- \chi_{E\cap B_R} (y)\right) }{|y|^{n+s_k}}\, dy\bigg| =2\int_{\Co Q_{r,h}} \frac{\chi_{F_k \Delta (E\cap B_R)} (y)}{|y|^{n+s_k}}\, dy \\ \leq& \;2 \frac{ |F_k \Delta (E\cap B_R)|}{\delta^{n+\sigma_1}} \xrightarrow{k\to \infty} 0,}
according to \eqref{convvvvv1}. The last two limits prove \eqref{second_piece}. Recalling \eqref{first_piece_conv}, we obtain that
\[ \lim_{k\to \infty} \I_{s_k}[F_k](0) = \I_s[E\cap B_R](0).\] 
We have that $\I_{s_k} [F_k](0)= \I_{s_k} [E_k\cap B_R](q_k)$, so
\bgs{|\I_{s_k}[E_k](q_k)-\I_s[E](0)|\leq &\;|\I_{s_k}[E_k](q_k) - \I_{s_k}[E_k\cap B_R](q_k)|+ |\I_{s_k}[F_k](0)- \I_s[E\cap B_R](0)| \\ &\;+ |\I_s[E\cap B_R](0)- \I_s[E](0)|.}
Since
\bgs{\label{ekbounded} | \I_{s_k}[E_k](q_k) - \I_{s_k} [E_k\cap B_R] (q_k) | +| \I_{s}[E](0) - \I_{s} [E\cap B_R] (0) | \leq  \frac{4 \omega_n }{\sigma_0} R^{-\sigma_0},    }
sending $R\to \infty$
\[ \lim_{k \to \infty} \I_{s_k}[E_k](q_k)=\I_s[E](0).\]
This concludes the proof of the first part of the Theorem.

\bigskip

In order to prove the second part of Theorem \ref{everything_converges}, we fix $R>1$ and we denote
\[F_k:=\mathcal R_k\big((E_k\cap B_R)-q_k\big),\]
where $\mathcal R_k\in SO(n)$ is a rotation such that
\[\mathcal R_k:\nu_{E_k}(0)\longmapsto\nu_E(0)=-e_n\quad\mbox{ and }\quad
\lim_{k\to\infty}|\mathcal R_k-\mbox{Id}|=0.\]
Thus, by Lemma \ref{supergraph_hp_for_proof} we know that $F_k\xrightarrow{C^2}E$ in a neighborhood of $0$.\\
To be more precise,
\eqlab{\label{convvvvv2}\lim_{k\to \infty} |F_k \Delta (E\cap B_R)| =0.}
Moreover, there exist $r,h>0$ small enough and $v_k,u\in C^2(\overline B'_r)$ such that
\bgs{&F_k\cap Q_{r,h}=\{(x',x_n)\in\Rn\,|\,x'\in B'_r,\,v_k(x')<x_n<h\},\\
&
E\cap Q_{r,h}=\{(x',x_n)\in\Rn\,|\,x'\in B'_r,\,u(x')<x_n<h\}}
and that
\eqlab{\label{convvvvv3} \lim_{k\to \infty} \|v_k -u\|_{C^2(\overline B'_r)} =0.}
Notice that $0\in\partial F_k$ and $\nu_{F_k}(0)=e_n$ for every $k$, that is,
\eqlab{\label{opossum4}v_k(0)=u(0)=0,\quad\nabla v_k(0)=\nabla u(0)=0.}

We claim that
\eqlab{\label{opossum3}\lim_{k\to\infty}(1-s_k)\big|\I_{s_k}[F_k](0)-\I_{s_k}[E\cap B_R](0)\big|=0.}
By \eqref{opossum4} 
%
and formula \eqref{complete_curv_formula} we have that
\bgs{\label{opossum5}
\I_{s_k}[F_k](0)&=2\int_{B'_r}\frac{dy'}{|y'|^{n+s_k-1}} \int_0^{\frac{v_k(y')}{|y'|} }\frac{dt}{(1+t^2)^{\frac{n+s_k}2}}
+\int_{\Co Q_{r,h}}\frac{\chi_{\Co F_k}(y)-\chi_{F_k}(y)}{|y|^{n+s_k}}\,dy\\
&= \I^{loc}_{s_k}[F_k](0) +\int_{\Co Q_{r,h}}\frac{\chi_{\Co F_k}(y)-\chi_{F_k}(y)}{|y|^{n+s_k}}\,dy.}
We use the same formula for $E\cap B_R$ and prove at first that
\bgs{\bigg|\int_{\Co Q_{r,h}}\frac{\chi_{\Co F_k}(y)-\chi_{F_k}(y)-\chi_{\Co(E\cap B_R)}(y)+\chi_{E\cap B_R}(y)}{|y|^{n+s_k}}\,dy\bigg|\le\frac{|F_k\Delta(E\cap B_R)|}{\delta^{n+s_k}}\le\frac{|F_k\Delta(E\cap B_R)|}{\delta^{n+1}},
}
(where we have used \eqref{continuity_eq2}), which tends to 0 as $k\to\infty$, by \eqref{convvvvv2}.

Moreover, notice that by the Mean Value Theorem and \eqref{opossum4} we have
\[|(v_k-u)(y')|\le\frac{1}{2}|D^2(v_k-u)(\xi')||y'|^2\le\frac{\|v_k-u\|_{C^2(\overline B'_r)}}{2}|y'|^2.\]
Thus
\bgs{ &\big| \I^{loc}_{s_k}[F_k](0)  - \I^{loc}_{s_k}[E\cap B_R](0) |
\le2\int_{B'_r} \frac{dy'}{|y'|^{n+s_k-1}} \bigg|\int_{\frac{u(y')}{|y'|}}^{\frac{v_k(y')}{|y'|}}\frac{dt}{(1+t^2)^\frac{n+s_k}{2}}\bigg|\\
&
\le2\int_{B'_r}|y'|^{-n-s_k}|(v_k-u)(y')|\,d y'
\le\frac{\omega_{n-1}\,\|v_k-u\|_{C^2(\overline B'_r)}}{1-s_k}r^{1-s_k},
}
hence by \eqref{convvvvv3} we obtain
\eqlab{\label{opossum7}
\lim_{k\to\infty}(1-s_k)\big|\I^{loc}_{s_k}[F_k](0)  - \I^{loc}_{s_k}[E\cap B_R](0)|=0.
}
This concludes the proof of claim \eqref{opossum3}.

Now we use the triangle inequality and have that
\bgs{
\big|(1-s_k)\I_{s_k}&[E_k](q_k)-H[E](0)\big|\le(1-s_k)\big|\I_{s_k}[E_k](q_k)-\I_{s_k}[F_k](0)\big|\\
&
+(1-s_k)\big|\I_{s_k}[F_k](0)-\I_{s_k}[E\cap B_R](0)\big|
+\big|(1-s_k)\I_{s_k}[E\cap B_R](0)-H[E](0)\big|.
}
The last term in the right hand side converges by Theorem 12 in \cite{Abaty}. As for the first term,
notice that
\[\I_{s_k}[F_k](0)=\I_{s_k}[E_k\cap B_R](q_k),\]
hence
\[\lim_{k\to\infty}(1-s_k)\big|\I_{s_k}[E_k\cap B_R](q_k)-\I_{s_k}[E_k](q_k)\big|\le\limsup_{k\to\infty}(1-s_k)\frac{2\omega_n}{s_k}R^{-s_k}=0.\]
Sending $k\to\infty$ in the triangle inequality above, we conclude the proof of the second part of Theorem \ref{everything_converges}.
\end{proof}

\begin{remark}
In relation to the second part of the proof, we point out that using the directional fractional mean curvature defined in \cite[ Definition 6, Theorem 8]{Abaty}, we can write
\bgs{ \I^{loc}_{s_k}[F_k](0)= &\;2\int_{\mathbb S^{n-2}}\bigg[\int_0^r\rho^{n-2}\bigg(\int_0^{v_k(\rho e)}\frac{dt}{(\rho^2+t^2)^\frac{n+s_k}{2}}\bigg)d\rho\bigg]d\mathcal H^{n-2}_e \\
=&\;2\int_{\mathbb S^{n-2}} \overline K_{s_k,e} d\mathcal H^{n-2}_e.}
One is then actually able to prove that
\bgs{\lim_{k\to\infty}(1-s_k)\overline K_{s_k,e}[E_k-q_k](0)=H_e[E](0),}
uniformly in $e\in\mathbb S^{n-2}$, by using formula \eqref{opossum7} and the first claim of Theorem 12 in \cite{Abaty}.
\end{remark}

\textcolor{black}{\begin{remark}\label{nuno}
The proof of Theorem \ref{everything_converges}, as well as the proof of the next Proposition \ref{propsto0}, settles the case in which $n\geq 2$. For $n=1$, the proof follows in the same way, after observing that the local contribution to the mean curvature is equal to zero because of symmetry. As a matter of fact, the formula in~\eqref{complete_curv_formula} for the mean curvature (which has no meaning for $n=1$) is not required.\\
We remark also that in our notation $\omega_0=0$. This gives consistency to the second claim of Theorem \ref{everything_converges} also for $n=1$.
\end{remark}}

\bigskip
We prove now the continuity of the fractional mean curvature as $s\to 0$.

\begin{proof}[Proof of Proposition \ref{propsto0}]
Up to a translation, we can take $q=0$ and $u(0)=0$. \\
For $R>2\max\{r,h\}$, we write
\bgs{  \I_{s_k}[E_k](q_k) =&\;
P.V.\int_{ Q_{r,h}(q_k)} \frac{\chi_{\Co E_k}(y) -\chi_{E_k}(y)}{|y-q_k|^{n+s_k}}\, dy
+\int_{\Co Q_{r,h}(q_k)} \frac{\chi_{\Co E_k}(y)-\chi_{E_k}(y)}{|y-q_k|^{n+s_k}}\, dy 
\\
=&\; P.V.\int_{ Q_{r,h}(q_k)} \frac{\chi_{\Co E_k}(y)-\chi_{E_k}(y)}{|y-q_k|^{n+s_k}}\, dy  + \int_{ B_R(q_k)\setminus Q_{r,h}(q_k)}\frac{\chi_{\Co E_k}(y)-\chi_{E_k}(y)}{|y-q_k|^{n+s_k}}\, dy
\\
&\; +\int_{\Co B_R(q_k)} \frac{\chi_{\Co E_k}(y)-\chi_{E_k}(y)}{|y-q_k|^{n+s_k}}\, dy  
 \\
=&\;I_1(k)+I_2(k)+I_3(k).}
Now  using  \eqref{complete_curv_formula}, \eqref{mathcalg} and \eqref{Holder_useful}  we have that 
\bgs{|I_1(k)|\leq &\;  2 \int_{B'_r(q_k')} \frac{| \mathcal G(s_k,u_k,q_k', y')|}{|y'-q_k'|^{n+s_k-1}} \,dy'
 \leq 2  \| u_k \|_{C^{1,\alpha}(\overline{B}'_r(q_k'))}\int_{B'_r(q_k')} \frac{  |y'-q_k'|^{\alpha}}{|y'-q_k'|^{n+s_k-1}} \,dy'
 \\
 \leq  &\; 2 C_2 \omega_{n-1} \frac{r^{\alpha-s_k}}{\alpha-s_k}.
}
Using \eqref{continuity_eq2} we also have that
\[|I_2(k)| \leq  \int_{ B_R(q_k)\setminus B_\delta(q_k)}\frac{dy}{|y-q_k|^{n+s_k}} 
= \omega_n\frac{\delta^{-s_k}-R^{-s_k}}{s_k}. \]
Thus
\eqlab{\label{kangaroo}
\lim_{k\to\infty}s_k\big(|I_1(k)|+|I_2(k)|\big)=0.}
Furthermore
\bgs{ \big | s_k I_3(k)- &\big(\omega_n-2 s_k \alpha_{s_k}(0,R,E) \big)\big|
\\
\leq &\; \bigg|s_k \int_{\Co B_R(q_k)} \frac{dy}{|y-q_k|^{n+s_k}}  -2 s_k \int_{\Co B_R(q_k)} \frac{\chi_{E_k}(y)}{|y-q_k|^{n+s_k}}\, dy   - \omega_n+2s_k \alpha_{s_k}(q_k,R,E))\bigg|
\\
&\;+2s_k |\alpha_{s_k}(q_k,R,E)- \alpha_{s_k}(0,R,E)|
\\
\leq& \; |\omega_n R^{-s_k}-\omega_n| + 2 s_k\bigg| \int_{\Co B_R(q_k)}\frac{\chi_{E_k}(y)}{|y-q_k|^{n+s_k}}\, dy  - \int_{\Co B_R(q_k)} \frac{\chi_{E}(y)}{|y-q_k|^{n+s_k}}\, dy \bigg| 
\\
&\;+ 2s_k |\alpha_{s_k}(q_k,R,E)- \alpha_{s_k}(0,R,E)|
\\
\leq &\; |\omega_n R^{-s_k}-\omega_n| + 2 s_k \int_{\Co B_R(q_k)}\frac{\chi_{E_k\Delta E}(y)}{|y-q_k|^{n+s_k}}\, dy  
+ 2s_k| \alpha_{s_k}(q_k,R,E)- \alpha_{s_k}(0,R,E)|
\\
\leq &\; |\omega_n R^{-s_k}-\omega_n| + 2 C_1 s_k R^{-n-s_k} +2s_k | \alpha_{s_k}(q_k,R,E)- \alpha_{s_k}(0,R,E)|,
 }
 where we have used that $|E_k\Delta E|<C_1$.
 
 Therefore, since $q_k\in B_d$ for every $k$, as a consequence of Proposition \ref{unifrq} it follows that
\eqlab{\label{kangaroo1}\lim_{k\to \infty}\big | s_k I_3(k)- &\big(\omega_n-2 s_k \alpha_{s_k}(0,R,E) \big)\big|=0.}

Hence, by \eqref{kangaroo} and \eqref{kangaroo1}, we get that
\bgs{ \lim_{k \to \infty} s_k \I_{s_k}[E_k](q_k)=\omega_n-2\lim_{k\to\infty}s_k\alpha_{s_k}(0,R,E)=\omega_n-2\alpha(E),}
concluding the proof.
\end{proof} 

\begin{proof}[Proof of Theorem \ref{asympts}]
Arguing as in the proof of Proposition \ref{propsto0}, by keeping fixed $E_k=E$ and $q_k=p$, we obtain
\bgs{ \liminf_{s\to0} s\, \I_s[E](p)=\omega_n-2\limsup_{s\to0}s\,\alpha_s(0,R,E)=\omega_n-2\overline{\alpha}(E),}
and similarly for the limsup.
\end{proof}

As a corollary of Theorem \ref{everything_converges} and Theorem \ref{asympts}, we have the following result.
\begin{theorem}\label{changeyoursign}
Let $E\subset\R^n$ and let $p\in\partial E$ be such that $\partial E\cap B_r(p)$ is $C^2$ for some $r>0$.
Suppose that the classical mean curvature of $E$ in $p$ is $H(p)<0$. Also assume that 
\[\overline \alpha(E) < \frac{\omega_n}2.\] Then there exist $\sigma_0<\tilde{s}<\sigma_1$ in $(0,1)$ such that

$(i)\quad\I_s[E](p)>0$ for every $s\in(0,\sigma_0]$, and actually
\[\liminf_{s\to0^+}s \;\I_s[E](p)=\omega_n- 2\overline \alpha(E),\]

$(ii)\quad\I_{\tilde{s}}[E](p)=0,$

$(iii)\quad\I_s[E](p)<0$ for every $s\in[\sigma_1,1)$,
and actually
\[ \lim_{s\to 1} (1-s)\;\I_s[E](p)= \omega_{n-1}H[E](p).\]

\end{theorem}

 \appendix
 \section{Some geometric observations}\label{appendicite}

\subsection{Sliding the balls}

For the convenience of the reader, we collect here some auxiliary and elementary
results of geometric nature, that are used in the proofs of the main results.

\begin{lemma}\label{slidetheballs}
Let $F\subset\Rn$ be such that\footnote{Concerning the statement of Lemma~\ref{slidetheballs},
we recall that the notation~$\overline{F}$ denotes the closure of the set~$F$,
when~$F$ is modified, up to sets of measure zero,
in such a way that~$F$ is assumed to contain its measure theoretic interior~$F_{int}$
and to have empty intersection with the exterior~$F_{ext}$, 
according to the setting described in Section~\ref{MEAS:ASS:SEC}.
For instance, if~$F$ is a segment in~$\R^2$, this convention implies that~$F_{int}=\emptyset$, $F_{ext}=\R^2$ and so~$F$ and~$\overline{F}$ in this case also reduce to the empty set.}
\[B_\delta(p)\subset F_{ext}\quad\textrm{for some }\delta>0\qquad\textrm{and}\qquad q\in\overline{F},\]
and let $c:[0,1]\longrightarrow\Rn$ be a continuous curve connecting $p$ to $q$, that is
\[c(0)=p\qquad\textrm{and}\qquad c(1)=q.\]
Then there exists $t_0\in[0,1)$ such that $B_\delta\big(c(t_0)\big)$ is an exterior tangent ball to $F$,
that is
\eqlab{\label{slide_ext_tg}
B_\delta\big(c(t_0)\big)\subset F_{ext}\qquad\textrm{and}\qquad\partial B_\delta\big(c(t_0)\big)\cap\partial F\not=\emptyset.}
\end{lemma}

\begin{proof}
Define
\eqlab{\label{slide_t}
t_0:=\sup\Big\{\tau\in[0,1]\,\big|\,\bigcup_{t\in[0,\tau]}B_\delta\big(c(t)\big)\subset F_{ext}\Big\}.
}
We begin by proving that
\eqlab{\label{jardine}
B_\delta\big(c(t_0)\big)\subset F_{ext}.
}
If $t_0=0$, this is trivially true by hypothesis.
Thus, suppose that $t_0>0$ and assume by contradiction that
\bgs{
B_\delta\big(c(t_0)\big)\cap \overline{F}\not=\emptyset.
}
Then there exists a point
\[
y\in\overline{F}=F_{int}\cup\partial F\quad\mbox{s.t.}\quad d:=|y-c(t_0)|<\delta.
\]
By exploiting the continuity of $c$, we can find $t\in[0,t_0)$
such that
\[
|y-c(t)|\le|y-c(t_0)|+|c(t_0)-c(t)|\le d+\frac{\delta-d}{2}<\delta,
\]
and hence $y\in B_\delta\big(c(t)\big)$.
However, this is in contradiction with the fact that, by definition of $t_0$, we have $B_\delta\big(c(t)\big)\subset F_{ext}$.
This concludes the proof of \eqref{jardine}.

We point out that, since $q\in\overline F$, by \eqref{jardine} we have that $t_0<1$.

Now we prove that $t_0$ as defined in \eqref{slide_t} satisfies \eqref{slide_ext_tg}.

Notice that by \eqref{jardine} we have
\eqlab{\label{slide_pf}
\overline{B_\delta\big(c(t_0)\big)}\subset \overline{F_{ext}}=F_{ext}\cup\partial F.
}
Suppose that
\[\partial B_\delta\big(c(t_0)\big)\cap\partial F=\emptyset.\]
Then \eqref{slide_pf} implies that
\[\overline{B_\delta\big(c(t_0)\big)}\subset F_{ext},\]
and, since $F_{ext}$ is an open set, we can find $\tilde\delta>\delta$ such that
\[B_{\tilde\delta}\big(c(t_0)\big)\subset F_{ext}.\]
By continuity of $c$ we can find $\eps\in(0,1-t_0)$ small enough such that
\[|c(t)-c(t_0)|<\tilde\delta-\delta,\qquad\forall\,t\in[t_0,t_0+\eps].\]
Therefore
\[B_\delta\big(c(t)\big)\subset B_{\tilde\delta}\big(c(t_0)\big)\subset F_{ext},\qquad\forall\,t\in[t_0,t_0+\eps],\]
and hence
\[\bigcup_{t\in[0,t_0+\eps]}B_\delta\big(c(t)\big)\subset F_{ext},\]
which is in contradiction with
the definition of $t_0$. Thus
\[\partial B_\delta\big(c(t_0)\big)\cap\partial F\not=\emptyset,\]
which concludes the proof.
\end{proof}

\subsection{Smooth domains}\label{A2}

Given a set $F\subset\R^n$, the signed distance function $\bar{d}_F$ from $\partial F$, negative inside $F$, is defined as
\begin{equation*}
\bar{d}_F(x)=d(x,F)-d(x,\Co F)\qquad\mbox{for every }x\in\R^n,
\end{equation*}
where
\[d(x,A):=\inf_{y\in A}|x-y|,\]
denotes the usual distance from a set $A$. Given an open set $\Omega\subset\R^n$, we denote by
\begin{equation*}
N_\rho(\partial\Omega):=\{|\bar{d}_\Omega|<\rho\}=\{x\in\R^n\,|\,d(x,\partial\Omega)<\rho\}
\end{equation*}
the tubular $\rho$-neighborhood of $\partial\Omega$.
For the details about the properties of the
signed distance function, we refer to \cite{GilTru,Ambrosio} and the references cited therein.

Now we recall the notion of (uniform) interior ball condition.
\begin{defn}
We say that an open set $\mathcal O$ satisfies an interior ball condition at $x\in\partial\mathcal O$ if
there exists a ball $B_r(y)$ s.t.
\begin{equation*}
B_r(y)\subset\mathcal O\qquad\textrm{and}\qquad x\in\partial B_r(y).
\end{equation*}
We say that the condition is ``strict'' if $x$ is the only tangency point, i.e.
\[\partial B_r(y)\cap\partial\mathcal O=\{x\}.\]
The open set $\mathcal O$ satisfies a uniform (strict) interior ball condition of radius $r$ if
it satisfies the (strict) interior ball condition at every point of $\partial\mathcal O$,
with an interior tangent ball of radius at least $r$.\\
In a similar way one defines exterior ball conditions.
\end{defn}
We remark that
if $\mathcal O$ satisfies an interior ball condition of radius $r$ at $x\in\partial\mathcal O$,
then the condition is strict for every radius $r'<r$.

\begin{remark}\label{ext_unif_omega}
Let $\Omega\subset\R^n$ be a bounded open set with $C^2$ boundary. It is well
known that $\Omega$ satisfies a uniform interior and exterior ball condition. We fix $r_0=r_0(\Omega)>0$
such that $\Omega$ satisfies a strict interior and a strict exterior ball contition of radius $2r_0$
at every point $x\in\partial\Omega$.
Then
\begin{equation}\label{r01}
\bar{d}_\Omega\in C^2(N_{2r_0}(\partial\Omega)),
\end{equation}
(see e.g. Lemma 14.16 in \cite{GilTru}).
\end{remark}

We remark that the distance function $d(-,E)$ is differentiable at $x\in\R^n\setminus\overline E$ if
and only if there is a unique point $y\in\partial E$ of minimum distance, i.e.
\[d(x,E)=|x-y|.\]
In this case, the two points $x$ and $y$ are related by the formula
\[y=x-d(x,E)\nabla d(x,E).\]

This generalizes to the signed distance function. In particular, if $\Omega$ is bounded and has $C^2$ boundary, then we can
define a $C^1$ projection function from the tubular $2r_0$-neighborhood $N_{2r_0}(\partial\Omega)$ onto $\partial\Omega$ by
assigning to a point $x$ its unique nearest point $\pi(x)$, that is
\[\pi:N_{2r_0}(\partial\Omega)\longrightarrow\partial\Omega,\qquad\pi(x):=x-\bar{d}_\Omega(x)\nabla\bar{d}_\Omega(x).\]
We also remark that
on $\partial\Omega$ we have that $\nabla\bar{d}_\Omega=\nu_\Omega$ and that
\[\nabla\bar{d}_\Omega(x)=\nabla\bar{d}_\Omega(\pi(x))=\nu_\Omega(\pi(x)),\qquad\forall x\in N_{2r_0}(\partial\Omega).\]
Thus $\nabla\bar{d}_\Omega$ is a vector field which extends
the outer unit normal to a tubular neighborhood of $\partial\Omega$, in a $C^1$ way.

Notice that given a point $y\in\partial\Omega$, for every $|\delta|<2r_0$ the point $x:=y+\delta\nu_\Omega(y)$ is such that $\bar{d}_\Omega(x)=\delta$ (and $y$ is its unique nearest point).
Indeed, we consider for example $\delta\in(0,2r_0)$. Then we can find an exterior tangent ball
\[B_{2r_0}(z)\subset\Co\Omega,\qquad\partial B_{2r_0}(z)\cap\partial\Omega=\{y\}.\]
Notice that the center of the ball must be
\[z=y+2r_0\nu_\Omega(y).\]
Then, for every $\delta\in(0,2r_0)$ we have
\[B_\delta(y+\delta\nu_\Omega(y))\subset B_{2r_0}(y+2r_0\nu_\Omega(y))\subset
\Co\Omega,\qquad\partial B_\delta(y+\delta\nu_\Omega(y))\cap\partial\Omega=\{y\}.\]
This proves that
\[|\bar{d}_\Omega(y+\delta\nu_\Omega(y))|=d(x,\partial\Omega)=\delta.\]
Finally, since the point $x$ lies outside $\Omega$, its signed distance function is positive.

\begin{remark}\label{c21}
Since $|\nabla\bar{d}_\Omega|=1$, the bounded open sets
\begin{equation*}
\Omega_\delta:=\{\bar{d}_\Omega<\delta\}
\end{equation*}
have $C^2$ boundary
\begin{equation*}
\partial\Omega_\delta=\{\bar{d}_\Omega=\delta\},
\end{equation*}
for every $\delta\in(-2r_0,2r_0)$.
\end{remark}

As a consequence, we know that for every $|\delta|<2r_0$ the set $\Omega_\delta$
satisfies a uniform interior and exterior ball condition of radius $r(\delta)>0$.
Moreover, we have that $r(\delta)\geq r_0$ for every $|\delta|\leq r_0$
(see also Appendix A in \cite{MR3436398}
for related results).

\begin{lemma}\label{geomlem}
Let $\Omega\subset\R^n$ be a bounded open set with $C^2$ boundary.
Then for every $\delta\in[-r_0,r_0]$ the set $\Omega_\delta$ 
satisfies a uniform interior and exterior ball condition of radius at least $r_0$, i.e.
\begin{equation*}
r(\delta)\geq r_0\qquad\textrm{for every }|\delta|\leq r_0.
\end{equation*}
\end{lemma}
\begin{proof}
Take for example $\delta\in[-r_0,0)$ and let $x\in\partial\Omega_\delta=\{\bar{d}_\Omega=\delta\}$.
We show that $\Omega_\delta$ has an interior tangent ball of radius $r_0$ at $x$. The other cases are proven in a similar way.

Consider the projection $\pi(x)\in\partial\Omega$ and the point
\[x_0:=x-r_0\nabla\bar{d}_\Omega(x)=\pi(x)-(r_0+|\delta|)\nu_\Omega(\pi(x)).\]
Then
\[B_{r_0}(x_0)\subset\Omega_\delta\quad\textrm{ and }\quad x\in\partial B_{r_0}(x_0)\cap\partial\Omega_\delta.\]
Indeed, notice that, as remarked above,
\[d(x_0,\partial\Omega)=|x_0-\pi(x)|=r_0+|\delta|.\]
Thus, by the triangle inequality we have that
\[d(z,\partial\Omega)\ge d(x_0,\partial\Omega)-|z-x_0|>|\delta|,\qquad\textrm{ for every }z\in B_{r_0}(x_0),\]
so $B_{r_0}\subset\Omega_\delta$. Moreover, by definition of $x_0$ we have
\[x\in\partial B_{r_0}(x_0)\cap\partial\Omega_\delta\]
and the desired result follows.
\end{proof}

To conclude, we remark that the sets $\overline{\Omega_{-\delta}}$ are retracts of $\Omega$, for every $\delta\in(0,r_0]$.
Indeed, roughly speaking, each set $\overline{\Omega_{-\delta}}$ is obtained by deforming $\Omega$ in normal direction,
towards the interior.
An important consequence is that if $\Omega$ is connected then $\overline{\Omega_{-\delta}}$ is path connected.

To be more precise, we have the following:

\begin{prop}\label{retract}
Let $\Omega\subset\Rn$ be a bounded open set with $C^2$ boundary.
Let $\delta\in(0,r_0]$ and define
\[\mathcal D:\Omega\longrightarrow\overline{\Omega_{-\delta}},\qquad\mathcal D(x):=
\left\{\begin{split}&x,& &x\in\Omega_{-\delta},\\
&x-\big(\delta+\bar{d}_\Omega(x)\big)\nabla\bar{d}_\Omega(x),& & x\in\Omega\setminus\Omega_{-\delta}.\end{split}\right.\]
Then $\mathcal D$ is a retraction of $\Omega$ onto $\overline{\Omega_{-\delta}}$, i.e. it is continuous and
$\mathcal D(x)=x$ for every $x\in\overline{\Omega_{-\delta}}$.
In particular, if $\Omega$ is connected, then $\overline{\Omega_{-\delta}}$ is path connected.
\end{prop}

\begin{proof}
Notice that the function
\[\Phi(x):=x-\big(\delta+\bar{d}_\Omega(x)\big)\nabla\bar{d}_\Omega(x)\]
is continuous in $\Omega\setminus\Omega_{-\delta}$ and $\Phi(x)=x$ for every $x\in\partial\Omega_{-\delta}$.
Therefore the function $\mathcal D$ is continuous.

We are left to show that
\[\mathcal D(\Omega\setminus\Omega_{-\delta})\subset\partial\Omega_{-\delta}.\]
For this, it is enough to notice that
\[\mathcal D(x)=\pi(x)-\delta\nu_\Omega(\pi(x))\qquad\textrm{for every }x\in\Omega\setminus\Omega_{-\delta}.\]
To conclude, suppose that $\Omega$ is connected and recall that if an open set $\Omega\subset\R^n$ is connected, then it is also path connected.
Thus $\overline{\Omega_{-\delta}}$,
being the continuous image of a path connected space, is itself
path connected.
\end{proof}   

\section{Collection of other useful results on nonlocal minimal surfaces}\label{appendicite2}

Here, we collect some auxiliary results on nonlocal minimal surfaces.
In particular, we recall the representation of
the fractional mean curvature when the set is a graph and
a useful and general version of the maximum principle.

\subsection{Explicit formulas for the fractional mean curvature of a graph}
We denote 
\[Q_{r,h}(x):=B'_r(x')\times(x_n-h,x_n+h),\] for $x\in\R^n,$ $r,h>0$. If $x=0$, we write $Q_{r,h}:=Q_{r,h}(0)$. Let also 
\[g_s(t):=\frac{1}{(1+t^2)^\frac{n+s}{2}}\qquad\textrm{and}\qquad G_s(t):=\int_0^tg_s(\tau)\,d\tau.\]
Notice that
\[0<g_s(t)\leq1,\quad\forall\,t\in\R\qquad\textrm{and}\qquad\int_{-\infty}^{+\infty}g_s(t)\,dt<\infty,\]
for every $s\in(0,1)$.

In this notation, we can write the fractional mean curvature of a graph as follows:

\begin{prop}
Let $F\subset\R^n$ and $p\in\partial F$ such that
\[F\cap Q _{r,h}(p)=\{(x',x_n)\in\R^n\,|\,x'\in B'_r(p'),\,v(x')<x_n<p_n+h\},\]
for some $v\in C^{1,\alpha}(\overline{B}'_r(p'))$. Then for every $s\in(0,\alpha)$
\eqlab{\label{complete_curv_formula}\I_s[F](p)&
=2\int_{B'_r(p')}\Big\{G_s\Big(\frac{v(y')-v(p')}{|y'-p'|}\Big)
-G_s\Big(\nabla v(p')\cdot\frac{y'-p'}{|y'-p'|}\Big)\Big\}\frac{dy'}{|y'-p'|^{n-1+s}}\\
&
\qquad\qquad+\int_{\R^n\setminus Q_{r,h}(p)}\frac{\chi_{\Co F}(y)-\chi_F(y)}{|y-p|^{n+s}}\,dy.}
\end{prop}

This explicit formula was introduced in \cite{regularity} (see also \cite{Abaty,lukes}) when $\nabla v(p)=0$. In \cite{bootstrap}, the reader can find the formula for the case of non-zero gradient. 

\begin{remark}
In the right hand side of \eqref{complete_curv_formula}
there is no need to consider the principal value, since the integrals are summable.
Indeed,
\bgs{\label{Holder_useful}
\Big|G_s\Big(&\frac{v(y')-v(p')}{|y'-p'|}\Big)
-G_s\Big(\nabla v(p')\cdot\frac{y'-p'}{|y'-p'|}\Big)\Big|
=\Big|\int_{\nabla v(p')\cdot\frac{y'-p'}{|y'-p'|}}^{\frac{v(y')-v(p')}{|y'-p'|}}g_s(t)\,dt\Big|\\
&
\leq\Big|\frac{v(y')-v(p')-\nabla v(p')\cdot(y'-p')}{|y'-p'|}\Big|\leq \|v\|_{C^{1,\alpha}(\overline B'_r(p'))}|y'-p'|^\alpha,
}
for every $y'\in B'_r(p')$.
As for the last inequality, notice that by the Mean value Theorem we have
\[v(y')-v(p')=\nabla v(\xi)\cdot(y'-p'),\]
for some $\xi\in B'_r(p')$ on the segment with end points $y'$ and $p'$. Thus
\bgs{|v(y')-v(p')&-\nabla v(p')\cdot(y'-p')|=|(\nabla v(\xi)-\nabla v(p'))\cdot(y'-p')|\\
&
\leq|\nabla v(\xi)-\nabla v(p')||y'-p'|\leq\|\nabla v\|_{C^{0,\alpha}(\overline B'_r(p'))}|\xi-p'|^\alpha|y'-p'|\\
&
\leq\|v\|_{C^{1,\alpha}(\overline B'_r(p'))}|y'-p'|^{1+\alpha}.
}
We denote for simplicity
\eqlab{ \label{mathcalg} \mathcal G(s,v,y',p'):= G_s\Big(&\frac{v(y')-v(p')}{|y'-p'|}\Big)
-G_s\Big(\nabla v(p')\cdot\frac{y'-p'}{|y'-p'|}\Big).}
With this notation, we have
\eqlab{\label{Holder_useful} |\mathcal G(s,v,y',p')|  \leq \|v\|_{C^{1,\alpha}(\overline B'_r(p'))}|y'-p'|^\alpha.}
\end{remark}


\subsection{Interior regularity theory and
its influence on the Euler-Lagrange equation inside the domain}\label{brr2}

In this Appendix we give a short review of the the Euler-Lagrange equation 
in the interior of the domain. In particular, by exploiting results which give an improvement of
the regularity of $\partial E$, we show that an $s$-minimal set is a classical solution
of the Euler-Lagrange equation almost everywhere.

First of all, we recall the definition of supersolution.
\begin{defn}
Let $\Omega\subset\R^n$ be an open set and let $s\in(0,1)$. A set $E$
is an $s$-supersolution in $\Omega$ if $P_s(E,\Omega)<\infty$ and
\begin{equation}\label{supersolution}
P_s(E,\Omega)\leq P_s(F,\Omega)\quad\textrm{for every set }E\textrm{ s.t. }E\subset F\textrm{ and }F\setminus\Omega=E\setminus\Omega.
\end{equation}
\end{defn}
We remark that \eqref{supersolution} is equivalent to
\begin{equation*}
A\subset\Co E\cap\Omega\qquad\Longrightarrow\qquad \Ll_s(A,E)-\Ll_s(A,\Co(E\cup A))\leq0.
\end{equation*}
In a similar way one defines $s$-subsolutions.

In \cite{nms} it is shown that a set $E$ which is an $s$-supersolution in $\Omega$
is also a viscosity supersolution of the equation $\I_s[E]=0$ on $\partial E\cap\Omega$.
To be more precise

\begin{theorem}[Theorem 5.1 of \cite{nms}]\label{viscsol}
Let $E$ be an $s$-supersolution in the open set $\Omega$. If $x_0\in\partial E\cap \Omega$ and $E$ has an interior tangent ball
at $x_0$, contained in $\Omega$,
i.e.
\begin{equation*}
B_r(y)\subset E\cap\Omega\quad\textrm{s.t.}\quad x_0\in\partial E\cap\partial B_r(y),
\end{equation*}
then
\begin{equation}\label{supersolution_ineq}
\liminf_{\rho\to0^+}\I_s^\rho[E](x_0)\geq0.
\end{equation}

\end{theorem}

In particular, $E$ is a viscosity supersolution in the following sense.

\begin{corollary}
Let $E$ be an $s$-supersolution in the open set $\Omega$ and let $F$ be an open set such that $F\subset E$.
If $x\in(\partial E\cap\partial F)\cap\Omega$ and $\partial F$ is $C^{1,1}$ near $x$,
then $\I_s[F](x)\geq0$.
\end{corollary}
\begin{proof}
Since $\partial F$ is $C^{1,1}$ near $x$, $F$ has an interior tangent ball at $x$. In particular, notice that this ball
is tangent also to $E$ at $x$ (from the inside).
Thus by Theorem \ref{viscsol}
\begin{equation*}
\liminf_{\rho\to0^+}\I_s^\rho[E](x)\geq0.
\end{equation*}
Now notice that
\begin{equation*}
F\subset E\qquad\Longrightarrow\qquad\chi_{\Co F}-\chi_F\geq\chi_{\Co E}-\chi_E,
\end{equation*}
so
\begin{equation*}
\I_s^\delta[F](x)\geq\I_s^\delta[E](x)\qquad\forall\,\delta>0.
\end{equation*}
Since $\I_s[F](x)$ is well defined, it is then enough to pass to the limit $\delta\to0$.
\end{proof}

\begin{remark}
Similarly, for an $s$-subsolution $E$ which has an exterior tangent ball at $x_0$ we obtain
\begin{equation}\label{subsolution_ineq}
\limsup_{\rho\to0^+}\I_s^\rho[E](x_0)\leq0.
\end{equation}
\end{remark}

Now we recall the following two regularity results.
If $E$ is $s$-minimal, having a tangent ball (either interior or exterior) at some point $x_0\in\partial E\cap\Omega$ is enough (via an improvement of flatness result) to have $C^{1,\alpha}$ regularity in a neighborhood of $x_0$ (see Corollary 6.2 of \cite{nms}).
%
Moreover, bootstrapping arguments prove that $C^{0,1}$ regularity
guarantees $C^\infty$ regularity (according to Theorem 1.1 of \cite{Bernstein}).

%

It is also convenient to introduce the notion of locally $s$-minimal set, which is useful when considering an unbounded domain $\Omega$.\\
We say that a set $E\subset\R^n$ is locally $s$-minimal in an open set $\Omega\subset\R^n$ if $E$ is $s$-minimal in every bounded
open set $\Omega'\subset\subset\R^n$.

Exploiting the regularity results that we recalled above, we obtain the following:
\begin{theorem}\label{EL_inside}
Let $\Omega\subset\R^n$ be an open set and let $E$ be locally $s$-minimal in $\Omega$. If $x_0\in\partial E\cap\Omega$
and $E$ has either an interior or exterior tangent ball at $x_0$,
then there exists $r>0$ such that $\partial E\cap B_r(x_0)$
is $C^\infty$ and
\begin{equation}\label{EL_eq_inside}
\I_s[E](x)=0\qquad\textrm{for every}\quad x\in\partial E\cap B_r(x_0).
\end{equation}
\end{theorem}
\begin{proof}
Since $x_0\in\partial E\cap \Omega$ and $\Omega$ is open, we can find $r>0$ such that $B_r(x_0)\subset\subset\Omega$.\\
The set $E$ is then $s$-minimal in $B_r(x_0)$. Moreover, by hypothesis we have a tangent ball (either interior or exterior)
to $E$ at $x_0$. Also notice that we can suppose that the tangent ball is contained in $B_r(x_0)$.\\
Thus, by Corollary 6.2 of \cite{nms} and Theorem 1.1 of \cite{Bernstein}, we know that $\partial E$ is $C^\infty$ in
$B_r(x_0)$ (up to taking another $r>0$ small enough).

In particular, $\I_s[E](x)$ is well defined for every $x\in\partial E\cap B_r(x_0)$ and $E$ has both an interior and an exterior tangent
ball at every $x\in\partial E\cap B_r(x_0)$ (both contained in $B_r(x_0)$).\\
Therefore, since an $s$-minimal set is both an $s$-supersolution and an $s$-subsolution, by
\eqref{supersolution_ineq} and \eqref{subsolution_ineq}, we obtain
\begin{equation*}
0\leq\liminf_{\rho\to0^+}\I_s^\rho[E](x)=\I_s[E](x)=\limsup_{\rho\to0^+}\I_s^\rho[E](x)\leq0,
\end{equation*}
for every $x\in\partial E\cap B_r(x_0)$, proving \eqref{EL_eq_inside}.
\end{proof}

Furthermore, we recall that if $E\subset\R^n$ is $s$-minimal in $\Omega$, then the singular set
$\Sigma(E;\Omega)\subset\partial E\cap\Omega$
has Hausdorff dimension at most $n-3$ (by the dimension reduction argument developed in Section 10 of \cite{nms}
and Corollary 2 of \cite{sing_cones}).

Now suppose that $E$ is locally $s$-minimal in an open set $\Omega$.
We observe that
we can find a sequence of bounded open sets with Lipschitz
boundaries $\Omega_k\subset\subset\Omega$ such that $\bigcup\Omega_k=\Omega$
(see e.g. Corollary 2.6 in \cite{mine_cyl_stuff}). Since $E$ is $s$-minimal in each $\Omega_k$ and
$\Sigma(E;\Omega)=\bigcup\Sigma(E;\Omega_k)$, we get in particular
\begin{equation}\label{singset}
\Ha^{n-2}(\Sigma(E;\Omega))\leq\sum_{k=1}^\infty\Ha^{n-2}(\Sigma(E;\Omega_k))=0
\end{equation}
(and indeed $\Sigma(E;\Omega)$ has Hausdorff dimension at most $n-3$, since we have inequality \eqref{singset}
with $n-d$ in place of $n-2$, for every $d\in[0,3)$).

As a consequence, a (locally) $s$-minimal set is a classical solution of the Euler-Lagrange equation,
in the following sense
\begin{theorem}\label{classicalsense}
Let $\Omega\subset\R^n$ be an open set and let $E$ be locally $s$-minimal in $\Omega$.
Then
\begin{equation*}
\I_s[E](x)=0\qquad\textrm{for every }x\in(\partial E\cap\Omega)\setminus\Sigma(E;\Omega),
\end{equation*}
and hence in particular for $\Ha^{n-1}$-a.e. $x\in\partial E\cap\Omega$.
\end{theorem}


\subsection{Boundary Euler-Lagrange inequalities for the fractional perimeter}\label{appendicite3}
We recall that a set $E$ is locally $s$-minimal in an open set $\Omega$ if it is $s$-minimal in every bounded open set compactly contained in $\Omega$.
In this section we show that the Euler-Lagrange equation of a locally $s$-minimal set $E$ holds (at least as an inequality)
also at a point $p\in\partial E\cap\partial\Omega$,
provided that the boundary $\partial E$ and the boundary $\partial\Omega$ do not intersect
``transversally'' in $p$.

To be more precise, we prove the following

\begin{theorem}\label{EL_boundary_coroll}
Let $s\in(0,1)$. Let $\Omega\subset\R^n$ be an open set and let $E\subset\R^n$
be locally $s$-minimal in $\Omega$. Suppose that $p\in\partial E\cap\partial\Omega$ is such that
$\partial\Omega$ is $C^{1,1}$ in $B_{R_0}(p)$, for some $R_0>0$. Assume also that
\eqlab{\label{obstaclehp}
B_{R_0}(p)\setminus\Omega\subset\Co E.
}
Then
\begin{equation*}
\I_s[E](p)\leq0.
\end{equation*}
Moreover, if there exists $R\in(0,R_0)$ such that
\begin{equation}\label{detach_hp}
\partial E\cap\big(\Omega\cap B_r(p)\big)\not=\emptyset\qquad\textrm{for every }r\in(0,R),
\end{equation}
then $$\I_s[E](p)=0.$$
\end{theorem}

We remark that
by hypothesis the open set $B_{R_0}(p)\setminus\overline{\Omega}$ is tangent to $E$ at $p$, from the outside.
Therefore, either \eqref{detach_hp} holds true, meaning roughly speaking that the boundary of $E$ detaches
from the boundary of $\Omega$ at $p$ (towards the interior of $\Omega$),
or $\partial E$ coincides with $\partial\Omega$ near $p$. See Figure \ref{fig:el}.

\begin{figure}[htbp]
\begin{center}
\includegraphics[width=120mm]{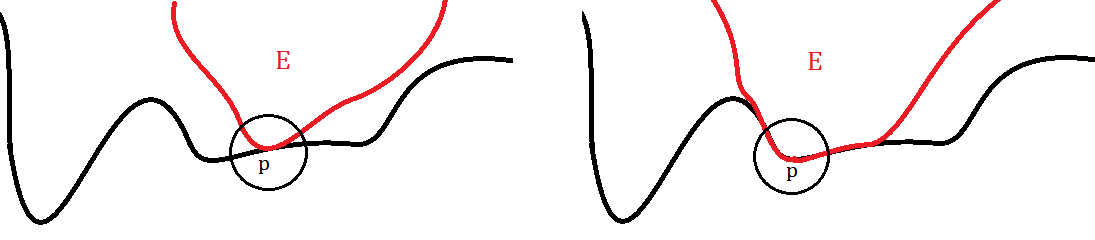}
\caption{\it Examples of a set which satisfies \eqref{detach_hp} (on the left)
and of a set whose boundary sticks to that of $\Omega$ near $p$ (on the right)}
\label{fig:el}
\end{center}
\end{figure}

Roughly speaking, the idea of the proof of Theorem \ref{EL_boundary_coroll} is the following.
The set $\mathcal O:=B_{R_0}(p)\setminus\overline{\Omega}$
plays the role of an obstacle in the minimization of the $s$-perimeter in $B_{R_0}(p)$.
The (local) minimality of $E$ in $\Omega$, together with hypothesis
\eqref{obstaclehp}, implies that $E$ solves this geometric obstacle type problem, which has
been investigated in \cite{GeomObst}. As a consequence, the set $E$ is a viscosity subsolution in $B_{R_0}(p)$ and we obtain that
$\I_s[E](p)\le0$.
Furthermore, the regularity result proved in \cite{GeomObst} guarantees that $\partial E$ is $C^{1,\sigma}$, with $\sigma>s$, near $p$. Thus, if $\partial E$ satisfies \eqref{detach_hp}, then we can exploit the Euler-Lagrange equation inside $\Omega$
and the continuity of $\I_s[E]$ to prove that $\I_s[E](p)=0$.

We now proceed to give a rigorous proof of Theorem \ref{EL_boundary_coroll}.

\begin{proof}[Proof of Theorem \ref{EL_boundary_coroll}]
We begin by observing that we can find a bounded and connected open set $\Omega'\subset\Omega$ such that
\begin{equation*}
\partial\Omega'\textrm{ is }C^{1,1}\qquad\textrm{and}\qquad
\Omega'\cap B_\frac{R_0}{2}(p)=\Omega\cap B_\frac{R_0}{2}(p).
\end{equation*}
Then, since $E$ is locally $s$-minimal in $\Omega$, we know that it is locally $s$-minimal also in $\Omega'$.
Hence, since $\Omega'$ is bounded and has regular boundary, by Theorem 1.7 of \cite{mine_cyl_stuff} we find that
$E$ is actually $s$-minimal in $\Omega'$.
Moreover $p\in\partial E\cap\partial\Omega'$
and
\begin{equation*}
B_\frac{R_0}{2}(p)\setminus\Omega'=B_\frac{R_0}{2}(p)\setminus\Omega
\subset B_{R_0}(p)\setminus\Omega\subset\Co E.
\end{equation*}
Therefore, we can suppose without loss of generality that $\Omega$ is a bounded and connected open set with $C^{1,1}$ boundary
$\partial\Omega$ and that $E$ is $s$-minimal in $\Omega$.

\smallskip

As observed in the proof of Theorem 5.1 of \cite{Graph}, the minimality of $E$ and hypothesis
\eqref{obstaclehp} imply that the set $\Co E$ is a solution, in $B_\frac{R_0}{4}(p)$, of the geometric obstacle type problem considered in \cite{GeomObst}.

More precisely, we remark that we can find a bounded and connected open set $\mathcal O$ with $C^{1,1}$ boundary, such that
\bgs{
\mathcal O\cap B_\frac{R_0}{4}(p)=B_\frac{R_0}{4}(p)\setminus\overline{\Omega}.
}
Then hypothesis \eqref{obstaclehp} guarantees that
\bgs{
\mathcal O\cap B_\frac{R_0}{4}(p)\subset\Co E.
}
Now, by arguing as in the proof of Theorem 5.1 of \cite{Graph},
we find that the minimality of $E$ (hence also of $\Co E$) in $\Omega$
implies that
\bgs{
P_s\Big(\Co E,B_\frac{R_0}{4}(p)\Big)
\le P_s\Big(F,B_\frac{R_0}{4}(p)\Big),
}
for every $F\subset\R^n$ such that
\bgs{
F\setminus B_\frac{R_0}{4}(p)=\Co E\setminus B_\frac{R_0}{4}(p)
\quad\mbox{and}\quad\mathcal O\cap B_\frac{R_0}{4}(p)\subset F.
}
In particular, as observed in \cite{GeomObst} (see the comment (2.2) there), the set $\Co E$ is a viscosity supersolution in $B_\frac{R_0}{4}(p)$, meaning that
the set $E$ is a viscosity subsolution in $B_\frac{R_0}{4}(p)$. Now, since the set $\Omega$ has $C^{1,1}$ boundary,
we can find an exterior tangent ball at $p\in\partial\Omega$.
By hypothesis \eqref{obstaclehp}, this means that we can find an exterior tangent ball at $p\in\partial E$ and hence
we have
\eqlab{\label{obst1}
\limsup_{\rho\to0^+}\I_s^\rho[E](p)\le0.
}
Furthermore, Theorem 1.1 of \cite{GeomObst} guarantees that $\partial E$ is $C^{1,\sigma}$ in $B_{R_0'}(p)$
for some $R'_0\in(0,R_0)$, and $\sigma:=\frac{1+s}{2}$ (see also Theorem 5.1 of \cite{Graph}).
In particular, since $\sigma>s$, we know that the $s$-fractional mean curvature of $E$ is well defined at $p$.
Therefore \eqref{obst1} actually implies that~$\I_s[E](p)\le0$, as claimed.

Now we suppose in addition
that \eqref{detach_hp} holds true, i.e. that
\begin{equation*}
\partial E\cap\big(\Omega\cap B_r(p)\big)\not=\emptyset\qquad\textrm{for every }r\in(0,R),
\end{equation*}
with $R<R'_0$.
By Theorem 1.1 of \cite{Bernstein} we know that $\partial E\cap\big(B_R(p)\cap\Omega\big)$ is $C^\infty$.
In particular, as observed in Theorem \ref{EL_inside}, we know that every point $x\in\partial E\cap\big(B_R(p)\cap\Omega\big)$ satisfies the Euler-Lagrange
equation in the classical sense, i.e.
\begin{equation}\label{EL_proof_eq1}
\I_s[E](x)=0\qquad\textrm{for every }x\in\partial E\cap\big(B_R(p)\cap\Omega\big).
\end{equation}
Since $\partial E\cap B_R(p)$ is $C^{1,\sigma}$, with $\sigma>s$,
we also know that $\I_s[E]\in C(\partial E\cap B_R(p))$ (by, e.g., Proposition \ref{rsdfyish} or Lemma 3.4 of \cite{Graph}).
Finally, we observe that by \eqref{detach_hp} we can find a sequence of points $x_k\in\partial E\cap\big(B_R(p)\cap\Omega\big)$
such that $x_k\longrightarrow p$.
Then, by the continuity of $\I_s[E]$ and \eqref{EL_proof_eq1} we get
\begin{equation*}
\I_s[E](p)=\lim_{k\to\infty}\I_s[E](x_k)=0,
\end{equation*}
concluding the proof.
\end{proof}


    \subsection{A maximum principle}
By exploiting the Euler-Lagrange equation, we can compare an $s$-minimal set with half spaces.
We show that if $E$ is $s$-minimal in $\Omega$ and the exterior data $E_0:=E\setminus\Omega$ lies above a half-space, then
also $E\cap\Omega$ must lie above that same half-space. This is indeed
a very general principle, that we now discuss in full detail.
To this aim, it is convenient to point out that
if $E\subset F$ and the boundaries of the two sets touch at a common point $x_0$ where the $s$-fractional mean curvatures coincide, then the two sets must be equal.
The precise result goes as follows:

\begin{lemma}\label{curv_rigidity}
Let $E,F\subset\R^n$ be such that $E\subset F$ and $x_0\in\partial E\cap\partial F$. Then
\begin{equation}\label{confront_curv_ineq}
\I_s^\rho[E](x_0)\geq\I_s^\rho[F](x_0)\qquad\textrm{for every }\rho>0.
\end{equation}
Furthermore, if
\begin{equation}\label{ineq_for_curvs}
\liminf_{\rho\to0^+}\I_s^\rho[F](x_0)\geq a\quad\textrm{and}\quad\limsup_{\rho\to0^+}\I_s^\rho[E](x_0)\leq a,
\end{equation}
then $E=F$, the fractional mean curvature is well defined in $x_0$ and $\I_s[E](x_0)=a$.

\begin{proof}
To get $(\ref{confront_curv_ineq})$ it is enough to notice that
\[
E\subset F\quad\Longrightarrow\quad\big(\chi_{\Co E}(y)-\chi_E(y)\big)
\geq\big(\chi_{\Co F}(y)-\chi_F(y)\big)\qquad\forall\,y\in\R^n.
\]
Now suppose that $(\ref{ineq_for_curvs})$ holds true. Then by $(\ref{confront_curv_ineq})$ we find that
\[
\exists\,\lim_{\rho\to0^+}\I_s[E](x_0)=\lim_{\rho\to0^+}\I_s[F](x_0)=a.
\]

To conclude, notice that if the two curvatures are well defined (in the principal value sense) in $x_0$ and are equal,
then
\begin{equation*}\begin{split}
0\leq\int_{\Co B_\rho(x_0)}&\frac{\big(\chi_{\Co E}(y)-\chi_E(y)\big)-\big(\chi_{\Co F}(y)-\chi_F(y)\big)}{|x_0-y|^{n+s}}dy\\
&
=\I_s^\rho[E](x_0)-\I_s^\rho[F](x_0)\xrightarrow{\rho\to0^+}0,
\end{split}\end{equation*}
which implies that $\chi_E(y)=\chi_F(y)$ for a.e. $y\in\R^n$, i.e. $E=F$.
\end{proof}\end{lemma}

\begin{prop}\label{maximum_principle}[Maximum Principle]
Let $\Omega\subset\R^n$ be a bounded open set with $C^{1,1}$ boundary. Let $s\in(0,1)$ and let $E$ be 
$s$-minimal in $\Omega$. If
\begin{equation}\label{ext_data_incl}
\{x\cdot\nu\leq a\}\setminus\Omega\subset\Co E,\end{equation}
for some $\nu\in\mathbb S^{n-1}$ and $a\in\R$, then
\[\{x\cdot\nu\leq a\}\subset \Co E.\]

\begin{proof}
First of all, we remark that up to a rotation and translation, we can suppose that $\nu=e_n$ and $a=0$.
Furthermore we can assume that
\[
\inf_{x\in\overline{\Omega}}x_n<0,
\]
otherwise there is nothing to prove.

If $E\cap\Omega=\emptyset$, i.e. $\Omega\subset\Co E$, we are done.
Thus we can suppose that $E\cap\Omega\not =\emptyset$.\\
Since $\overline{E}\cap\overline{\Omega}$ is compact, we 
have
\[
b:=\min_{x\in\overline{E}\cap\overline{\Omega}}x_n\in\R.
\]
Now we consider the set of points which realize the minimum
above, namely we set
$$\mathcal P:=\{p\in\overline{E}\cap\overline{\Omega}\,|\,p_n=b\}.$$
Notice that
\begin{equation}\label{confronto_first_incl}
\big\{x_n\leq\min\{b,0\}\big\}\subset\Co E,
\end{equation}
so we are reduced to prove that $b\geq0$.

We argue by contradiction and suppose that $b<0$. We will prove that $\mathcal P=\emptyset$.
We remark that $\mathcal P\subset\partial E\cap\overline{\Omega}$.

Indeed, if $p\in\mathcal P$, then by $(\ref{confronto_first_incl})$ we have that
$B_\delta(p)\cap\{x_n\leq b\}\subset\Co E$
for every $\delta >0$, so $|B_\delta(p)\cap\Co E|\geq\frac{\omega_n}{2}\delta^n$
and $p\not\in E_{int}$. Therefore, since $\overline{E}=E_{int}\cup\partial E$, we 
find that $p\in\partial E$.

Roughly speaking, we are sliding upwards the half-space $\{x_n\leq t\}$ until we first touch the set $\overline{E}$. Then the contact points must belong to the boundary of $E$.

Notice that the points of $\mathcal P$ can be either inside $\Omega$ or on $\partial\Omega$.
In both cases we can use the Euler-Lagrange equation to get a contradiction. The precise argument goes as follows.

First, if $p=(p',b)\in\partial E\cap\Omega$, then since $H:=\{x_n\leq b\}\subset\Co E$,
we can find an exterior tangent ball to $E$ at $p$ (contained in $\Omega$), so $\I_s[E](p)=0$.

On the other hand, if $p\in\partial E\cap\partial\Omega$, then
$B_{|b|}(p)\setminus\Omega\subset\Co E$ and hence (by Theorem 5.1 of \cite{Graph})
$\partial E\cap B_r(p)$ is $C^{1,\frac{s+1}{2}}$ for some $r\in(0,|b|)$, and $\I_s[E](p)\leq0$ by Theorem \eqref{EL_boundary_coroll} .

In both cases, we have that
\[
p\in\partial H\cap\partial E,\quad H\subset \Co E\quad\textrm{and}\quad\I_s[\Co E](p)=-\I_s[E](p)\geq0=\I_s[H](p),
\]
and hence Lemma \ref{curv_rigidity} implies $\Co E=H$. However, since $b<0$, this contradicts
$(\ref{ext_data_incl})$.

This proves that $b\geq0$, thus concluding the proof.
\end{proof}
\end{prop}

{F}rom this, we obtain a strong comparison principle with planes,
as follows:

\begin{corollary}
Let $\Omega\subset\R^n$ be a bounded open set with $C^{1,1}$ boundary.
Let $E\subset\R^n$ be $s$-minimal in $\Omega$, with
$\{x_n\leq0\}\setminus\Omega\subset\Co E$.
Then

$(i)\quad$ if $|(\Co E\setminus\Omega)\cap\{x_n>0\})|=0$, then $E=\{x_n>0\}$;

$(ii)\quad$ if $|(\Co E\setminus\Omega)\cap\{x_n>0\}|>0$,
then for every $x=(x',0)\in\Omega\cap\{x_n=0\}$ there exists $\delta_x\in(0,d(x,\partial\Omega))$ s.t.
$B_{\delta_x}(x)\subset\Co E$. Thus
\begin{equation}
\{x_n\leq0\}\cup\bigcup_{(x',0)\in\Omega}B_{\delta_x}(x)\subset\Co E.
\end{equation}

\begin{proof}
First of all, Proposition \ref{maximum_principle} guarantees that
\begin{equation*}
\{x_n\leq0\}\subset\Co E.
\end{equation*}

$(i)\quad$ Notice that since $E$ is $s$-minimal in $\Omega$, also $\Co E$ is $s$-minimal in $\Omega$.\\
Thus, since $\{x_n>0\}\setminus\Omega\subset E=\Co(\Co E)$, we can use again Proposition \ref{maximum_principle}
(notice that $\{x_n=0\}$ is a set of measure zero) to
get $\{x_n>0\}\subset E$, proving the claim.

$(ii)\quad$ Let $x\in\{x_n=0\}\cap \Omega$.

We argue by contradiction. Suppose that $|B_\delta(x)\cap E|>0$ for every $\delta>0$.\\
Notice that,
since $B_\delta(x)\cap\{x_n\leq0\}\subset\Co E$ for every $\delta>0$, this implies that $x\in\partial E\cap \Omega$.
Moreover, we can find an exterior tangent ball to $E$ in $x$, namely
\begin{equation*}
B_\eps(x-\eps\,e_n)\subset\{x_n\leq0\}\cap\Omega\subset\Co E\cap\Omega.
\end{equation*}
Thus the Euler-Lagrange equation gives $\I_s[E](x)=0$.

Let $H:=\{x_n\leq0\}$.
Since $x\in\partial H$, $H\subset\Co E$ and also $\I_s[H](x)=0$,
Lemma \ref{curv_rigidity} implies $\Co E=H$.
However this contradicts the hypothesis
\begin{equation*}
|(\Co E\setminus\Omega)\cap\{x_n>0\}|>0,
\end{equation*}
which completes the proof.
\end{proof}\end{corollary}

%
%

\bibliography{biblio}
\bibliographystyle{plain}

\end{document}